\documentclass[12pt,a4paper,twoside]{article}

\usepackage[active]{srcltx}
\usepackage[utf8]{inputenc}
\usepackage{hyperref}
\usepackage[english]{babel}
\usepackage{authblk}

\usepackage{xcolor}
\usepackage{verbatim}

\usepackage[a4paper,hcentering,vcentering]{geometry}
\usepackage{url,epsfig,amssymb,amsthm,amsmath}

\numberwithin{equation}{section}

\numberwithin{theorems}{section}

\numberwithin{corollarys}{section}

\newtheorem{theorem}{Theorem}[section]
\newtheorem{definition}[theorem]{Definition}
\newtheorem{corollary}[theorem]{Corollary}
\newtheorem{lemma}[theorem]{Lemma}
\newtheorem{proposition}[theorem]{Proposition}
\newtheorem{remark}[theorem]{Remark}

\newcommand{\ellX}{\ell^\infty\left(\mathcal{X}\right)}
\newcommand{\ellY}{\ell^\infty\left(\mathcal{Y}\right)}

\definecolor{JungleGreen}{cmyk}{0.99,0,0.52,0}
\definecolor{RawSienna}{cmyk}{0,0.72,1,0.45}
\definecolor{bulgarianrose}{rgb}{0.28, 0.02, 0.03}
\definecolor{Magenta}{cmyk}{0,1,0,0}
\definecolor{airforceblue}{rgb}{0.36, 0.54, 0.66}
\definecolor{darkpastelgreen}{rgb}{0.01, 0.75, 0.24} 
\definecolor{brightgreen}{rgb}{0.4, 1.0, 0.0} 
\definecolor{applegreen}{rgb}{0.55, 0.71, 0.0} 
\definecolor{asparagus}{rgb}{0.53, 0.66, 0.42} 
\definecolor{calpolypomonagreen}{rgb}{0.12, 0.3, 0.17} 
\definecolor{britishracinggreen}{rgb}{0.0, 0.26, 0.15} 

\newcommand{\C}{\ensuremath{\mathbb{C}}}

\newcommand{\N}{\ensuremath{\mathbb{N}}}

\newcommand{\R}{\ensuremath{\mathbb{R}}}

\newcommand{\Z}{\ensuremath{\mathbb{Z}}}

\newenvironment{eq*}
  {\begin{equation*}}
  {\end{equation*}}

\newenvironment{eqa*}
  {\begin{eqnarray*}}
  {\end{eqnarray*}}

\begin{document}

\title{\LARGE{\textbf{Degenerate fixed points of maps in Banach spaces and lattices with decay and their invariant manifolds}}}%
\author[1,3]{Inmaculada Baldomá}
\author[1,3]{Pau Martín}
\author[2]{Donato Scarcella}
\affil[1]{Departament de Matemàtiques, Universitat Politècnica de Catalunya, Diagonal, 647, 08028 Barcelona, Spain.}
\affil[2]{Departament de Matemàtiques i Informàtica, Universitat de Barcelona (UB), Gran Via, 585, 08007 Barcelona, Spain.}
\affil[3]{Centre de Recerca Matemàtica (CRM), Carrer de l'Albareda, 08193 Bellaterra, Spain}

\affil[ ]{\textit{Contributing authors:} immaculada.baldoma@upc.edu; p.martin@upc.edu; donato.scarcella@ub.edu}
\date{\null}%
\maketitle

\begin{abstract}
Degenerate fixed points of maps appear in many interesting problems in Celestial Mechanics~\cite{Moser01}, Economics~\cite{Lee2021} and Chemistry~\cite{BorFM26}. Lattice systems, that is, dynamical systems consisting in an infinite array of finite dimensional subsystems interacting locally among them appear in many models of Biology, Physics and Mathematics. In this work, we extend the results concerning the existence of stable and unstable invariant manifolds of degenerate fixed points to lattice systems with decay properties. As an example, we find such manifolds in perturbations of the Toda lattice.
\end{abstract}

\tableofcontents

\section{Introduction}

Dynamical systems consisting in an infinite array of finite dimensional subsystems interacting locally among them appear in many models of Biology, Physics and Mathematics.
In Biology, they have been used to model arrays of neurons (\cite{Hoppenstaedt86,HoppenstaedtI97,BorisyukEFT05,Izhikevich07}). In statistical mechanics, one finds arrays of coupled oscillators in the famous Fermi-Pasta-Ulam-Tsingou experiment~\cite{FPU}. There is huge amount of literature dealing with this type of systems,
see for instance the surveys~\cite{FPU08,BatesCC03,Mallet-Paret03,ChazottesF05,Kan93,Peyrard04,BraunK98,BraunK04,MR2023441} and the references therein, which cover several points of view.
In the present work, we will use the formalism introduced in~\cite{FdlLM11} to quantify the decay of the interactions with respect to the distance of the nodes in the lattice. This formalism has been used in~\cite{FONTICH20093136} to obtain whiskered tori in this setting. 
See also~\cite{dolgopyat2026kolmogorovinvarianttorustheorem}, where full dimensional tori are obtained. Also in this setting, the existence of Arnold diffusion in some model of lattices with decay was studied in~\cite{https://doi.org/10.1002/cpa.22191}.

The purpose of the present work is to extend to the setting of lattice maps with decay (or vector fields in lattices) the study of the existence of stable and unstable invariant manifolds associated to degenerate  fixed points. In here, a degenerate (or parabolic) fixed point of a map is a fixed point where the map is tangent to the identity (in the case of vector fields, tangent to the 0 vector field).
The case of hyperbolic fixed points and hyperbolic sets in lattice maps with decay was studied in~\cite{FONTICH20112887}, where, in particular, it was shown that their invariant manifolds inherit the decay properties of the map. 

Degenerate fixed points appear, for instance, at the parabolic infinity in Sitnikov problem and the restricted 3-body problem (see~\cite{McGehee73}). Their invariant manifolds give rise to very rich homoclinic phenomena~\cite{Moser01}. Other examples of such behavior can be found in Economics~\cite{Lee2021} and Chemistry~\cite{BorFM26}. In these examples, the stable and unstable invariant manifolds of the degenerate fixed point are 1-dimensional.
The case of higher but finite dimensional invariant manifolds of degenerate fixed points was studied in~\cite{BFM1,BFM2}, and extended to  degenerate degenerate invariant tori in~\cite{BFMARMA}, with applications to celestial mechanics.

 The purpose of the present paper is to extend these results to infinite dimensions, with special emphasis in applications to dynamical systems defined on lattices with decay. The extension of the results in~\cite{BFM1,BFM2} to Banach spaces is rather straightforward, but we include it for completeness and because these results are needed to deal with the particular case of lattice systems with decay, which we also develop. As an example of application, we apply the results to perturbations of the Toda lattice~\cite{Toda67}.

 To find and describe the invariant manifolds of degenerate fixed points, we use a version of the parametrization method~\cite{CabreFL03a} adapted to this setting. Let us briefly describe the method. Let $F:B\subset E \to E$ be a smooth or analytic map, where $B$ is a ball around 0 in the Banach space $E$. Assume that $F$ is a local diffeo around 0 and 
 $F(0) = 0$. The parametrization method looks for an embedding $K:V\to E$, such that $0 \in K(V)$ or $0 \in \mathrm{closure}(K(V))$, and a reparametrization $R:V\to V$ such that 
 \begin{equation}
 \label{eq:invariancia_intro}
 F \circ K = K \circ R.
 \end{equation}
 If such $K$ and $R$ are found, then $K(V)$ is an invariant manifold.
 The dynamics of $F$ on $K(V)$ is conjugated to $R$. In many situations, there is certain freedom to choose $R$. In our setting, we will be able to choose $R$ as a finite sum of suitable functions. 
 It is worth to remark that, in our case, the fixed point is degenerate. Hence, dynamics is essentially driven by the non-linear terms of $F$, unlike what happens in the hyperbolic case. We remark that, under suitable assumptions, $V$ is an open subset of $E$. Then, $K$ and $R$ can be interpreted as a normal form procedure (see Corollary~\ref{cor:conjucagio}).

 We formulate our theorems as \emph{a posteriori} statements. On the one hand, we prove that, if certain approximate solution of~\eqref{eq:invariancia_intro} is found, a true solution can be found. On the other,  we find approximate solutions of the invariance equation~\eqref{eq:invariancia_intro} in an algorithmic way (which is already a nontrivial statement, also depending on the non-linear terms of $F$). The existence of the invariant manifolds is a consequence of both results.

 We pay special attention to lattices, which we model over $\ell^{\infty}$, which means that the maps have well defined components and variables. The maps we consider are \emph{short range}, that is, the dependence of the component $j$ of the map with respect to the variable $i$ goes to 0 when $|i-j|$ goes to 0. This decay is quantified by a decay function $\Gamma$ (see Definition~\ref{def:decay_fun}).

 We formulate our theorems both for general Banach spaces, Theorems~\ref{ThmMapsPSBanach} and~\ref{ThmMapsASBanach}, where the existence and regularity of invariant manifolds of degenerate fixed points is established,  and, with additional hypothesis, for  lattices with decay (Theorems~\ref{ThmMapsPS} and~\ref{ThmMapsAS}), where we prove that the invariant manifolds inherit the decay properties of the maps.

 We stress that we will focus only on the proofs of Theorems~\ref{ThmMapsPS}
and~\ref{ThmMapsAS}. The proofs of the corresponding results for maps on general Banach
spaces, namely Theorems~\ref{ThmMapsPSBanach} and~\ref{ThmMapsASBanach}, are omitted, since they follow the same
scheme of the ones in~\cite{BFM1,BFM2}. Indeed, as the reader can verify, the arguments given below can be adapted
to Banach space setting with only straightforward modifications.

 It is worth to remark that we work with maps in the analytic category. This implies the use of complexifications of Banach spaces and functions on them. For completeness, we have included these details in Appendix~\ref{A}. The assumption of analyticity is based in two facts. The first one is that many examples fall into this category. The second one is that, although the result probably holds if the maps are $C^r$, the proof would be much more cumbersome.

As an application, in Section~\ref{sec:Todalattice}, we apply our results to perturbations of the Toda lattice, to obtain stable invariant manifolds of the degenerate fixed point at infinity. 

The structure of the paper is as follows. In Section~\ref{sec:main_results_i_Toda}
we introduce the definitions, problem, main statements and the application to the Toda lattice.
The rest of the paper is devoted to the proof of Theorems~\ref{ThmMapsPS} and~\ref{ThmMapsAS}. Section~\ref{section_cohomological_equation} is dedicated to studying the existence of analytic homogeneous solutions of some partial differential equation that is the model of all the cohomological equations we need to solve. The proof of Theorem~\ref{ThmMapsAS} is contained in Section~\ref{Proof_ThmMapsAS}, whereas the one of Theorem~\ref{ThmMapsPS} is contained in Section \ref{Proof_ThmMapsPS}.  Finally, Appendix~\ref{A} contains some results on complexification of Banach spaces that we include here for the sake of completeness and Appendix~\ref{B} some technical definitions on maps with decay.

\section*{Acknowledgements}
I.B., P.M. and D.S. have been partially supported by the grant PID2024-158570NB-I00  funded by
MCIN/AEI/10.13039/501100011033 and “ERDF A way of making Europe”.  D.S. also acknowledges partial support from the ICREA Acadèmia 2023 grant awarded to Dr. Marcel Guardia Munàrriz. 
This work is also supported by the Spanish State Research Agency, through the Severo Ochoa and Mar\'ia de Maeztu Program for Centers and Units of Excellence in R\&D (CEX2020-001084-M).

\section{Invariant manifolds of parabolic fixed points}
\label{sec:main_results_i_Toda}
We devote this section to introduce the notations and definitions we need to present our theorems, as well to the statement of the main results of the article. More concretely, in Section~\ref{Notation}  we introduce several notations we use throughout the paper. In Section~\ref{SetUp_Hyp}, we define the maps we are interested in and a the assumptions we need in order to ensure the existence of invariant manifolds. Finally, in~Section \ref{Main_Results}, we state the main results.

\subsection{Notation}\label{Notation}

This section is divided into two parts. First, in Section~\ref{sec: notation Banach spaces}, we introduce the notation and conventions used throughout. These will be employed in the sequel without further explicit reference. Most of them appear in previous works concerning the existence of invariant manifolds associated with parabolic fixed points or tori~\cite{BFM1, BFM2, BFMARMA}. Second, in Section~\ref{sec: notation lattices}, we introduce the definition of the Banach spaces of functions defined on $d$-dimensional lattices with decay, following~\cite{FdlLM11,FONTICH20093136,FdlLY15}.

\subsubsection{General notation for Banach spaces}\label{sec: notation Banach spaces}
We fix the following notation for sets:
\begin{itemize}
\item Let $E$ be a Banach space. Given $r>0$ and $x_0 \in E$, we denote by $B_r(x_0) \subset E$ the open ball of radius $r$ centered at $x_0$. When $x_0=0$, we simply write $B_r$. Moreover, for a given set $U \subset E$,   $\overline{U}$ will denote its closure. 
\item  Let $E_1$ and $E_2$ be Banach spaces with associated norms $|\cdot|_1$ and $|\cdot|_2$, respectively. When we consider the product space $E_1 \times E_2$, for each $z = (x, y) \in E_1 \times E_2$, we define the norm $|z| = \max\{|x|_1, |y|_2\}$. 
The space $E_1 \times E_2$ equipped with the above-mentioned norm $|\cdot|$ is a Banach space.

We introduce the natural projections $\pi_x(x,y) = x$ and $\pi_y(x,y) = y$.  With the product topology, they are continuous. 
To shorten notation, we will write $F_* = \pi_* \circ F$, $* = x,y$.


In what follows, we denote by $E$, $E_1$, $E_2$,  $E_3$, and $G$ Banach spaces. Moreover, for notational simplicity, and whenever no confusion arises, we use the same symbol $|\cdot|$ for the norms associated with all these spaces, the meaning being clear from the context.
\item We denote by $\C E$ a complexification of the Banach space $E$ (we refer to Appendix \ref{A} for a very brief introduction to the complexification of Banach spaces). In this work, whenever we consider the complexification of a Banach space, we adopt the notation $\C E = E \oplus iE$ (see~\eqref{EiE}) and, for all $x \in \C E$, we denote $x = \mathrm{Re}x + i\mathrm{Im}x$. Moreover, we equip $\C E$ with the following \emph{reasonable complexification} norm
\begin{equation}
|x| = \max \{|\mathrm{Re}\,x|, |\mathrm{Im}\, x|\} \hspace{2mm} \mbox{for all $x \in \C E$}
\end{equation}
where Definition \ref{reasonablecomplex} provides the notion of a  \emph{reasonable norm}, and Proposition \ref{ReasNorm} shows that the above norm is indeed reasonable on the complexification $\C E$.
Furthermore, we will deal with complex linear operators $M$ defined on, and taking values in, suitable complexifications $\C E$ of Banach spaces $E$. We write $M = M_1 + i M_2$ where $M_1$ and $M_2$ are real linear operators. Consequently, $|M| \le |M_1| + |M_2|$. 
\item Given a function $f: U \subset E \to G$, we denote by $Df(x)$ its Fréchet differential evaluated at $x \in U$. For functions of the form $f: U \subset E_1 \times E_2 \to G$, we denote by $D_x f(x,y)$ and $D_y f(x,y)$ the partial derivatives with respect to the variables $x \in E_1$ and $y \in E_2$, respectively, evaluated at $(x,y) \in U$.
\item  We will use analogous definitions for functions depending on parameters.  Let $\Lambda \subset E_3$ be an open set of parameters. The notions introduced above naturally extend to functions defined on $U \times \Lambda$.
\item  When considering functions depending on parameters $\lambda \in \Lambda$, if the composition $f(x,\lambda)=h(g(x,\lambda),\lambda)$ is well-defined,  we simply write $f = h \circ g$, leaving the dependence on $\lambda$ implicit.

\item In this work, we look for invariant manifolds as the sum of homogeneous functions. For this reason, it is convenient to introduce the following spaces of functions. Given  
an open set $U \subset E$, an open set of parameters $\Lambda \subset E_3$, and a parameter $\ell \in \N$, we define
\begin{align}\label{Hspaces}
\mathcal{H}^{\ge \ell} =& \left\{ h \in C^0(U\times \Lambda, G) \;\middle| \; \begin{aligned} &\mbox{for $u \in U$, $|h(u ,\lambda)| = \mathcal{O}\left(|u|^{\ell}\right)$}\\
 & \mbox{ uniformly in $\lambda \in \Lambda $} \end{aligned}\right\} \notag  \\
 \mathcal{H}^{> \ell} =& \left\{ h \in C^0(U \times \Lambda, G) \;\middle| \; \begin{aligned} &\mbox{for $u \in U$, $|h(u,\lambda)| = o\left(|u|^{\ell}\right)$}\\
 & \mbox{ uniformly in $\lambda \in \Lambda $} \end{aligned}\right\}\\
 \mathcal{H}^\ell =& \left\{ h \in C^0(U \times \Lambda, G) \;\middle| \; \begin{aligned} &\mbox{$\forall \mu \in \R$, $\forall u \in U$, s.t. $\mu u \in U$}\\
 &h(\mu u, \lambda) = \mu^\ell h(u, \lambda), \hspace{3mm} \mbox{$\forall \lambda \in \Lambda$}
 \end{aligned}\right\} \notag  
\end{align}
In this work, we will omit the reference to $U$,  $\Lambda$, and $G$. It will be specified by the context.
\end{itemize}

\subsubsection{General notation for lattices}\label{sec: notation lattices}
We are interested in studying maps of the form $F:U\subset E_1 \times E_2 \to E_1 \times E_2$ having a parabolic fixed point. We aim to establish some conditions to ensure the existence of invariant manifolds. When $E_1$ and $E_2$ are replaced by lattices and $F$ exhibits some decay properties, we aim to ensure that the invariant manifolds inherit the same properties. For this reason, we want to introduce some Banach spaces related to $d$-dimensional lattices.  For this purpose, we have the following
\begin{definition}
\label{EllInfty}
Let $\mathcal{X} = \{\mathcal{X}_i\}_{i \in \Z^d}$ be a family of Banach spaces and for each $i \in \Z^d$ we denote by $|\cdot|_i$ the norm associated with $\mathcal{X}_i$.  We define
\begin{equation*}
\ell^\infty \left(\mathcal{X}\right) = \left\{x = \{x_i\}_{i  \in \Z^d} \in \prod_{i  \in \Z^d} \mathcal{X}_i \;\middle| \; \sup_{i  \in \Z^d}|x_i|_i < \infty \right\}.
\end{equation*}
\end{definition}
We observe that $\ell^\infty \left(\mathcal{X}\right)$, equipped with the norm $|x| = \sup_{i \in \mathbb{Z}^d} |x_i|_i$, is a Banach space.  For $j \in \Z^d$, we denote by $\pi_j : \ell^\infty \left(\mathcal{X}\right) \to \mathcal{X}_j$ the projection $\pi_j(x = \{x_i\}_{i \in \Z^d}) = x_j$. Clearly, $|\pi_j|=1$ for all $i \in \Z^d$.

Since we aim to study maps with specific decay properties, it is natural to introduce appropriate weighted norms. To this end, following~\cite{JdlL00}, we first define a decay function, which will then be used to construct Banach spaces of maps that satisfy the desired decay conditions. 

\begin{definition}\label{def:decay_fun}
\textit{A decay function} is a map $\Gamma : \Z^d \to \R^+$ such that 
\begin{enumerate}
\item $\sum_{i \in\Z^d} \Gamma(i) \le 1$,
\item $\sum_{i \in\Z^d}  \Gamma(i-j)\Gamma(j-k)\le \Gamma(i-k) \quad \mbox{with $i,k \in\Z^d$}$.
\end{enumerate}
\end{definition}
In~\cite{JdlL00}, it was shown that such functions exist.
Here, we will consider $\Gamma$ as a fixed decay function. We will use it to measure how $j$ component of the maps depends on the $i$ variable.  We introduce some spaces of maps defined on $\ell^\infty$ having decay properties associated with the decay function $\Gamma$. 

We consider the following families of Banach spaces $\mathcal{X} = \{\mathcal{X}_i\}_{i \in \Z^d}$ and $\mathcal{Y} = \{\mathcal{Y}_i\}_{i \in \Z^d}$ and, following~\cite{FdlLM11}, we define the Banach space of linear maps with decay $\Gamma$ by
\begin{equation}\label{LGamma}
L_\Gamma \left(\ell^\infty\left(\mathcal{X}\right), \ell^\infty\left(\mathcal{Y}\right)\right) = \left\{ A \in L \left(\ell^\infty\left(\mathcal{X}\right), \ell^\infty\left(\mathcal{Y}\right)\right) \;\middle| \; |A|_\Gamma < \infty \right\},
\end{equation}
where $L$ stands for the space of continuous linear maps, and the associated norm is given by
\begin{equation}\label{Def:GammaLinearNorm}
|A|_\Gamma = \max\left\{|A|, \gamma (A)\right\}, \quad \mbox{with} \quad \gamma(A) = \sup_{i,j \in \Z^d} \sup_{\substack{|u|\le 1 \\ \pi_l u=0, l \ne j}} \left|\left(Au\right)_i\right|\Gamma(i-j)^{-1}.
\end{equation}
The underlying idea of a linear map with decay is that if a vector $v \in \ell^\infty(\mathcal{X})$ has its ‘mass’ concentrated around its $j$-th component, then $Av$ will likewise exhibit its ‘mass’ concentrated around the same component, with the same rate of decay. This is the property that characterizes linear maps with decay $\Gamma$. 

We remark that the quantity $\gamma(A)$ in~\eqref{Def:GammaLinearNorm} is only a seminorm. It is possible to find non-trivial maps $A$ such that $\gamma(A) = 0$. This is the reason to include $|A|$ in the definition of the norm. This is because $\ell^\infty(\mathcal{X})$ is not separable.

In the present paper, we deal with real-analytic functions. For this reason, we need to introduce the Banach space of the $C^1$ functions with decay between $\ell^\infty$ spaces.  Given an open set of parameters $\Lambda \subset E_3$,  a third family of Banach spaces $\mathcal{Z} = \{\mathcal{Z}_i\}_{i \in \Z^d}$ and an open subset $U \subset \ell^\infty\left(\mathcal{X}\right) \times \ell^\infty\left(\mathcal{Y}\right)$, we define
\[
C^1_\Gamma \left(U \times \Lambda, \ell^\infty\left(\mathcal{Z}\right)\right) = \left \{ F \in C^1 \left(U \times \Lambda, \ell^\infty \left(\mathcal{Z}\right)\right), \; |F|_{C^1_\Gamma} <\infty\right \}
\]
with the norm 
\begin{multline*}
    |F|_{C^1_\Gamma} = \max \left\{\sup_{(x,y,\lambda) \in U\times \Lambda} \left|F(x,y,\lambda)\right|, \sup_{(x,y,\lambda) \in U\times \Lambda}|D_xF(x,y,\lambda)|_\Gamma, \right .\\  \left .\sup_{(x,y,\lambda) \in U\times \Lambda}|D_yF(x,y,\lambda)|_\Gamma\right\}.
\end{multline*}
The space of functions $C^1_\Gamma \left(U \times \Lambda, \ell^\infty\left(\mathcal{Z}\right)\right)$ endowed with $|F|_{C^1_\Gamma}$ is a Banach space.  When 
$U \subset \ell^\infty\left(\mathcal{X}\right)$ is an open subset, we define
\begin{align}\label{C1_GammaSpace}
C^1_\Gamma \left(U \times \Lambda, \ell^\infty\left(\mathcal{Y}\right)\right) = &\Big\{ F \in C^1 \left(U \times \Lambda,\ell^\infty \left(\mathcal{Y}\right)\right) \,\big|\, DF(x , \lambda) \in L_\Gamma \hspace{2mm} \forall (x, \lambda) \in U \times \Lambda, \nonumber\\
& \sup_{(x, \lambda) \in U \times \Lambda} \left|F(x, \lambda)\right| < \infty, \quad  \sup_{(x, \lambda) \in U \times \Lambda} \left|DF(x, \lambda)\right|_\Gamma < \infty \Big\}
\end{align}
with the norm
\begin{equation}\label{C1_GammaNorm}
    |F|_{C^1_\Gamma} = \max \left\{\sup_{(x, \lambda) \in U \times \Lambda}|F(x, \lambda)|, \hspace{2mm}\sup_{(x, \lambda) \in U \times \Lambda}|DF(x, \lambda)|_\Gamma\right\}.
\end{equation}
We will use the same notation for maps defined on some open subset of suitable complexifications of $\ell^\infty\left(\mathcal{X}\right) \times \ell^\infty\left(\mathcal{Y}\right)$ or $\ell^\infty\left(\mathcal{X}\right)$.

A brief summary of the fundamental properties of the above norms is provided in Appendix \ref{B}. For further details, we refer to~\cite{MdlL00},~\cite{FdlLM11}, and~\cite{FdlLY15}.

\subsection{Set up and hypothesis}\label{SetUp_Hyp}
In this section, we describe the maps we are considering and the hypotheses we need to ensure the existence of invariant manifolds associated to a degenerate fixed point. 
Let $U \subset E_1 \times E_2$ be an open set such that $(0,0) \in U$ and let $\Lambda \subset E_3$ be an open set of parameters. Given positive integers $N$, $M\ \ge 2$ and $r \ge N$,  for all $\lambda \in \Lambda$, we consider the following map
\begin{equation}
\label{F}
\begin{cases}
F_\lambda:U \longrightarrow E_1 \times E_2, \quad F_\lambda(x,y) = \begin{pmatrix} x + p(x,y, \lambda) + f(x,y, \lambda)\\
y + q(x,y, \lambda) + g(x,y, \lambda) \end{pmatrix},\\
f(x,y, \lambda) = F_x^{N+1}(x,y, \lambda)+ \cdots + F^r_x(x,y, \lambda) + F_x^{>r}(x,y, \lambda),\\
g(x,y, \lambda) = F_y^{M+1}(x,y, \lambda)+ \cdots + F^r_y(x,y, \lambda) + F_y^{>r}(x,y, \lambda),\\
p \in \mathcal{H}^N, \hspace{2mm} q \in \mathcal{H}^M, \hspace{2mm} F_x^j, F_y^j \in \mathcal{H}^j \hspace{2mm}\mbox{and} \hspace{2mm} F_x^{>r}, F_y^{>r} \in \mathcal{H}^{>r},\\
D^j F_x^{>r}, D^j F_y^{>r} \in \mathcal{H}^{>r-j} \hspace{2mm}\mbox{for $j=1,2$},\\
p, q, F_x^j, F_y^j, F_x^{>r}, \hspace{1mm}\mbox{and} \hspace{1mm} F_y^{>r} \hspace{1mm}\mbox{are analytic}.
\end{cases}
\end{equation}
It is straightforward to verify that $(0,0)$ is a fixed point for $F_\lambda$ and $DF_{\lambda}(0,0) = \mathrm{Id}$  for all $\lambda \in \Lambda$.  

The aim of the present paper is to provide conditions for the existence of invariant manifolds tangent to $E_1\times \{0\}$ at $(0,0)$, the $x$-subspace, for maps as in~\eqref{F}. First, we consider the case where $E_1$ and $E_2$ are general Banach spaces. Then, we will take $E_1$ and $E_2$ to be lattices as in Definition~\ref{EllInfty}, assume that $F$ satisfies some decay properties, and we will prove that the invariant manifolds inherit these decay properties. 

We introduce the stable set of $(0,0)$ related to $\mathcal{U} \subset E_1\times E_2$ as 
\begin{equation}\label{def stable set}
    W_\mathcal{U}^{\mathrm{s}} = \left \{ (x,y) \in E_1\times E_2 \;\middle|\;  F^k_\lambda (x,y) \in \mathcal{U}, \,\forall k\geq 0, \,
     \lim_{k\to \infty} F^k_\lambda(x,y) = (0,0)   \right \}
\end{equation}
and analogously $W_\mathcal{U}^{\mathrm{u}}$ is the stable set of $F^{-1}_\lambda$ related to $\mathcal{U}$. In order to describe $W_{\mathcal{U}}^{\mathrm{s}}$, we will use the so-called parameterization method. It consists of two main steps. First, an approximate parameterization of the invariant manifold is explicitly constructed as a solution to a suitable functional equation. Then, a fixed-point argument refines this approximate solution, yielding an invariant manifold for the dynamical system under consideration. We refer to~\cite{CFL03a,CFL03b,CFL03c,HCLFM16} for a general presentation of the method and to~\cite{BFM1, BFM2, BFM20, BFLM07, BFMARMA} for the application of this method to parabolic objects. 

First, we need to introduce some special sets. Given $V \subset E_1$ such that $0 \in \partial V$ and $\varrho >0$, we introduce 
\begin{equation}
\label{Vrho}
V_\varrho = V \cap B_\varrho.
\end{equation}
\begin{definition}\label{def:star-shaped}
Given $V \subset E_1$, we say that $V$ is star-shaped with respect to $0$ if $0 \in \partial V$ and for all $x \in V$ and $\mu \in (0, 1]$,  $\mu x \in V$. 
\end{definition}

The reason behind Definition~\ref{def:star-shaped} is that in the parabolic case a stable invariant manifold, instead of being defined in a whole neighborhood of $0$, it is only defined over a domain $V$ such that $0 \in \partial V$.  Nonetheless, some regularity at the origin is still preserved. This motivates the following definition.

\begin{definition}
\label{C1reg}
Let $V\subset E_1$, an open set, $x_0 \in \overline{V}$ and $f:V\cup\{x_0\}\subset E_1 \to E_2$. We say that $f$ is $C^1$ at $x_0$ if $f$ is $C^1$ in $V\cap ({B_\varepsilon(x_0)\setminus \{x_0\})}$, for some $\varepsilon>0$ and $\lim_{x \to x_0, x\in V}Df(x)$ exists.
\end{definition}

We consider the following hypotheses on the map $F$ in~\eqref{F}.
We assume the existence of  $V \subset E_1$, open and star-shaped with respect to $0$, and  $\varrho >0$ such  that 
\begin{itemize}
\item[H1] The homogeneous function $p$ satisfies the following weak contracion condition
\begin{equation*}
a_p = - \sup_{x \in V_\varrho, \, \lambda \in \Lambda}\displaystyle{|x + p(x,0, \lambda)| - |x| \over |x|^N}>0. 
\end{equation*}
\item[H2]The homogeneous function $q$ satisfies 
\begin{equation*}
    q(x,0,\lambda) =0 \hspace{2mm} \mbox{for all $x \in V_\varrho$  and $\lambda \in \Lambda$}.
\end{equation*}
\item[H3] There exists a constant $a_V>0$ such that, for all $x \in V_\varrho$  and $\lambda \in \Lambda$.
\begin{equation*}
\mathrm{dist}(x + p(x,0,  \lambda), (V_\varrho)^c) \ge a_V|x|^N.
\end{equation*}
where $(V_\varrho)^c$ is the complementary set of $V_\varrho \subset E_1$.
We will see that the latter implies that the set $V_\varrho$ is invariant for the map $x \to x + p(x,0,  \lambda)$.
\end{itemize}
Following~\cite{BFM1, BFM2, BFMARMA}, in order to quantify the weight of the homogeneous functions $p$ and $q$, we also consider the following constants
\begin{align}
&b_p = \sup_{x \in V_\varrho, \lambda \in \Lambda}{|p(x,0, \lambda)| \over |x|^N}, & A_p &= - \sup_{x \in V_\varrho , \lambda \in \Lambda}\displaystyle{|\mathrm{Id} + D_xp(x,0, \lambda)| - 1 \over |x|^{N-1}},\nonumber
\end{align}
\begin{gather}\label{constants}
B_p = \sup_{x \in V_\varrho, \lambda \in \Lambda}{|\mathrm{Id} - D_xp(x,0, \lambda})| -1\over |x|^{N-1},  
\\
B_q = - \sup_{x \in V_\varrho, \lambda \in \Lambda}\displaystyle{|\mathrm{Id} - D_yq(x,0, \lambda)| - 1 \over |x|^{M-1}},\nonumber
\end{gather}
\begin{align}
&c_p = \begin{cases} a_p, & \text{if } B_q \le 0,\\
b_p, & \text{otherwise}
\end{cases} 
&d_p &= \begin{cases} a_p, &\text{if } A_p \le 0,\\
b_p, &\text{otherwise}.
\end{cases}
\end{align}
Furthermore, when $E_1 = \ellX$ and $E_2 = \ellY$, we define
\begin{equation}
\label{GammaConstantsApproxThm}
A^\Gamma_p = \sup_{x \in V_\varrho, \lambda \in \Lambda}\displaystyle{|D_xp(x,0, \lambda)|_\Gamma \over |x|^{N-1}}, \quad B^\Gamma_q = \sup_{x \in V_\varrho, \lambda \in \Lambda}\displaystyle{|D_yq(x,0, \lambda)|_\Gamma \over |x|^{M-1}}.
\end{equation}
 We assume that the constants defined in H1, H3,~\eqref{constants} and~\eqref{GammaConstantsApproxThm} are finite.
\begin{remark}
    If the map $F_\lambda$ in~\eqref{F} is independent of the parameter $\lambda$, the finiteness of the above constants follows from the homogeneity of $p$ and $q$.
\end{remark}

\begin{remark}\label{rmk:MleN}
    Let $F_\lambda$ be as in~\eqref{F}, and we assume the existence of an invariant manifold associated to the origin tangent to $\{y=0\}$. Then, after a close-to-identity change of coordinates, $F_\lambda$ has to satisfy that $M \le N$ and $q(x,0 , \lambda)=0$. We refer to Appendix A in~\cite{BFMARMA} for the proof.  For this reason, in the present work, we only consider the case $M \le N$. 
\end{remark}

To finish this section, for a given $\beta >0$, we define the set 
\begin{equation}\label{def Vrho beta}
 V_{\varrho,\beta} = \{(x,y) \in E_1 \times E_2 : x\in V_\varrho, |y| < \beta |x|   \}
\end{equation}
and complex extensions of the sets $V$ and $V_\varrho$. We need to introduce them since they will be the domains on which the invariant manifolds we are looking for are defined. Given $\gamma >0$ we define
\begin{equation}
\label{def:OmegaSet}
\begin{aligned}
\Omega(\gamma) &= \left\{x \in \C E_1 : \mathrm{Re}\,x \in V, \,|\mathrm{Im}\,x|<\gamma|\mathrm{Re}\,x|\right\}\\
\Omega(\varrho, \gamma) &= \left\{x \in \C E_1 : \mathrm{Re}\,x \in V_\varrho, \, |\mathrm{Im}\,x|<\gamma|\mathrm{Re}\,x|\right\}.
\end{aligned}
\end{equation}
We observe that, if $x \in \Omega(\gamma)$ with $\gamma \le 1$, then $|x| = |\mathrm{Re}\,x|$. It will be used in this work without explicit reference.

\subsection{Main results}\label{Main_Results}

We present results for maps $F_ \lambda:U\subset E_1\times E_2 \to E_1\times E_2$ as in~\eqref{F} defined on general Banach spaces in Section~\ref{main results Banach} and for the particular case of lattices, namely $E_1=\ell^{\infty}(\mathcal{X})$, $E_2= \ell^{\infty}(\mathcal{Y})$, in Section~\ref{main results lattices}

In both cases, we first enunciate a \textit{a posteriori} result (see Theorems~\ref{ThmMapsPSBanach} and~\ref{ThmMapsPS}). Roughly speaking, it shows that if we have a good approximation of the invariant manifold associated with the map $F_\lambda$ in~\eqref{F}, then there exists an invariant manifold of $F_\lambda$ close to it. Afterward, we prove that, under some conditions, a good approximation can be found as a sum of homogeneous functions (see Theorems~\ref{ThmMapsASBanach} and~\ref{ThmMapsAS}). 

Finally, in Section~\ref{consequences main results} we present several results which are direct consequences of the previous results. 

\subsubsection{Main results for maps on Banach spaces}\label{main results Banach}

The first result is as follows. 
\begin{theorem}[a posteriori result]\label{ThmMapsPSBanach}
    Let $F_\lambda$ be as in~\eqref{F} satisfying hypotheses H1-H3 for some star-shape $V$ and $\varrho_0 >0$. Furthermore, we assume that $A_p > b_p$ and 
\begin{equation}\label{thm:hyp_constB}
\begin{aligned}
    &B_q >0, \hspace{19mm} \mbox{if $M<N$},\\
    &B_q > -Na_p, \hspace{10mm} \mbox{if $M=N$}.
\end{aligned}
\end{equation}
We fix $r$, $\ell$ and $\ell_0$ in such a way that 
\begin{equation}\label{thm_apost:condition_l}
   r \ge \ell > \ell_0 = \max\left\{N-1 + {B_p \over a_p}, N+1, 2N-M\right\}.
\end{equation}
Moreover, for some $\rho_0>0$, we assume the existence of analytic maps $K^\le: V_{\rho_0}  \times \Lambda\to U$ and $R: V_{\rho_0}  \times \Lambda\to V_{\rho_0}$ of the form
\begin{equation}\label{thm:hyp_existence_KR}
    K^\le (x, \lambda) - (x,0) = \mathcal{O} (|x|^2), \quad R(x, \lambda) - x - p(x,0, \lambda) = \mathcal{O}(|x|^{N+1})
\end{equation}
uniformly in $\lambda \in \Lambda$, such that 
\begin{equation*}
    F_\lambda \circ K^\le - K^\le \circ R = \mathcal{O}(|x|^\ell). 
\end{equation*}

Then, there exists $0 < \rho \le \rho_0$ and a unique function $K^>$, which is 
real-analytic on a complex extension of $V_\rho  \times \Lambda$, 
\begin{equation*}
    K^> : V_\rho \times \Lambda \to U,
\end{equation*}
satisfying that $K^> = \mathcal{O}(|x|^{\ell -N+1})$ uniformly in $\lambda \in \Lambda$ and that $K = K^\le + K^>$ satisfies the invariance equation
\begin{equation*}
    F_\lambda \circ K - K \circ R = 0. 
\end{equation*}
In addition, for some $\beta>0$ small enough, $K( V_{\rho}) \subset W^{\mathrm{s}}_{  V_{\rho,\beta}}$ (see~\eqref{def stable set} and~\eqref{def Vrho beta}) and, if $B_q>0$, then 
$K(\widehat V_{\rho}) = W^{\mathrm{s}}_{\widehat V_{\rho,\beta}}$ with $\widehat V_\rho$ defined as in~\eqref{Vrho} for $\widehat V \subset V$.

\end{theorem}

The following theorem provides conditions that guarantee the existence of an approximation of the invariant manifolds associated with $F$ satisfying the hypotheses of Theorem \ref{ThmMapsPS}. We express them as a sum of homogeneous functions in the variable $x$. This construction is explicit and allows for considerable freedom in their choice. We focus on solutions that yield the simplest possible representation of the dynamics on the stable manifolds. 

\begin{theorem}[approximated result]
\label{ThmMapsASBanach}
Let $F_\lambda$ be as in~\eqref{F}, and we assume that it satisfies hypotheses H1-H3 for some star-shaped set $V$ and $\varrho_0 >0$. Furthermore, we assume that 
\begin{equation}\label{thm:hyp_ApdpMN}
A_p > b_p\quad \mbox{or} \quad  M<N
\end{equation}
and
\begin{equation}
\begin{aligned}\label{Thm:hyp_2}
&D_yq(x,0 , \lambda) \hspace{1mm} \mbox{is invertible $\forall (x, \lambda) \in V_\varrho  \times \Lambda$ \hspace{5mm} if $M<N$,}\\
&2 + {B_q \over c_p} > 0 \hspace{59mm} \mbox{if $M=N$.}
\end{aligned}
\end{equation}
Then, for any $N \le \ell \le r$ there exist $0 < \varrho \le \varrho_0$ and analytic maps $K :V_\varrho \times \Lambda \to U$ and $R:V_\varrho  \times \Lambda \to V_\varrho$ such that 
\begin{equation}
\label{thm:approx_sol}
F_\lambda \circ K - K \circ R \in \mathcal{H}^{>\ell}.
\end{equation}
Moreover, $K$ and $R$ are $C^1$ at the origin in the sense of Definition \ref{C1reg}, and $K$ and $R$ can be represented as a sum of analytic homogeneous functions $K^j$, $R^j \in \mathcal{H}^j$ of the form
\begin{equation}\label{thm:formaKR}
\begin{aligned}
&K_x (x , \lambda) = x + \sum_{l = 2}^{\ell - N+1} K_x^l(x, \lambda), \quad K_y (x, \lambda) = \sum_{l = 2}^{\ell - M+1} K_y^l(x, \lambda), \\
&R(x, \lambda) = x + p(x,0, \lambda) +\sum_{l = N+1}^{\min\{\ell, \ell_*\}}R^l(x, \lambda)
\end{aligned}
\end{equation}
with $\ell_*$ defined by
\begin{equation}\label{thm:ellell*}
\ell_* = \begin{cases}N-1 + \left[{B_p \over a_p}\right] \quad \mbox{if $A_p > b_p$ and $M = N$},\\
\ell \hspace{28mm} \mbox{if $M<N$}.
\end{cases}
\end{equation}
\end{theorem}

\begin{remark}
    We observe that one can extend $K$ and $R$ to $V$ by the homogeneity of their terms. 
\end{remark}

The proof of the above theorems follows essentially the same strategy as the proof of the corresponding finite-dimensional results in~\cite{BFM1,BFM2}. The only substantial difference concerns how we establish the analyticity of the homogeneous solutions $K$ and $R$ in Theorem \ref{ThmMapsASBanach} and $K^{>}$ in Theorem \ref{ThmMapsPSBanach}. In this case, it is proven using the theory recalled in Appendix \ref{A} concerning analytic functions defined on Banach spaces. In particular, Theorem \ref{TheoremAnaiytic} allows us to deduce analyticity from weak analyticity (see Definition \ref{weakanalytic}) and local boundedness.

\subsubsection{Main results for maps on lattices}\label{main results lattices}

In the case of lattices, under further conditions controlling the decay of the functions involved in the previous results, it is possible to specify the decay properties of the parameterization of the invariant manifold. We remark that, unlike what happens in Theorem~\ref{ThmMapsPSBanach}, next theorem only holds if $N=M$ in~\eqref{F}.

\begin{theorem}[a posteriori result]\label{ThmMapsPS}
Consider $E_1 = \ellX$, $E_2 = \ellY$ and let $F_\lambda: U\subset E_1\times E_2 \to E_1 \times E_2$ as in~\eqref{F}. Assume $N=M$ and that the conditions of Theorem~\ref{ThmMapsPSBanach} for some star-shaped $V$, $\varrho_0>0$ and $K^{\leq} : V_{\rho_0} \times \Lambda \to U$ hold.

Let $\rho>0$, $R:V_\rho \times \Lambda  \to V_\rho$ and $K^>:V_\rho \times \Lambda  \to U$ such that $K=K^{\ell} + K^>$ satisfies the invariance equation $F_\lambda \circ K - K\circ R=0$. 

We fix $r, \ell $ such that 
\begin{equation}
    r \ge \ell > \ell_1 = N + {\max\{A_p^\Gamma, B_p^\Gamma\} + A_p^\Gamma \over a_p}
\end{equation}
and we assume that $p,F_x^i$ satisfy that there exist $\sigma>0$, $\gamma_0>0$ such that 
\begin{equation}
\begin{aligned}
\label{thm:aporsteriori_sol_Hyp_decay}
&p, F_x^i \in C^1_\Gamma \left(B_\sigma \times B_\sigma \times \Lambda , \C\ell^{\infty}\left(\mathcal{X}\right)\right), 
&q, F_y^i \in C^1_\Gamma \left(B_\sigma \times B_\sigma  \times \Lambda , \C\ell^{\infty}\left(\mathcal{Y}\right)\right)
\end{aligned}
\end{equation}
for all $N+1 \le i \le \ell-1$.  

Assume that $R$ satisfies the following decay property: 
\begin{equation}\label{thm:aporsteriori_sol_Hyp_decay_2}
\begin{aligned}
   &\sup_{(x, \lambda) \in  \Omega(\varrho_0, \gamma_0) \times \Lambda}  {\left|D_xR(x ,\lambda ) - \mathrm{Id} - D_xp(x,0 ,\lambda)\right|_\Gamma \over |x|^N} <\infty,\\
   &\sup_{(x, \lambda) \in  \Omega(\varrho_0, \gamma_0) \times \Lambda}  {\left|D_xK^\le(x ,\lambda) - (\mathrm{Id},0)\right|_\Gamma \over |x|} <\infty.
\end{aligned}
\end{equation}

Then, there exist $0< \varrho_1 \le \varrho_0$ and $0<\gamma_1 \le \gamma_0$ such that, for all $0 < \varrho \le \varrho_1$, and $0<\gamma<\gamma_1$, $K^> \in C^1_\Gamma(\Omega(\varrho,\gamma) \times \Lambda, \C\ell^{\infty}\left(\mathcal{X}\right) \times \C\ell^{\infty}\left(\mathcal{Y}\right))$ and
\begin{equation*}
   \sup_{(x, \lambda) \in  \Omega(\varrho, \gamma)  \times \Lambda}{|D_xK^>(x ,\lambda)|_\Gamma \over |x|^{\ell -N+1}}< \infty.
\end{equation*}
\end{theorem}

The following theorem is the counterpart for lattices of the approximation result, Theorem~\ref{ThmMapsASBanach}. This theorem, that provides approximate solutions of the invariance equation, unlike Theorem~\ref{ThmMapsPS}, is also valid if $M<N$.


\begin{theorem}[approximated result]
\label{ThmMapsAS}
Consider $E_1 = \ellX$, $E_2 = \ellY$ and let $F_\lambda: U\subset E_1\times E_2 \to E_1 \times E_2$ as in~\eqref{F} under the conditions of Theorem~\ref{ThmMapsPSBanach} for some star-shaped $V$, $\varrho_0>0$ and $K^{\leq} : V_{\rho_0} \times \Lambda \to U$. 

Let $K_x^l, K_y^l, R^l \in \mathcal{H}^l$ as in~\eqref{thm:formaKR} satisfying~\eqref{thm:approx_sol}. We assume the existence of $\sigma>0$ such that 
\begin{equation}
\begin{aligned}
\label{thm:approx_sol_Hyp_decay}
&p, F_x^i \in C^1_\Gamma \left(B_\sigma \times B_\sigma  \times \Lambda, \C\ell^{\infty}\left(\mathcal{X}\right)\right), \quad F_x^\ell(\cdot,0) \in C^1_\Gamma \left(B_\sigma\times \Lambda, \C\ell^{\infty}\left(\mathcal{X}\right)\right)\\
&q, F_y^j \in C^1_\Gamma \left(B_\sigma \times B_\sigma  \times \Lambda, \C\ell^{\infty}\left(\mathcal{Y}\right)\right)\hspace{6mm} F_y^\ell(\cdot,0) \in C^1_\Gamma \left(B_\sigma  \times \Lambda, \C\ell^{\infty}\left(\mathcal{Y}\right)\right)\\
&\sup_{(x, \lambda) \in \Omega(\varrho_0, \gamma_0)  \times \Lambda}\left|\left(D_yq(x,0, \lambda)\right)^{-1} \right|_\Gamma < \infty
\end{aligned}
\end{equation}
for all $N+1 \le i\le \ell -1$ and $M+1 \le j\le \ell -1$. In addition, we assume that 
\begin{align}
\label{thm:approx_sol_Cond_Param2}
&{B_q^\Gamma  + A_p^\Gamma \over a_p} < 1 \hspace{19mm} \mbox{if $M=N$}.
\end{align}
Then, there exist $\varrho > 0$ and $\gamma >0$ sufficiently small such that 
\begin{equation}\label{thm:regularity}
\begin{aligned}
  &K_x^l \in  C^1_\Gamma \left(\Omega (\varrho, \gamma) \times \Lambda, \C\ell^{\infty}\left(\mathcal{X}\right)\right) \quad \mbox{for all $l = 2,\dots,  \ell-N+1$,}\\
  & R^l \in  C^1_\Gamma \left(\Omega (\varrho, \gamma) \times \Lambda, \C\ell^{\infty}\left(\mathcal{X}\right)\right) \quad \mbox{for all $l = N,\dots,\min\left\{\ell, \max\left\{\ell_*, \hat \ell\right\}\right\}$,}\\
  &K_y^l \in C^1_\Gamma \left(\Omega (\varrho, \gamma) \times \Lambda, \C\ell^{\infty}\left(\mathcal{Y}\right)\right) \quad \mbox{for all $l = 2,\dots,  \ell-M+1$,}
\end{aligned}
\end{equation}
where $\ell_*$ is defined in~\eqref{thm:ellell*} and $\hat \ell = N + \left[2{A^\Gamma_p \over a_p}\right]$.
\end{theorem}

\begin{remark}
    We emphasize that assumptions \eqref{thm:aporsteriori_sol_Hyp_decay_2} and \eqref{thm:approx_sol_Hyp_decay}, which provide a control with respect to the seminorm $\gamma(\cdot)$ defined in~\eqref{Def:GammaLinearNorm} for some operators on complex extensions of suitable real domains, cannot, in general, be deduced only from the corresponding bounds on the real domain. Indeed, boundedness with respect to the seminorm $\gamma(\cdot)$ is not stable under complexification, even on arbitrarily small complex neighborhoods.We refer to Appendix \ref{app:contr_ex_gamma_norm_ext_real_complex} for an explicit counterexample.
\end{remark}

\begin{remark}
For the sake of notational simplicity,  we omit the explicit dependence on the parameter $\lambda$ throughout the rest of this paper. Unless otherwise specified, all the statements and estimates are understood to hold uniformly with respect to $\lambda \in \Lambda$.
\end{remark}

\subsubsection{Further results}\label{consequences main results}

The results stated in Sections~\ref{main results Banach} and~\eqref{consequences main results} provides several corollaries. 

The first consequence is an existence result which can be generalized to maps, $G$, satisfying that for some $\mathfrak{n} \in \mathbb{N}$, $F=G^{\mathfrak{n}}$ (here $G^{\mathfrak{n}} =G \circ \cdots ^{\mathfrak{n}} \circ G$). In other words, we consider $G:U \subset E_1 \times E_2 \to E_1\times E_2$, $(0,0) \in U $ with $U$ and open subset and 
\begin{equation}\label{def G}
G(x,y) = \begin{pmatrix}
    \mathbf{A} x + \mathbf{f}(x,y) \\ 
    \mathbf{B} y + \mathbf{g}(x,y) 
\end{pmatrix}
\end{equation}
with $\mathbf{A}, \mathbf{B}$ satisfying that there exists some $\mathfrak{n} \in \mathbb{N}$ such that 
\begin{equation} \label{propietat AB}
\mathbf{A}^{\mathfrak{n}} = \mathrm{Id}_{|E_1}, \qquad  \mathbf{B}^{\mathfrak{n}} = \mathrm{Id}_{|E_2}               
\end{equation}
and $\mathbf{f}, \mathbf{g} \in \mathcal{H}^{\geq 2}$ analytic functions. 

\begin{corollary} Let $G$ be as in~\eqref{def G} and let $\mathfrak{n}$ be the minimum integer such that~\eqref{propietat AB} holds. Assume that $F=G^{\mathfrak{n}}$ is under the conditions of Theorems~\ref{ThmMapsPSBanach}-~\ref{ThmMapsASBanach} (resp. Theorems~\ref{ThmMapsPS}-~\ref{ThmMapsAS}) for some $\rho_0 >0$ and let $V, \rho, K, R$ be (respectively) the star-shaped with respect to $0$ (see Definition~\ref{def:star-shaped}), $\rho>0$ and $K, R$ the functions satisfying the invariance equation $F\circ K -K\circ R=0$. 

Then 
\begin{equation}\label{def Vrho beta Fn}
\mathcal{W} := \bigcup_{j=0}^{\mathfrak{n}-1} G^j \left (K (V_\rho)\right ) \subset W_{\mathcal{V}_{\rho,\beta}}^{\mathrm{s}}, \qquad \mathcal{V}_{\rho,\beta} = \bigcup_{j=0}^{\mathfrak{n}-1} G^j (V_{\rho,\beta})
\end{equation}
where $W_{\mathcal{V}_{\rho,\beta}}^{\mathrm{s}}$ is the stable set in~\eqref{def stable set} for the map $G$. In addition, if the constant $B_q$ corresponding to $F$ is positive, 
$\mathcal{W}= W_{\widehat{\mathcal{V}}_{\rho,\beta}}$ where $\widehat{\mathcal{V}}_{\rho,\beta}$ is the set defined in~\eqref{def Vrho beta Fn} related to $\widehat{V}\subset V$. 
\end{corollary}
\begin{remark}
 The counterpart result in the finite-dimensional case was proven in~\cite{BFMARMA} and can be easily adapted to the current infinite-dimensional setting.    
\end{remark}

The second result is a conjugation result whose proofs follows directly from the previous results. 
\begin{corollary}[Conjugation]
\label{cor:conjucagio}
Let $F:U\subset E_1 \to E_1$ be of the form $F(x) = x + p(x) + f(x)$ with $p,f$ satisfying the corresponding conditions in Theorems~\ref{ThmMapsPSBanach}-~\ref{ThmMapsASBanach} (resp. Theorems~\ref{ThmMapsPS}-~\ref{ThmMapsAS}). 

Then the map $F$ is conjugated to a map $R:V_\rho \to V_\rho$ of the form~\eqref{thm:formaKR}. In addition, if $h$ is the conjugation, $R,h$ are real analytic in a complex extension of $V_\rho$. 
\end{corollary}

Finally, the last result is related with vector fields. 
\begin{corollary}[Periodic orbits] \label{cor:vector fields} Consider the analytic vector field $X:U \times \mathbb{R}\subset E_1\times E_2 \times \mathbb{R} \to E_1\times E_2$ of the form
\begin{equation*}
X(x,y,t) = \begin{pmatrix} p(x,y) + f(x,y,t) \\ 
 q(x,y) + g(x,y,t)
 \end{pmatrix} ,\qquad X(x,y,t+1)=X(x,y,t).
\end{equation*}
Assume that there exist $V$ and $\rho_0$ such that $p,q,f,g$ satisfy the conditions under Theorems~\ref{ThmMapsPSBanach}-~\ref{ThmMapsASBanach} (resp. Theorems-~\ref{ThmMapsPS}-~\ref{ThmMapsAS}) uniformly in $t\in \mathbb{R}$. 

Then, if $\rho, \beta$ are small enough, there exist $K:V_\rho \times \mathbb{R} \to U, Y:V_\rho \to V_\rho$ real analytic functions defined in complex continuations of their domains such that 
\begin{equation}\label{invariance for flows}
X(K(x,t),t)-D_xK(x,t) Y(x) + \partial_t K(x,t)=0,
\end{equation}
$K(x,t+1)=K(x)$, 
and $K(V_{\rho} \times \mathbb{R}) \subset W^{\mathrm{s}}_{V_{\rho,\beta}}$
where 
\[
W^{\mathrm{s}}_{\mathcal{U}} = \left \{(x,y) \in E_1\times E_2 :  \Phi(t;t_0,x,y) \in \mathcal{U}, \forall t\ge0, \, \lim_{t\to +\infty} \Phi(t;t_0,x,y) =(0,0) \right \}
\]
and $\Phi(t;t_0,x,y)$ is the flow associated to the vector field $X$. 

In addition, when $B_q>0$, 
$K(\widehat V_{\rho} \times \mathbb{R}) =W^{\mathrm{s}}_{\widehat V_{\rho,\beta}}$ with $\widehat V_{\rho,\beta}$ the set defined in~\eqref{def Vrho beta} related to $\widehat{V} \subset V$ a star shaped subset with $0\in \partial \widehat V$.
\end{corollary}

\begin{remark} Corollary~\ref{cor:vector fields} follows from Theorems~\ref{ThmMapsPS} and~\ref{ThmMapsAS} considering the time $1-$map. The proofs is analogous to the one in the previous works~\cite{BFM1}, \cite{BFMARMA}.     

In fact, it is possible to prove that also for vector fields, the parameterization $K$ and the vector $Y$ can be written as sums of homogeneous functions. 
\end{remark}

\subsection{Perturbation of Toda lattices}
\label{sec:Todalattice}

Toda lattices is a completely integrable system defined in $\ell^{\infty}(\mathcal{X}) \times \ell^\infty (\mathcal{Y})$ with $\mathcal{X}=\{\mathcal{X}_i\}_{i\in \mathbb{Z}}$, $\mathcal{X}_i = \mathbb{R}^{2}$ and $\mathcal{Y}=\mathcal{X}$. It is a model for wave propagation along an infinite number of particles in a line.
It is described by the (formal) Hamiltonian
\[
H_0(\mathbf{q},\mathbf{p}) = \sum_{n\in \mathbb{Z}} \frac{1}{2} p_n^2+ V(q_{n+1}-q_n), \qquad V(z) = e^{-z} +z-1
\]
where $q_n$ are the positions and $p_n$, the momenta and $(\mathbf{q}, \mathbf{p}) = (q_n, p_n)_{n\in \mathbb{Z}}$. The equations of motion are 
\begin{equation}
\label{eq:Toda_em}
\begin{aligned}
    \dot{q}_n & = \partial_{p_n} H_0(\mathbf{q},\mathbf{p})=p_n,\\
    \dot{p}_n & =-\partial_{q_n} H_0(\mathbf{q},\mathbf{p})= e^{- (q_n-q_{n-1})} - e^{-(q_{n+1}-q_n)}.
    \end{aligned}
\end{equation}
Observe that $E = \{(\mathbf{q}, \mathbf{p}) = (q_n, p_n)_{n\in \mathbb{Z}}\mid \; q_n-q_{n-1} = +\infty,\; p_n = 0,\; n \in \mathbb{Z}\}$ is a set of equilibria of the equations. In what follows, we find perturbations of the Toda lattice preserving these equilibria and in such a way that $E$ possesses an stable invariant manifold. 

In Flaschka's variables, $(\mathbf{x}, \mathbf{y})= (x_n,y_n)_{n\in \mathbb{Z}}$, 
\begin{equation}\label{change variables Toda first}
x_n = -\frac{1}{2}p_n, \qquad y_n = \frac{1}{2} \mathrm{e}^{-\frac{1}{2} (q_{n+1}-q_n)},
\end{equation}
equations~\eqref{eq:Toda_em} become
\begin{equation}
\begin{pmatrix}
\label{eq:Toda_en_Flaschka}
    \dot{x}_n \\ \dot{y}_n 
\end{pmatrix} = \mathbf{T}_n(\mathbf{x},\mathbf{y}):= \begin{pmatrix}  2 (y_n^2- y_{n-1}^2) \\
     y_n (x_{n+1}-x_n)
    \end{pmatrix}.
\end{equation}
Observe that $(\mathbf{x}, \mathbf{y}) = (0,0)$ is a degenerate fixed point of~\eqref{eq:Toda_en_Flaschka}, corresponding to $q_n-q_{n-1} = +\infty$, $p_n = 0$, $n \in \mathbb{Z}$ in the original variables. 

We introduce the notation $\mathbf{T} = (\mathbf{T}_n)_{n\in \mathbb{Z}}$. 
Next result is an immediate consequence of Corollary~\ref{cor:vector fields}.

\begin{corollary} \label{cor: Toda 1}
Let $U \subset \mathcal{X}\times \mathcal{Y}$ be an open set, with $(0,0) \in U$.
Let $X:U \times \mathbb{R}\to \mathcal{X}\times \mathcal{Y}$, satisfying the hypotheses of Corollary~\ref{cor:vector fields} with $N=M=2$, for some $V\subset \mathcal{X}$, $\rho_*,\beta_*>0$. Consider the system
\[
\begin{pmatrix}
    \dot{\mathbf{x}} \\ \dot{\mathbf{y}}
\end{pmatrix} = \mathbf{T}(\mathbf{x},\mathbf{y})+ \varepsilon X(\mathbf{x},\mathbf{y},t),
\]
with $\varepsilon>0$. 

Then, if $V_{\rho_*}\subset \{ \mathbf{x} \in \mathcal{X} : x_{n+1}-x_n>0\}$ there exist $\beta_\varepsilon  \leq \beta_*$, $\rho_\varepsilon \leq \rho_*$, $K:V_{\rho_\varepsilon} \times \mathbb{R}\to U$ and $Y:V_{\rho_\varepsilon} \to V_{\rho_\varepsilon}$ satisfying the invariance condition~\eqref{invariance for flows} and $K(V_{\rho_\varepsilon} \times \mathbb{R}) \subset W^{\mathrm{s}}_{V_{\rho_\varepsilon,\beta_\varepsilon}}$. As in Corollary~\ref{cor:vector fields}, in the special case that $B_q>0$, $K(\widehat V_{\rho_\varepsilon} \times \mathbb{R}) = W^{\mathrm{s}}_{\widehat V_{\rho_\varepsilon,\beta_\varepsilon}}$ with $\widehat V \subset V$ an star shaped domain with $0\in \partial \widehat V$. 
\end{corollary}
\begin{proof}
    The proof follows straightforward from Corollary~\ref{cor:vector fields} and the fact that $\mathbf{T}(\mathbf{x},0)=(0,0)$. 
\end{proof}
 \begin{remark} As an example, the vector field  $X(\mathbf{x},\mathbf{y},t) = (-x_n^2 , 0)_{n\in \mathbb{Z}} + \widetilde{X}(\mathbf{x},\mathbf{y},t)$, with $\widetilde{X} \in \mathcal{H}^{\geq 3}$ satisfies the conditions of Corollary~\ref{cor: Toda 1}.


 \end{remark}

Another class of perturbations of the Toda lattice can be obtained as follows.
For $c,d, \ell \in \mathbb{R}^+$ such that $d = c \ell$, consider the modified Flaschka's variables
$(\mathbf{x}, \mathbf{y})=(x_n,y_n)_{n\in \mathbb{Z}}$, 
\[
x_n = d\mathrm{e}^{-\frac{1}{\ell} (q_{n+1}-q_n)}, \qquad y_n = -c p_n, 
\]
In these variables, equations~\eqref{eq:Toda_em} become
\begin{equation}\label{Toda modified}
\begin{pmatrix}
    \dot{x}_n \\ \dot{y}_n 
\end{pmatrix} = \mathbf{M}_n(\mathbf{x},\mathbf{y}):= \begin{pmatrix}  x_n (y_{n+1}-y_n) \\ 2 (x_n^\ell- x_{n-1}^\ell) 
    \end{pmatrix}.
\end{equation}
As before we write $\mathbf{M}= (\mathbf{M}_n)_{n\in \mathbb{Z}}$. 
\begin{corollary} 
\label{cor Toda 2}
Take $\ell \geq 3$ in~\eqref{Toda modified}. Let  $X:U \times \mathbb{R}\subset \mathcal{X}\times \mathcal{Y} \times \mathbb{R}\to \mathcal{X}\times \mathcal{Y}$, $U$ an open set, be under the hypotheses of Corollary~\ref{cor:vector fields} with $N=M=2$, for some $V\subset \mathcal{X}$, $\rho_*,\beta_*>0$ and consider the system
\[
\begin{pmatrix}
    \dot{\mathbf{x}} \\ \dot{\mathbf{y}}
\end{pmatrix} = \mathbf{M}(\mathbf{x},\mathbf{y})+ \varepsilon X(\mathbf{x},\mathbf{y},t)
\]
with $\varepsilon>0$. 

Then  there exist $\beta_\varepsilon  \leq \beta_*$, $\rho_\varepsilon \leq \rho_*$, $K:V_{\rho_\varepsilon} \times \mathbb{R}\to U$ and $Y:V_{\rho_\varepsilon} \to V_{\rho_\varepsilon}$ satisfying the invariance condition~\eqref{invariance for flows} and $K(V_{\rho_\varepsilon} \times \mathbb{R}) \subset W^{\mathrm{s}}_{V_{\rho_\varepsilon,\beta_\varepsilon}}$. As in Corollary~\ref{cor:vector fields}, in the special case that $B_q>0$, $K(\widehat V_{\rho_\varepsilon} \times \mathbb{R}) = W^{\mathrm{s}}_{\widehat V_{\rho_\varepsilon,\beta_\varepsilon}}$ with $\widehat V \subset V$ an star shaped domain with $0\in \partial \widehat V$.     
\end{corollary}


Couplings between different Toda  lattices can also be considered. Indeed, let  $\mathbf{T}^{k}$, $k\in \mathbb{N}$ a family of Toda lattices, namely
\[
\begin{pmatrix}
    \dot{\mathbf{x}}^k \\ \dot{\mathbf{y}}^k
\end{pmatrix} = \mathbf{T}^{k}(\mathbf{x}^k,\mathbf{y}^k) + \varepsilon^k  X^k (\mathbf{x},\mathbf{y},t), \qquad \mathbf{x} = (\mathbf{x}^k)_k, \, \mathbf{y} =(\mathbf{y}^k)_k. 
\]
Again $\mathbf{T}^k (\mathbf{x}^k,0)=0$ so that if $X = (X^k)_k$ satisfies the conditions of Corollary~\ref{cor:vector fields}, the existence of stable invariant manifold is guaranteed. 

%
%
%
%

\section{The cohomological equation}
\label{section_cohomological_equation}

This section is dedicated to solving the linearized problem associated with our dynamical problem. For this purpose, let $E_1 = \ell^\infty(\mathcal{X})$ and $E_2 = \ell^\infty(\mathcal{Y})$.  Let $V \subset E_1$ be an open set, star-shaped with respect to $0$.  Let $\mathbf{p} :V \to E_1$, $\mathbf{Q}:V \to \mathcal{L}(E_k, E_k)$ and $\mathbf{w}:V \to E_k$ for $k \in \{1,2\}$ be such that $\mathbf{p} \in \mathcal{H}^N$, $\mathbf{Q} \in \mathcal{H}^{N-1}$, $\mathbf{w} \in \mathcal{H}^{\mathfrak{m} +N}$ with $N \ge 2$ and $\mathfrak{m} \ge 1$. This section aims to find solutions $h:V \to E_k$, $k \in \{1,2\}$,  with $h \in \mathcal{H}^{\mathfrak{m}+1}$, of the  equation  
\begin{equation}
\label{CE}
Dh(x) \mathbf{p}(x) - \mathbf{Q}(x) h(x) = \mathbf{w}(x).
\end{equation}

Let $V_{\varrho_0}$ be as in~\eqref{Vrho}.  We assume the following conditions
\begin{itemize}
\item[HP1] $\mathbf{p} \in C^1(V_{\varrho_0})$ and
\begin{equation}\label{CH:def_a_p}
a_{\mathbf{p}} = -\sup_{x \in V_{\varrho_0}} {|x + \mathbf{p}(x)| - |x| \over |x|^N}>0. 
\end{equation}
\item[HP2] There exists a constant $a^{\mathbf{p}}_V>0$ such that, for all $x \in V_{\varrho_0}$,
\begin{equation*}
\mathrm{dist}(x + \mathbf{p}(x), (V_\varrho)^c) \ge a^{\mathbf{p}}_V|x|^N.
\end{equation*}
\end{itemize}
To quantify the weight of the homogeneous functions $\mathbf{p}$ and $\mathbf{Q}$, we introduce the following constants
\begin{align}
&b_\mathbf{p} = \sup_{x \in V_\varrho}{|\mathbf{p}(x)| \over |x|^N}, \hspace{38mm} A_\mathbf{p} = - \sup_{x \in V_\varrho}\displaystyle{|\mathrm{Id} + D\mathbf{p}(x)| - 1 \over |x|^{N-1}}, \nonumber\\
\label{ConstantsCE}
&B_\mathbf{Q} = -\sup_{x \in V_\varrho}{|\mathrm{Id} - \mathbf{Q}(x)| -1\over |x|^{N-1}}, \hspace{18mm} A_\mathbf{Q} = \sup_{x \in V_\varrho}{|\mathrm{Id} + \mathbf{Q}(x)| -1\over |x|^{N-1}},\\
&\hspace{35mm}c_\mathbf{p} = \begin{cases} a_\mathbf{p} , \quad \mbox{if $B_\mathbf{Q}  \le 0$},\\
b_\mathbf{p} ,\quad \mbox{otherwise.}
\end{cases} \nonumber
\end{align}
In addition, to control the decay properties of the homogeneous functions $\mathbf{p}$ and $\mathbf{Q}$, we define
\begin{equation}
\label{GammaConstantsCE}
A^\Gamma_\mathbf{p} = \sup_{x \in V_\varrho}\displaystyle{|D\mathbf{p}(x)|_\Gamma \over |x|^{N-1}}, \quad B^\Gamma_\mathbf{Q} = \sup_{x \in V_\varrho}\displaystyle{|\mathbf{Q}(x)|_\Gamma \over |x|^{N-1}}.
\end{equation}

We need to introduce the following two differential equations, which play a fundamental role in the solution of equation~\eqref{CE}.  For this purpose, we consider
\begin{align}
\label{flowp}
&{dx\over dt} = \mathbf{p}(x)\\
\label{flowQ}
&{d\psi \over dt}(t,x) = \mathbf{Q}\circ \varphi(t,x) \psi(t,x)
\end{align}
where we denote by $\varphi(t,x)$ the flow of~\eqref{flowp} and by $M(t,x)$ the fundamental matrix  of~\eqref{flowQ} such that $M(0,x) = \mathrm{Id}$. Moreover, using the uniqueness of solutions of~\eqref{flowp} and homogeneity, one can see that 
\begin{equation}
\label{varphiM}
\varphi(t, \lambda x) = \lambda \varphi(\lambda^{N-1}t, x), \quad M(t, \lambda x) = M(\lambda^{N-1}t, x)
\end{equation}
wherever they are defined.

\begin{theorem}\label{PCE}
Let $\mathbf{p} \in \mathcal{H}^N$, $\mathbf{Q} \in \mathcal{H}^{N-1}$, and $\mathbf{w}\in \mathcal{H}^{\mathfrak{m} + N}$ be defined on an open set $V$ star-shaped with respect to $0$, with $N \ge 2$ and $\mathfrak{m} \ge 1$. We assume that $\mathbf{p}$ and $\mathbf{Q}$ satisfy the hypotheses $HP1$ and $HP2$, for some $\varrho_0 >0$. Moreover, 
\begin{equation*}
   A_\mathbf{p} >b_\mathbf{p} 
\end{equation*}
and $\mathbf{p}$, $\mathbf{Q}$, $\mathbf{w}$ are real-analytic functions on $\Omega(\gamma_0)$ for some $\gamma_0 >0$.  Then, if
\begin{equation}
\label{CondCE}
\mathfrak{m} + 1 + {B_\mathbf{Q} \over c_\mathbf{p}} >  0
\end{equation}
there exists a unique solution $h \in \mathcal{H}^{\mathfrak{m} +1}$ of~\eqref{CE} real-analytic in $\Omega(\gamma)$ for $\gamma$ small enough, such that for all $x \in \Omega(\gamma)$
\begin{equation}
\label{solution}
h(x) = \int_\infty^0 M^{-1}(t,x) \mathbf{w}\circ \varphi(t,x) dt.
\end{equation}
Moreover, if we assume that
\begin{equation}
\label{Reg_pQw}
\mathbf{p} \in C_\Gamma^1(\Omega(\varrho_0,\gamma_0), \C E_1), \quad \mathbf{Q},  \mathbf{w} \in  C_\Gamma^1(\Omega(\varrho_0,\gamma_0), \C E_k) \quad \mbox{with $k=1,2$}
\end{equation}
and
\begin{equation*}
\mathfrak{m} - \left({B^\Gamma_\mathbf{Q}  + A_\mathbf{p}^\Gamma \over a_\mathbf{p}}\right) >  0,
\end{equation*}
then, for $\varrho$ and $\gamma$ sufficiently small we have that $h \in C_\Gamma^1(\Omega(\varrho, \gamma), \C E_k)$ with $k=1,2$.
\end{theorem}

\begin{remark}\label{rmk:ConstProp}
    We observe that, under the hypotheses of the above theorem, the constants $A_\mathbf{p}$, $B_\mathbf{Q}$, $a_\mathbf{p}$, $b_\mathbf{p}$, $A_\mathbf{Q}$, $A^\Gamma_{\mathbf{p}}$ and $B^\Gamma_{\mathbf{Q}}$ are finite. Moreover, they satisfy $|a_\mathbf{p}| \le b_\mathbf{p}$, $a_\mathbf{p} \ge {A_\mathbf{p} \over N}$, $B_\mathbf{Q} \le A_\mathbf{Q}$ and $- B_{D\mathbf{p}} \ge Na_\mathbf{p}>0$. We refer to Lemma 3.6 in~\cite {BFM2} for the proof.
    
\end{remark}

\begin{remark}
    We observe that, if $0 < \varrho_1 < \varrho_2$ then the corresponding constants $A^{1,2}_\mathbf{p}$, $a^{1,2}_\mathbf{p}$, $b^{1,2}_\mathbf{p}$,  $A^{1,2}_\mathbf{Q}$, $B^{1,2}_\mathbf{Q}$,  ${A^{\Gamma}}^{1,2}_\mathbf{p}$ and ${B^{\Gamma}}^{1,2}_\mathbf{Q}$, associated to $\varrho_1$ and $\varrho_2$, respectively, defined in~\eqref{CH:def_a_p},~\eqref{ConstantsCE} and~\eqref{GammaConstantsCE} satisfy $A^1_\mathbf{p} \ge A^2_\mathbf{p}$, $a^1_\mathbf{p} \ge a^2_\mathbf{p}$, $b^1_\mathbf{p} = b^2_\mathbf{p}$, $A^1_\mathbf{Q} \le A^2_\mathbf{Q}$, $B^1_\mathbf{Q} \ge B^2_\mathbf{Q}$, ${A^{\Gamma}}^1_\mathbf{p} = {A^{\Gamma}}^2_\mathbf{p}$, and ${B^{\Gamma}}^1_\mathbf{Q} = {B^{\Gamma}}^2_\mathbf{Q}$. In addition, if $HP1$ and $HP2$ are satisfied for $\varrho_2 >0$, they are also satisfied for all $0 < \varrho_1 < \varrho_2.$ As a consequence, we can always consider $\varrho$ small enough throught the proof of Theorem \ref{PCE}. 
    
    The proof of the above equalities and inequalities is similar to that of Lemma~$3.7$ in~\cite{BFM2}. We emphasize that the constants $A^\Gamma_\mathbf{p}$ and $B^\Gamma_\mathbf{Q}$ do not depend on $\rho$, which can be proved using the homogeneous character of $\mathbf{p}$ and $\mathbf{Q}$.
\end{remark}

The proof of Theorem~\ref{PCE} is the content of Section~\ref{PPCE}, whereas Section~\ref{PropPhiM} contains some estimates associated with the flows of~\eqref{flowp} and~\eqref{flowQ}.

\subsection{Properties of $\varphi(t,x)$ and $M(t,x)$.}\label{PropPhiM}
This section collects some properties of the solutions of equations~\eqref{flowp} and~\eqref{flowQ}. First, we consider the following lemma that provides a useful Taylor formula with control of the remainder and a sufficient condition for the complex set $\Omega(\varrho, \gamma)$ to be invariant under the map $x \to x+\mathbf{p}(x)$.
\begin{lemma}
\label{Lemma3.8}
Let $\varrho$, $\gamma >0$.
\begin{itemize}
\item[1.] If $x \in \Omega (\varrho, \gamma)$ and
$\varkappa : \Omega(\varrho, \gamma) \to  \C \ell^\infty(\mathcal{X})$ is a real-analytic function belonging to $\mathcal{H}^\ell$ then
\begin{equation*}
\varkappa(x) = \varkappa(\mathrm{Re}x) + i D_x\varkappa(\mathrm{Re}x) \mathrm{Im}x + \gamma^2 \mathcal{O}(|x|^\ell).
\end{equation*}
\item[2.] If $HP2$ is satisfied and $A_\mathbf{p} > b_\mathbf{p}$, then there exists $\gamma_0 \in (0,1)$ such that for any $0 < \gamma \le \gamma_0$, the complex set $\Omega(\varrho, \gamma)$ is an invariant set for the map $x \to x + \mathbf{p}(x)$.
\end{itemize}
\end{lemma}
\begin{proof}
The proof of $\textit{1.}$ follows from the Taylor formula with reminders~\eqref{TaylorFormulaRem} and the fact that an analytic function is locally bounded (see point~\textit{2} of Theorem~\ref{TheoremAnaiytic}). The proof of $\textit{2.}$ is similar to that of Lemma~$3.8$ in~\cite{BFM2}. It is a consequence of~$\textit{1.}$ and hypothesis~HP2. 
\end{proof}

Now, we denote by $C$ a generic positive constant, which may take different values at different places, and we define
\begin{equation*} 
\alpha = {1 \over N-1}. 
\end{equation*}

\begin{lemma}
\label{lemma3.9}
We assume hypotheses $HP1$ and $HP2$ for $\varrho_0 >0$,  $A_\mathbf{p} > b_\mathbf{p}$ and that $\mathbf{p}$ has an analytic extension to $\Omega(\gamma_0)$ for some $0 < \gamma_0 \le 1$. Then, there exist $0 < \varrho_1 \le \varrho_0$ and $0 < \gamma_1 \le \gamma_0$ such that for any $0 < \varrho \le \varrho_1$ and $0 < \gamma \le \gamma_1$,  the set $\Omega(\varrho, \gamma)$ is invariant be the complexified flow, i.e. $\varphi(t,x) \in \Omega(\varrho, \gamma)$, for $t\ge 0$ and $x \in \Omega(\varrho, \gamma)$.
\end{lemma}
\begin{proof}
We have that $\varphi(t,0) = 0$ for all $t$ and $\varphi$ is real-analytic. Then,  thanks to Taylor's formula~\eqref{TaylorFormulaRem}, Theorem \ref{TheoremAnaiytic}, for $\gamma \ge 0$ and $\varrho$ small enough
\begin{equation*}
|\varphi(t,x)| \le C |x|, \quad t \in [0,1],  \hspace{2mm}x \in \Omega(\varrho, \gamma).
\end{equation*}
In the latter, we use point \textit{2} of Theorem \ref{TheoremAnaiytic} because we need $\varphi$ to be locally bounded.  Now, following the lines of Lemma $3.9$ of~\cite{BFM2} and using~\eqref{TaylorFormulaRem},  Theorem \ref{TheoremAnaiytic},  the latter, $\mathbf{p} \in \mathcal{H}^N$, HP2 and Lemma \ref{Lemma3.8}, one can conclude the proof. 
\end{proof}

\begin{lemma}
\label{lemma3.10}
We assume that $HP1$ and $HP2$ are satisfied for some $\varrho_0 >0$, $A_\mathbf{p} > b_\mathbf{p}$ and $\mathbf{p}$ has an analytic extension to $\Omega(\gamma_0)$ for some $\gamma_0 \le 1$.  Then, for any $0 < a < a_\mathbf{p}$ and $b > b_\mathbf{p}$ there exists $\gamma \le \gamma_0$ such that for $t \ge 0$, $\varphi$ is analytic in $\Omega(\gamma)$ and 
\begin{equation*}
{|x| \over \left(1 + (N-1) bt|x|^{N-1}\right)^\alpha} \le |\varphi(t,x)| \le {|x| \over \left(1 + (N-1) at|x|^{N-1}\right)^\alpha}
\end{equation*}
for any $t \ge 0$ and $x \in \Omega(\gamma)$.
 \end{lemma}
 \begin{proof}
 We refer to Lemma~$3.10$ of~\cite{BFM2}. 
 \end{proof}
 
 \begin{lemma}
 \label{lemma3.11}
We assume that $HP1$ and $HP2$ are satisfied for some $\varrho_0 >0$,  $A_\mathbf{p} > b_\mathbf{p}$ and $\mathbf{p}$ and $\mathbf{Q}$ have an analytic extension to $\Omega(\gamma_0)$ for some $\gamma_0 \le 1$. Then for any  $0 < a < a_\mathbf{p}$, $b > b_\mathbf{p}$,  $A > A_{\mathbf{Q}}$ and $B< B_{\mathbf{Q}}$ there exists $\gamma \le \gamma_0$ such that for $t \ge 0$, $M(t,x)$ is analytic in $\Omega(\gamma)$ and 
 \begin{align*}
 &\left(1 + c(N-1) t |x|^{N-1}\right)^{\alpha{B \over c}} \le |M(t,x)| \le \left(1 + \delta(N-1) t |x|^{N-1}\right)^{\alpha{A \over \delta}}\\
 &\left(1 + \delta(N-1) t |x|^{N-1}\right)^{-\alpha{A \over \delta}} \le |M^{-1}(t,x)| \le \left(1 + c(N-1) t |x|^{N-1}\right)^{-\alpha{B \over c}}
 \end{align*}
 for $x \in \Omega(\gamma)$ with
 \begin{equation}
 \label{cdelta}
 c=\begin{cases}a, \quad \mbox{if $B \le 0$}\\
 b, \quad \mbox{otherwise}
 \end{cases} \hspace{10mm}  \delta=\begin{cases}a, \quad \mbox{if $A \ge 0$}\\
 b, \quad \mbox{otherwise}.
 \end{cases}
 \end{equation}
 \end{lemma}
  \begin{proof}
 We refer to Lemma $3.11$ of~\cite{BFM2}. 
 \end{proof}
 
\begin{lemma}
\label{lemma:StimaDPhi}
We assume that $HP1$ and $HP2$ hold for some $\varrho_0 > 0$, that $A_\mathbf{p} > b_\mathbf{p}$, and that $\mathbf{p}$ admits an analytic extension to $\Omega(\gamma_0)$ for some $\gamma_0 \le 1$ and satisfies~\eqref{Reg_pQw}. Then,  for any $0 < a \le a_\mathbf{p}$ and $A^\Gamma \ge A^\Gamma_{\mathbf{p}} $, there exists $\gamma < \gamma_0$ such that 
 \begin{equation}
 \label{StimaDPhiGamma}
 \left|D\varphi(t, x)\right|_\Gamma \le \left|\mathrm{Id}\right|_\Gamma \left(1 + (N-1)a t |x|^{N-1}\right)^{\alpha {A^\Gamma \over a}}
 \end{equation}
 for all $t \ge 0$ and $x \in \Omega(\gamma)$.
 \end{lemma}
 \begin{proof}
By Lemma~\ref{Lemma3.8}, there exist $\rho>0$ and $\gamma>0$ such that, for $\gamma$ sufficiently small, the set $\Omega(\rho,\gamma)$ is invariant under the flow $\varphi$.
Hence the following system is well-defined for all $x \in \Omega(\rho, \gamma)$ 
 \begin{equation*}
 \begin{cases}
&{d \over dt} \psi (t, x) = D\mathbf{p}\circ \varphi(t,x) \psi(t, x)\\
&\psi(0, x) = \mathrm{Id}
 \end{cases}
 \end{equation*}
and one can see that $D\varphi$ is the solution of the latter.

 This means that, using the fundamental theorem of calculus, we can rewrite $D\varphi$ as
 \begin{equation*}
 D\varphi(t, x) = \mathrm{Id} + \int_0^t D\mathbf{p}\circ \varphi(s,x) D\varphi(s,x) ds
 \end{equation*}
 and, thanks to Propositon \ref{prop:NormGammaProd}, we obtain the following estimate
 \begin{equation}
 \label{ProofDPhiPrimaStima}
 |D\varphi(t, x)|_\Gamma \le |\mathrm{Id}|_\Gamma + \int_0^t |D\mathbf{p}\circ \varphi(s,x)|_\Gamma |D\varphi(s,x)|_\Gamma ds
 \end{equation}
for all $t \ge 0$ and $x \in \Omega(\rho, \gamma)$.  In what follows, we provide a suitable estimate for the norm $ |D\mathbf{p}\circ \varphi(s,x)|_\Gamma$, and use the Gronwall inequality to prove~\eqref{StimaDPhiGamma} for all $t \ge 0$ and $x \in \Omega(\rho, \gamma)$.  Then, we extend this estimate to all $x \in \Omega(\gamma)$ by exploiting the homogeneous character of $D\varphi$.

For this purpose,  after a Taylor extension,  we can write $D\mathbf{p}\circ \varphi(s,x)$ as
\begin{equation*}
D\mathbf{p}\circ \varphi(s,x) = D\mathbf{p}\circ \mathrm{Re}\,\varphi(s,x) + i \int_0^1 D^2 \mathbf{p} \circ \left(\mathrm{Re} \,\varphi(s,x) + \tau i \mathrm{Im}\,\varphi(s,x)\right)d \tau \mathrm{Im}\, \varphi(s,x)
\end{equation*}
for all $x \in \Omega(\varrho, \gamma)$. Using~\eqref{GammaConstantsCE},  hypothesis~\eqref{Reg_pQw}, Proposition \ref{cor:cauchyBS_Decay}, and the definition~\eqref{def:OmegaSet}, we can estimate the left-hand side of the latter as
 \begin{align*}
 |D\mathbf{p}\circ \varphi(s,x) |_\Gamma &\le  |D\mathbf{p}\circ \mathrm{Re}\varphi(s,x) |_\Gamma + C
 |x|\gamma |\varphi(s, x)|^{N-1} \le \left(A_\mathbf{p}^\Gamma + C\gamma\right) |\varphi(s, x)|^{N-1} 
 \end{align*}
 for all $x \in \Omega(\varrho, \gamma)$, where we recall that  $C$ is a generic positive constant which may take different values at different places. We define
 \begin{equation}\label{DefAGamma}
 A^\Gamma = A_\mathbf{p}^\Gamma + C\gamma
 \end{equation}
 and thanks to~\eqref{ProofDPhiPrimaStima}, the above estimate, and the Gronwall inequality, we obtain that 
 \begin{equation}
 \label{ProofDPhiQuasiUltimaStima}
  |D\varphi(t, x)|_\Gamma \le |\mathrm{Id}|_\Gamma \mathrm{exp}\left(A^\Gamma\int_0^t |\varphi(s, x)|^{N-1} ds\right).
 \end{equation}
Observing that $A^\Gamma \ge 0$, by Lemma \ref{lemma3.10}, one can verify the following inequality
\begin{equation*}
A^\Gamma\int_0^t |\varphi(s, x)|^{N-1} ds \le {A^\Gamma \over (N-1)a} \ln (1 + (N-1)at|x|^{N-1}).
\end{equation*}
Now, replacing the latter in~\eqref{ProofDPhiQuasiUltimaStima}, one can prove~\eqref{StimaDPhiGamma} for all $t \ge 0$ and $x \in \Omega(\rho, \gamma)$.  Furthermore, using the identity~\eqref{varphiM}, one can show that $D\varphi(t, x) = D\varphi(\lambda^{1-N}t, \lambda x)$ for all $\lambda \in \R$ and use this equality to extend the inequality~\eqref{StimaDPhiGamma} to all $x \in \Omega(\gamma)$.
 \end{proof}
 
 \begin{lemma}
 \label{LemmaStimaM-1Gamma}
 We assume that $HP1$ and $HP2$ are satisfied for some $\varrho_0 >0$,  that $A_\mathbf{p}>b_\mathbf{p}$,  and that $\mathbf{p}$ and $\mathbf{Q}$ admit an analytic extension to $\Omega(\gamma_0)$ for some $\gamma_0 \le 1$ and satisfy~\eqref{Reg_pQw}.  Then,  for any $0 < a \le a_\mathbf{p}$ and $B^\Gamma \ge B^\Gamma_{\mathbf{Q}} $, there exists $\gamma \le \gamma_0$ such that 
 \begin{equation}
 \label{StimaM-1Gamma}
 \left|M^{-1}(t, x)\right|_\Gamma \le \left|\mathrm{Id}\right|_\Gamma \left(1 + (N-1)a t |x|^{N-1}\right)^{\alpha {B^\Gamma \over a}}
 \end{equation}
 for all $t \ge 0$ and $x \in \Omega(\gamma)$.
 \end{lemma}
 \begin{proof}
 Thanks to Lemma \ref{Lemma3.8}, there exists $\rho>0$ and $\gamma>0$ such that the set $\Omega(\rho, \gamma)$ is invariant by the flow $\varphi$ if $\gamma$ is suitably small.  Let $x \in \Omega(\rho, \gamma)$ one can see that $M^{-1}$ is the solution of the system
 \begin{equation*}
 \begin{cases}
&{d \over dt} \psi (t, x) = - \psi(t, x) \mathbf{Q}\circ \varphi(t,x)\\
&\psi(0, x) = \mathrm{Id}.
 \end{cases}
 \end{equation*}
 From now on, the proof is similar to that of Lemma \ref{lemma:StimaDPhi}. Thanks to the fundamental theorem of calculus and Proposition \ref{prop:NormGammaProd}, one can verify that 
 \begin{equation*}
 |M^{-1}(t, x)|_\Gamma \le |\mathrm{Id}|_\Gamma + \int_0^t |\mathbf{Q}\circ \varphi(s,x)|_\Gamma |M^{-1}(s,x)|_\Gamma ds
 \end{equation*}
 and using a Taylor extension,~\eqref{GammaConstantsCE},  hypothesis~\eqref{Reg_pQw},  the definition~\eqref{def:OmegaSet}, and the Gronwall inequality,  one can prove~\eqref{StimaM-1Gamma} for all $t \ge 0$ and $x \in \Omega(\rho, \gamma)$ where
 \begin{equation}\label{DefBGamma}
 B^\Gamma = B_\mathbf{Q}^\Gamma + C\gamma.
 \end{equation}
 Furthermore, noting that $M^{-1}(t, x) = M^{-1}(\lambda^{1-N}t, \lambda x)$ for all $\lambda \in \R$ one can extend the inequality~\eqref{StimaM-1Gamma} to all $x \in \Omega(\gamma)$.
\end{proof}

 
 \subsection{Proof of Theorem \ref{PCE}}\label{PPCE}

 The proof of Theorem \ref{PCE} is divided into two parts. First, we prove that, under suitable assumptions, a unique analytic solution of~\eqref{CE} exists and has the form given by~\eqref{solution}. In the second part, we verify that if certain conditions are satisfied, the solution of the cohomological equation exhibits the required decay properties. To this end, we recall that $E_1 = \ell^\infty(\mathcal{X})$ and $E_2 = \ell^\infty(\mathcal{Y})$, and we have the following

 \begin{lemma}
\label{ProofCE:Analytic_solution}
Let $\mathbf{p} \in \mathcal{H}^N$ be defined on $V$ and satisfying hypotheses $HP1$ and $HP2$ for some $\varrho_0>0$. We consider $\mathbf{Q} : V \to \mathcal{L}(E_k, E_k)$ and $\mathbf{w} : V \to E_k$ such that $\mathbf{Q} \in \mathcal{H}^{N-1}$ and $\mathbf{w} \in \mathcal{H}^{\mathfrak{m} + N}$ with $\mathfrak{m} \ge 1$ and $k=1,2$. In addition, we assume that $A_\mathbf{p} > b_\mathbf{p}$ and that $\mathbf{p}$, $\mathbf{Q}$ and $\mathbf{w}$ have analytic extensions to $\Omega(\gamma_0)$ for some $\gamma_0 \le 1$. If $\mathfrak{m}  + 1+ {B_\mathbf{Q} \over c_\mathbf{p}}>0$, with $c_\mathbf{p}$ and $B_\mathbf{Q} $ defined in~\eqref{ConstantsCE}, then there exists $0 <\gamma\le\gamma_0$ small enough  and a unique solution $h$ of~\eqref{CE} of the form~\eqref{solution}, namely 
\begin{equation*}
h: V \to E_k, \quad h(x) = \int_\infty^0M^{-1}(t,x)\mathbf{w} \circ \varphi(t,x) dt
\end{equation*}
such that $h \in \mathcal{H}^{\mathfrak{m} + 1}$ and $h$ is real-analytic in $\Omega(\gamma)$.
\end{lemma}
\begin{proof}
We consider $0 < a < a_\mathbf{p}$, $b >b_\mathbf{p}$ and $B<B_\mathbf{Q}$ in such a way that 
\begin{equation}
\label{Proof:CE_cond_proof}
\mathfrak{m}  + 1+ {B \over c}>0
\end{equation}
with $c$ defined in~\eqref{cdelta}. We fix $\varrho$ and $\gamma$ in such a way that $\Omega(\varrho, \gamma)$ is invariant by the flow of $\varphi$ if $\varrho$ and $\gamma$ are suitably small.

The first part of the proof is contained in~\cite{BFM2}. For this reason, it is omitted. It remains to prove the regularity of $h$.  We want to verify that $h$ is weakly analytic in $\Omega(\varrho, \gamma)$ (see Definition \ref{weakanalytic}) and locally bounded.  Then, thanks to Theorem \ref{TheoremAnaiytic} we can conclude that $h$ is analytic. To this end, we introduce the following norm
\begin{equation*}
\|\mathbf{w}\| = \sup_{x \in \Omega(\varrho, \gamma)}{|\mathbf{w}(x)| \over |x|^{\mathfrak{m} + N}}.
\end{equation*}
Let $x_0 \in \Omega(\varrho, \gamma)$, we choose $r>0$ in such a way that the open ball $B_r(x_0)\subset \Omega(\varrho, \gamma)$.  The existence of $r$ is justified because $\Omega(\varrho, \gamma)$ is open.  For all $x \in B_r(x_0)$ we consider $g \in \Omega(\varrho, \gamma)$ and $z \in \C$ in such a way that $x = x_0 + zg$. Then, for all $L \in \left(\C E_k\right)^*$  
\begin{align*}
&\left|L\int_0^\infty M^{-1}(x_0 + z h, t)\mathbf{w} \circ \varphi(x_0 + zh, t)dt\right| \\
&\hspace{20mm}\le |L|\int_0^\infty| M^{-1}(x_0 + z h, t)\mathbf{w} \circ \varphi(x_0 + zh, t)dt| \\
&\hspace{20mm}\le |L|\int_0^\infty\| \mathbf{w}\|{|x|^{\mathfrak{m} +N} \over (1 + a(N-1)t|x|^{N-1})^\nu}dt\\
&\hspace{20mm}\le |L|\int_0^\infty\| \mathbf{w}\|{\left(|x_0| + r_0\right)^{\mathfrak{m} +N} \over (1 + a(N-1)t\left(|x_0| - r_0\right)^{N-1})^\nu}dt\\
\end{align*}
where~\eqref{Proof:CE_cond_proof} implies that  $\nu = \alpha\left(\mathfrak{m} + N + {B \over c}\right)>1$.  In the first inequality $|L|$ stands for the operator norm. The second inequality is a consequence of Lemmas \ref{lemma3.10} and \ref{lemma3.11}.  In the last inequality, we use that $x \in B_r(x_0)$ implies $\|x_0\| - r \le \|x\| \le \|x_0\| +r$. By the dominated convergence theorem, $Lg(x_0 + zh)$ is analytic at $z=0$, and hence $h$ is weakly analytic at $x_0 \in \Omega(\varrho, \gamma)$. 
Repeating the argument for every $x_0 \in \Omega(\varrho, \gamma)$, we deduce that $h$ is weakly analytic on $\Omega(\varrho, \gamma)$. 

Using similar arguments, one can show that $h$ is locally bounded, and hence, by Theorem~\ref{TheoremAnaiytic}, it follows that $h$ is analytic on $\Omega(\varrho, \gamma)$. Finally, the fact that $M^{-1}$, $\mathbf{w}$, and $\varphi$ are real-analytic implies that $h$ is real-analytic as well.

In addition, since $h$ is homogeneous, we can uniquely extend it to an analytic homogeneous function on $\Omega(\gamma)$. 

\end{proof}

In the second part of this section, we prove the second part of Theorem \ref{PCE}.

\begin{lemma}
We assume that the hypotheses of Lemma \ref{ProofCE:Analytic_solution} are satisfied. In addition, let $\gamma_0 \le 1$ be such that 
\begin{equation}
\label{Reg_pQw_2}
\mathbf{p} \in C_\Gamma^1(\Omega(\varrho_0,\gamma_0), \C E_1), \quad \mathbf{Q},  \mathbf{w} \in  C_\Gamma^1(\Omega(\varrho_0,\gamma_0), \C E_k) \quad \mbox{with $k=1,2$}
\end{equation}
and
\begin{equation}
\label{GammaCondCE}
\mathfrak{m} - \left({B^\Gamma_\mathbf{Q}  + A_\mathbf{p}^\Gamma \over a_\mathbf{p}}\right) >  0.
\end{equation} 
Then, for $ \varrho < \varrho_0$ and $\gamma< \gamma_0$ sufficiently small, the unique solution $h$ of~\eqref{CE} satisfies $h \in C_\Gamma^1(\Omega(\varrho, \gamma), \C E_k)$ with $k=1,2$. 
\end{lemma}


\begin{proof}
Using similar arguments as the ones in the proof of Lemma \ref{ProofCE:Analytic_solution}, one can prove the existence of $\varrho$ and $\gamma$ suitably small in such a way that 
\begin{equation*}
     \sup_{x \in \Omega(\varrho, \gamma)}|h(x)| < \infty.
\end{equation*}
Remembering the defintion of the space $C^1_\Gamma$ and the associated norm $|\cdot|_{C^1_\Gamma}$ (see~\eqref{C1_GammaSpace} and~\eqref{C1_GammaNorm}, respectively), it remains to verify that 
\begin{equation*}
    \sup_{x \in \Omega(\varrho, \gamma)}|Dh(x)|_\Gamma<\infty
\end{equation*}
for $\varrho$ and $\gamma$ small enough. To this end, we consider $0 < a < a_\mathbf{p}$, $B^\Gamma > B^\Gamma_{\mathbf{Q}}$ and $A^\Gamma > A^\Gamma_\mathbf{p}$ in such a way that 
\begin{equation}
\label{GammaCondCE2}
\mathfrak{m} - \left({B^\Gamma  + A^\Gamma \over a}\right) >  0.
\end{equation}
Furthermore,  we fix $\rho$ and $\gamma$ in such a way that $\Omega(\rho, \gamma)$ is invariant under the flow of $\varphi$ if $\gamma$ is suitably small. 

We observe that the differential of $h$ equals
\begin{align*}
Dh(x) &= \int_0^\infty DM^{-1}(t,x) \mathbf{w}\circ \varphi(t,x)dt\nonumber\\
&+ \int_0^\infty M^{-1}(t,x)D\mathbf{w} \circ \varphi(t,x)D\varphi(t,x)dt
\end{align*}
and using Proposition \ref{prop:NormGammaProd} and Proposition \ref{prop:NormGammaProdBil}, we obtain that 
\begin{align}
\label{proof:Def_Dh}
|Dh(x)|_\Gamma &\le \int_0^\infty |DM^{-1}(t,x)|_\Gamma |\mathbf{w}\circ \varphi(t,x)|dt\nonumber\\
&+ \int_0^\infty |M^{-1}(t,x)|_\Gamma |D\mathbf{w} \circ \varphi(t,x)|_\Gamma |D\varphi(t,x)|_\Gamma dt
\end{align}
for all $x \in \Omega(\rho, \gamma)$. We need to estimate the right-hand side of the latter.  The proof is divided into two steps. First, we provide an upper bound of the norm $|DM^{-1}(t,x)|_\Gamma$. Then, we estimate the integrals on the right-hand side of the latter.  

We observe that $M$ satisfies the following properties
\begin{equation*}
\begin{aligned}
&{d \over dt} \left(M^{-1}(t,x) DM(t,x)\right) = M^{-1}(t,x) D\left(\mathbf{Q}\circ\varphi(t,x)\right)M(t,x)\\
&M(s+t, x) = M(s, \varphi(t,x))M(t,x)\\
&M^{-1}(t,x)DM(t,x) = -DM^{-1}(t,x) M(t,x).
\end{aligned}
\end{equation*}
Using the above properties and $DM^{-1}(0,x) = 0$, one can verify that 
\begin{equation*}
DM^{-1}(t, x) = - \int_0^t M^{-1}(s, x)D\left(\mathbf{Q}\circ \varphi\right)(s,x) M^{-1}(t - s, \varphi(s,x))ds
\end{equation*}
for all $t \ge 0$ and $x \in \Omega(\rho, \gamma)$. Thanks to Proposition \ref{prop:NormGammaProd}
\begin{align}
\label{FirstStimaDMGamma}
&|DM^{-1}(t, x)|_\Gamma \nonumber\\
&\hspace{5mm}\le \int_0^t |M^{-1}(s, x)|_\Gamma |D\left(\mathbf{Q}\circ \varphi\right)(s,x)|_\Gamma |M^{-1}(t - s, \varphi(s,x))|_\Gamma ds
\end{align}
for all $t \ge 0$ and $x \in \Omega(\rho, \gamma)$. We want to find an upper bound for each term contained in the integral in the last line of the above inequality.  First, using Lemma \ref{lemma3.10} and Lemma \ref{LemmaStimaM-1Gamma}, we observe that 
\begin{align}
\label{StimaM-1M-1Gamma}
&|M^{-1}(s, x)|_\Gamma |M^{-1}(t - s, \varphi(s,x))|_\Gamma \nonumber\\
&\hspace{5mm}\le |M^{-1}(s, x)|_\Gamma |\mathrm{Id}|_\Gamma \left(1 + (N-1)a (t-s)|\varphi(s,x)|^{N-1}\right)^{\alpha {B^\Gamma\over a}}\nonumber \\
&\hspace{5mm}\le C \left(1 + (N-1)a t|x|^{N-1}\right)^{\alpha {B^\Gamma\over a}}
\end{align}
for all $t \ge 0$ and $x \in \Omega(\rho, \gamma)$. On the other hand, we have that 
\begin{equation*}
D\left(\mathbf{Q}\circ \varphi\right)(s,x) = D\mathbf{Q}\circ \varphi(s,x) D\varphi(s,x)
\end{equation*}
and using~\eqref{GammaCondCE}, the homogeneous character of $\mathbf{Q}$,  Proposition \ref{prop:NormGammaProd}, Lemma \ref{lemma3.10} and Lemma \ref{lemma:StimaDPhi}
\begin{align}
\label{StimeD(QPhi)Gamma}
|D\left(\mathbf{Q}\circ \varphi\right)(s,x)|_\Gamma &\le |D\mathbf{Q}\circ \varphi(s,x)|_\Gamma |D\varphi(s,x)|_\Gamma \nonumber\\
&\le C \left(\sup_{x \in\Omega(\varrho_0, \gamma_0)}|D\mathbf{Q}(x)|_\Gamma\right) |\varphi(s,x)|^{N-2} |D\varphi(s,x)|_\Gamma \nonumber\\
& \le C {|\mathrm{Id}|_\Gamma |x|^{N-2} \over (1 + (N-1)a s |x|^{N-1})^{1 - \alpha}} (1 + (N-1)a s |x|^{N-1})^{\alpha{A^\Gamma \over a}}\nonumber \\
&\le C {|x|^{N-2} \over (1 + (N-1)a s |x|^{N-1})^{1 - \alpha \left(1 + {A^\Gamma \over a}\right)}}.
\end{align}
Now, replacing the estimates~\eqref{StimaM-1M-1Gamma} and~\eqref{StimeD(QPhi)Gamma} into~\eqref{FirstStimaDMGamma} and observing that $ {A^\Gamma \over a} \ne -1$, we obtain that 
\begin{align}
\label{EstimateDM-1Final}
|DM^{-1}(t, x)|_\Gamma &\le C (1 + (N-1)a t |x|^{N-1})^{\alpha {B^\Gamma \over a}} \int_0^t{|x|^{N-2} \over (1 + (N-1)a s |x|^{N-1})^{1 - \alpha \left(1 + {A^\Gamma \over a}\right)}}\nonumber\\
&\le C |x|^{-1}(1 + (N-1)a t |x|^{N-1})^{\alpha \left(1 + {B^\Gamma \over a} + {A^\Gamma \over a}\right)}
\end{align}
where the last line of the latter is obtained after the computation of the integral and a trivial estimate.  This concludes the first part of the proof. Now, we want to estimate the two integrals in the right-hand side of~\eqref{proof:Def_Dh}.  For this purpose,  using Lemma \ref{lemma3.10},~\eqref{EstimateDM-1Final},  hypothesis~\eqref{Reg_pQw}, the homogeneous character of $\mathbf{w}$ and~\eqref{GammaCondCE2}
\begin{align}
\label{StimehC1I1}
\int_0^\infty |DM^{-1}(t,x)|_\Gamma |\mathbf{w}\circ \varphi(t,x)|dt &\le C \int_0^\infty {|x|^{\mathfrak{m} + N -1} \over \left(1 + (N-1)a t |x|^{N-1}\right)^\nu}dt\nonumber\\
&\le C |x|^{\mathfrak{m}}
\end{align}
with $\nu = \alpha \left(\mathfrak{m} + N -\left(1 + {B^\Gamma \over a} + {A^\Gamma \over a}\right)\right)$. Similarly, in this case, the last line of the latter is derived by computing the integral and then applying a trivial estimation. Similarly to the previous case,  thanks to Lemma \ref{lemma3.10}, Lemma \ref{LemmaStimaM-1Gamma}, Lemma \ref{lemma:StimaDPhi},  hypothesis~\eqref{Reg_pQw}, the homogeneous character of $\mathbf{w}$ and~\eqref{GammaCondCE2}, one can verify that 
\begin{equation}
\label{StimehC1I2}
 \int_0^\infty |M^{-1}(t,x)|_\Gamma |D\mathbf{w} \circ \varphi(t,x)|_\Gamma |D\varphi(t,x)|_\Gamma dt \le C |x|^{\mathfrak{m}}.
\end{equation}
Replacing~\eqref{StimehC1I1} and~\eqref{StimehC1I2} into~\eqref{proof:Def_Dh} and using the homogenous character of $Dh$ (we recall that $Dh \in \mathcal{H}^{\mathfrak{m}+1}$) one can conclude the proof of this lemma.
\end{proof}
This concludes the proof of Theorem \ref{PCE}.

\section{Proof of Theorem \ref{ThmMapsAS}}\label{Proof_ThmMapsAS}
The proof follows closely the approach in~\cite{BFM2}. The only modifications are those required to establish the decay properties of the invariant manifolds.

First, we observe that, by Lemma~\ref{Lemma3.8} and $R(x) - (x + p(x,0)) \in \mathcal{H}^{\ge N+1}$, we have that $R(V_\varrho) \subset V_\varrho$. Hence, if  $K$ is defined $V_\varrho$, so is $K\circ R$. Moreover, if $K(x) - (x,0) \in \mathcal{H}^{\ge 2}$ and $x \in V_\varrho$, then $K(x) \in U$ and hence $F \circ K$ is well-defined as well.

Following the notation introduced in~\cite{BFM1,BFM2}, for $h$ such that its projections have different orders, we will write $h \in \mathcal{H}^{\ge l_1} \times \mathcal{H}^{\ge l_2}$ if $h_x \in \mathcal{H}^{\ge l_1}$ and $h_y \in \mathcal{H}^{\ge l_2}$. We will use the same notation for the spaces $\mathcal{H}^{> l}$ and $\mathcal{H}^{l}$.

\subsection{The inductive procedure}\label{sec:Inductive_procedure}

Let $N \le \ell$ and $j \in \N$ such that $1 \le j \le \ell - N + 1$. We will prove by induction over $j$ the existence of $K^{\le j}$ and $R^{\le j + N -1}$ of the form
\begin{equation}\label{proof:FormKandR}
K^{\le j}(x) = \sum_{l = 1}^j K^l(x), \qquad R^{\le j + N -1}(x) = x + \sum_{l = N}^{j+N-1} R^l (x),
\end{equation}
with $K^1(x) = (x,0)^T$, $R^N(x) = p(x,0)$ and 
\begin{equation}\label{proofThmApprox:RegKR}
\begin{aligned}
    &K^l \in \mathcal{H}^l, \quad R^{l+N-1} \in \mathcal{H}^{l+N-1}, \\
    &K_x^l,  R^{l+N-1} \in C_\Gamma^1(\Omega(\varrho, \gamma), \C \ell^\infty(\mathcal{X})), \quad K_y^l \in C_\Gamma^1(\Omega(\varrho, \gamma), \C \ell^\infty(\mathcal{Y})),
    \end{aligned}
\end{equation}
for all $1 \le l \le j$ and suitably small $0 < \varrho < \varrho_0$ and $0< \gamma < \gamma_0$, satisfying
\begin{equation}
\label{proofThmApprox:iterative_eq}
    E^{>j} := F \circ K^{\le j} - K^{\le j}\circ R^{\le j +N-1} = \left(E_x^{>j}, E_y^{>j}\right) \in \mathcal{H}^{>j+N-1}\times \mathcal{H}^{>j+M-1}.
\end{equation}
We recall that the map $F$ is defined in~\eqref{F}.
Starting from this point, and for the sake of notational simplicity, we will omit the codomains $\ell^\infty(\mathcal{X})$ and $\ell^\infty(\mathcal{Y})$ in the notation $C^1_\Gamma$, while retaining the domain notation. The target spaces will be clear from context.

We recall that we want to verify~\eqref{thm:approx_sol}. If $N=M$, it is a consequence of~\eqref{proofThmApprox:iterative_eq} with $j=\ell-N+1$. Otherwise, if $M<N$, we have to perform an extra induction procedure for values of $j$ such that $\ell - N +2 \le j \le \ell - M +1$ (see Section \ref{sec:extra_ind}). 
We recall that we are only considering the case $M \leq N$, since the case $M > N$ can be reduced to $M = N$ (see Remark~\ref{rmk:MleN}).

Let $j=1$. Taking $K^{\le 1} = (x,0)^T$ and $R^{\le N}(x) = x + p(x,0)$ one has that 
\begin{align*}
    &E_x^{>1} (x) = x + p(x,0) + f(x,0) - R^{\le N}(x) = f(x,0) \in \mathcal{H}^{\ge N+1} \subset \mathcal{H}^{>N},\\
    &E_y^{>1} (x) = q(x,0) + g(x,0) = g(x,0) \in \mathcal{H}^{\ge M+1} \subset \mathcal{H}^{>L},
\end{align*}
where, in the last line of the latter, we have used hypothesis H2. Thanks to hypothesis~\eqref{thm:approx_sol_Hyp_decay}, we have that
\begin{equation*}
    K^1, R^{\le N} \in C_\Gamma^1 \left(\Omega(\varrho_1, \gamma_1)\right),
\end{equation*}
with $\varrho_1 = \varrho_0$ and $\gamma_1 =\gamma_0$. Thus, condition~\eqref{proofThmApprox:RegKR} is verified for $K^1$ and $R^{N}$. Now, suppose that~\eqref{proofThmApprox:iterative_eq} is verified for $j -1 \ge 1$, $K^{\le j-1}$ and $R^{j+N-2}$. Moreover, we assume the existence of $\varrho_{j-1}$ and $\gamma_{j-1}$ with
\begin{equation*}
    0< \varrho_{j-1} \le \varrho_1 \quad \mbox{and} \quad 0 < \gamma_{j-1} \le \gamma_1
\end{equation*}
such that 
\begin{equation}\label{proof:Ind_hyp_ass_Gammareg}
    K^l \in \mathcal{H}^l, \quad R^{l+N-1} \in \mathcal{H}^{l+N-1}, \quad 
    K^l, R^{\le l + N -1} \in C_\Gamma^1 \left(\Omega(\varrho_{j-1}, \gamma_{j-1})\right), \quad
\end{equation}
for all $1 \le l \le j-1$. We want to look for the condition that $K^j$ and $R^{j+N-1}$ have to satisfy in order to prove~\eqref{proofThmApprox:RegKR} and~\eqref{proofThmApprox:iterative_eq} for $j$, $K^{\le j} = K^{\le j-1} + K^j$ and $R^{\le j+N-1} = R^{\le j+N-2} + R^{j+N-1}$. For this purpose, at each step of the proof, we will choose $\gamma_j \le \gamma_{j-1}$ and $\varrho_j \le \varrho_{j-1}$. To avoid introducing additional notation, we will always refer to these (possibly smaller) values as $\gamma_j$ and $\varrho_j$.

Expanding equation~\eqref{proofThmApprox:iterative_eq} in a neighborhood of $(x,0)$, and using the fact that $j-1+N \le \ell \le r$, together with the inductive hypothesis and the homogeneous character of $F$, $K^{\le j-1}$ and $R^{j + N - 1}$, one can prove the existence of $E_x^{j + N -1} \in \mathcal{H}^{j + N -1}$ and $E_y^{j + M -1} \in \mathcal{H}^{j + M -1}$ such that
\begin{equation}\label{proofThmApprox:iterative_eq_2}
\begin{aligned}
    &E_x^{>j-1} = F_x \circ K^{\le j-1} - K_x^{\le j-1}\circ R^{\le j+N-2} = E_x^{j+N-1} + \hat E_x^{>j}\\
    &E_y^{>j-1} = F_y \circ K^{\le j-1} - K_y^{\le j-1}\circ R^{\le j+N-2} = E_y^{j+M-1} + \hat E_y^{>j}
\end{aligned}
\end{equation}
with $\hat E_x^{>j} \in \mathcal{H}^{>j + N -1}$ and $\hat E_y^{>j} \in \mathcal{H}^{>j + M -1}$.

Moreover, using~\eqref{proof:Ind_hyp_ass_Gammareg}, hypothesis~\eqref{thm:approx_sol_Hyp_decay}, and the Cauchy inequalities  established in Proposition \ref{prop:cauchyBS_Decay} and Proposition \ref{cor:cauchyBS_Decay}, it immediately follows the existence of $0<\gamma_j\le \gamma_{j-1}$ such that 
\begin{equation}\label{proof:Reg_ExEy}
    E_x^{j + N -1}, E_y^{j + M -1} \in C^1_\Gamma \left(\Omega\left(\varrho_{j-1}, \gamma_j\right)\right).
\end{equation}

Starting from~\eqref{proofThmApprox:iterative_eq_2}, in order to prove~\eqref{proofThmApprox:RegKR} and~\eqref{proofThmApprox:iterative_eq} for $j$, we have to choose $K^j$ and $R^{j+N-1}$ in such a way that 
\begin{equation}\label{proof:eq_for_Kx}
    \hspace{-1.5mm}DK_x^j(x) p(x,0) -  D_xp(x,0) K^j_x(x) - D_yp(x,0)K_y^j(x) + R^{j +N -1}(x) = E_x^{j + N -1}(x)
\end{equation}
and, noting that $M$ and $N$ may de different
\begin{equation}\label{proof:eq_for_Ky}
    DK_y^j(x)p(x,0) - D_yq(x,0)K^j_y(x) -E_y^{j+M-1}(x) \in \mathcal{H}^{>j+M-1}
\end{equation}
we refer to Section $4.1$ of~\cite{BFM2} for more details. Once a solution $K^j_y$ to equation~\eqref{proof:eq_for_Ky} has been found, we aim to find the simplest possible representation of the dynamics on the stable manifold. It translates into choosing $R^{j+N-1}=0$ if we are able to solve the equation 
\begin{equation}\label{proof:eq_for_Kx_2}
    DK_x^j(x) p(x,0) -  D_xp(x,0) K^j_x(x) = E_x^{j + N -1}(x) + D_yp(x,0)K_y^j(x).
\end{equation}
Depending on $M$ and $N$, we obtain two distinct equations that ensure that condition~\eqref{proof:eq_for_Ky} holds:
\begin{itemize}
    \item If $M=N$, then condition~\eqref{proof:eq_for_Ky} is satisfied if
    \begin{equation}\label{proof:eq_for_Ky_N=M}
        DK_y^j(x)p(x,0) - D_yq(x,0)K^j_y(x) = E_y^{j+M-1}(x).
    \end{equation}
    \item If $M<N$,
    \begin{equation}\label{proof:eq_for_Ky_N>M}
         - D_yq(x,0)K^j_y(x) = E_y^{j+M-1}(x).
    \end{equation}
\end{itemize}

The rest of the proof is divided into three parts. First, we solve equations~\eqref{proof:eq_for_Ky_N=M} and~\eqref{proof:eq_for_Ky_N>M} for the component $K_y^j$ (see Section \ref{sec:sol_Ky}). In Section \ref{sec:sol_Kx}, we solve equation~\eqref{proof:eq_for_Kx} in order to find $K_x^j$. Finally, in Section \ref{sec:extra_ind}, when $M<N$ we deal with the extra induction procedure for values of $j$ such that $\ell - N +2 \le j \le \ell -M +1$ in order to prove~\eqref{thm:approx_sol}.

\subsubsection{Resolution of equations~\eqref{proof:eq_for_Ky_N=M} and~\eqref{proof:eq_for_Ky_N>M} for $K_y^j$.} \label{sec:sol_Ky}

We recall that $2 \le j \le \ell - M +1$. We consider the case $M<N$. We can solve equation~\eqref{proof:eq_for_Ky_N>M} using that $D_yq(x,0)$ is invertible (see~\eqref{Thm:hyp_2}). The unique solution $K_y^j \in \mathcal{H}^j$ of~\eqref{proof:eq_for_Ky_N>M} is
\begin{equation*}
    K_y^j (x) = -\left(D_yq(x,0)\right)^{-1}E_y^{j + M -1}(x)
\end{equation*}
and thanks~\eqref{thm:approx_sol_Hyp_decay},~\eqref{proof:Reg_ExEy}, and the algebric property contained in Proposition \ref{prop:NormGammaProdBil}, we have that 
\begin{equation}\label{proof:Kjyreg1}
    K_y^j \in C_\Gamma^1 \left(\Omega (\varrho_{j-1}, \gamma_j)\right).
\end{equation}

Now, let $M=N$, by~\eqref{thm:hyp_ApdpMN} we have that $A_p > b_p$. In this case, $K_y^j$ has to satisfies 
\begin{equation*}
    DK_y^j (x) p(x,0) -  D_yq(x,0) K^j_y(x) = E_y^{j + L -1}(x).
\end{equation*}
 We claim that letting $\mathbf{p}(x) = p(x,0)$ and $\mathbf{Q}(x) = D_yq(x,0)$ as indicated, we can solve the above equation using Theorem \ref{PCE}. Indeed, the constant $a_\mathbf{p}$, $b_\mathbf{p}$, $A_\mathbf{p}$, $B_\mathbf{Q}$, $A^\Gamma_\mathbf{p}$ and $B^\Gamma_\mathbf{Q}$ are
\begin{align*}
    &a_\mathbf{p} = a_p > 0 \hspace{2mm} \mbox{(by H1)}, \\
    &b_\mathbf{p} = b_p > 0, \hspace{2mm} B^\Gamma_\mathbf{Q} = B^\Gamma_q>0 \hspace{1mm}\mbox{and} \hspace{1mm} A^\Gamma_\mathbf{p} = A^\Gamma_p > 0 \hspace{2mm} \mbox{(by definition),} \hspace{2mm}\\
    &  A_\mathbf{p} = A_p \quad B_\mathbf{Q} = B_q, \quad 
\end{align*}
see definitions~\eqref{constants}, and the hypotheses of Theorem \ref{PCE} are satisfied. Then a solution $K^j_y \in \mathcal{H}^j$ of equation~\eqref{proof:eq_for_Ky_N=M} exists. Moreover, there exist $0< \varrho_j \le \varrho_{j-1}$ and $0<\gamma_j \le \gamma_{j-1}$ small enough such that 
\begin{equation}\label{proof:Kjyreg2}
    K_y^j \in C_\Gamma^1 \left(\Omega (\varrho_j, \gamma_j)\right).
\end{equation}

\subsubsection{Resolution of equation~\eqref{proof:eq_for_Kx} for $K_x^j$.} \label{sec:sol_Kx}

Here, we want to find a solution $K_x^j$ of equation~\eqref{proof:eq_for_Kx} for $2 \le j \le \ell - N +1$. We point out that Section \ref{sec:sol_Ky} provides a solution $K_y^j \in \mathcal{H}^j$ of~\eqref{proof:eq_for_Ky} satisfying $K_y^j \in C_\Gamma^1 \left(\Omega (\varrho_j, \gamma_j)\right)$ (we stress that~\eqref{proof:Kjyreg1} implies~\eqref{proof:Kjyreg2}). We observe that $D_yp(\cdot,0) K^j_y \in \mathcal{H}^{j+N-1}$ and thanks to~\eqref{thm:approx_sol_Hyp_decay}, \eqref{proof:Kjyreg1},~\eqref{proof:Kjyreg2}, Proposition \ref{prop:NormGammaProdBil}, Proposition \ref{prop:NormGammaProd}, and Proposition \ref{prop:cauchyBS_Decay}, there exists $0 < \gamma_j \le \gamma_{j-1}$ suitably small such that $D_yp(\cdot,0) K^j_y \in C_\Gamma^1 \left(\Omega (\varrho_j, \gamma_j)\right)$. We can add this term to $E_x^{j+N-1}$ that now satisfies
\begin{equation*}
    E_x^{j+N-1} \in C_\Gamma^1 \left(\Omega (\varrho_j, \gamma_j)\right)
\end{equation*}
instead of~\eqref{proof:Reg_ExEy}. Hence, we can rewrite equation~\eqref{proof:eq_for_Kx} as
\begin{equation}\label{proof:eq_for_Kx_3}
    DK_x^j(x) p(x,0) -  D_xp(x,0) K^j_x(x) + R^{j+N-1}(x) = E_x^{j + N -1}(x).
\end{equation}
We have considerable freedom solving this equation. An option is to choose $K_x^j(x)=0$ and $R^{j+N-1}(x) = E_x^{j + N -1}(x)$. Proceeding in this way, however, we do not obtain a normal form result for $R$. Consequently, the dynamics do not take a simple form.
In what follows, we will use different strategies for $M=N$ and $M<N$. For this reason, we consider the two cases separately. 

Let $M=N$, by~\eqref{thm:hyp_ApdpMN} we have that $A_p > b_p$. We want to choose $R^{j+N-1}(x) = 0$ and, using Theorem \ref{PCE} with $\mathbf{p}(x)=p(x,0)$ and $\mathbf{Q}(x)=D_xp(x,0)$, to solve equation~\eqref{proof:eq_for_Kx_3} for $K_x^j$. To this end, we observe that by H1 and Remark~\ref{rmk:ConstProp}, the constant $B_{\mathbf{Q}} = - B_p \le -Na_p <0$. This means that we cannot apply this strategy if $j$ is not large enough. For this reason, 
if $j > {B_p \over a_p}$, we set $R^{j + N -1} = 0$ and choose $K_x^j$ as the homogeneous solution of 
\begin{equation}\label{proof:eq_for_Kx_4}
    DK_x^j(x) p(x,0) -  D_xp(x,0) K^j_x(x) = E_x^{j + N -1}(x).
\end{equation}
We observe that equation~\eqref{proof:eq_for_Kx_4} can be solved because the condition $j > {B_p \over a_p}$ ensures that hypothesis~\eqref{CondCE} of Theorem~\ref {PCE} is satisfied. 

Otherwise, we take $K_x^j =0$ and set $R^{j+N-1}$ as the solution of~\eqref{proof:eq_for_Kx_3}, namely, $R^{j+N-1} = E_x^{j + N -1}$. Concerning the regularity, in the first case, $K_x^j$ solves~\eqref{proof:eq_for_Kx_4}. Since $\mathbf{p}(x) = p(x,0)$ and $\mathbf{Q}(x) = D_x p(x,0)$, condition~\eqref{GammaCondCE} in Theorem \ref{PCE} translates into
\begin{equation*}
    j - 1 -2{A^\Gamma_p \over a_p} >0,
\end{equation*} 
where we recall that we are considering the case $j > {B_p \over a_p}$. Then, if $$j > \max\left\{{B_p \over a_p}, 1 + 2{A^\Gamma_p \over a_p}\right\},$$ Theorem \ref{PCE} ensures the existence of $0 < \varrho_j \le \varrho_{j-1}$ and $0<\gamma_j \le \gamma_{j-1}$ suitably small such that the solution $K_x^j$ of~\eqref{proof:eq_for_Kx_4} satisfies
\begin{equation*}
    K_x^j \in C_\Gamma^1 \left(\Omega (\varrho_j, \gamma_j)\right).
\end{equation*}
On the other hand, if we choose $K_x^j =0$ and $R^{j+N-1} = E_x^{j + N -1}$, then obviously 
\begin{equation*}
    R^{j+N-1} \in C_\Gamma^1 \left(\Omega (\varrho_j, \gamma_j)\right).
\end{equation*}

Finally, when $M < N$, we also take $K_x^j =0$ and $R^{j+N-1} = E_x^{j + N -1}$ (see Section~$4.4$ of~\cite{BFM2}). 

We want to point out that if we are not interested in having a normal form for~$R$, equation~\eqref{proof:eq_for_Kx_3} can be solved by simply taking $K_x^j =0$ and $R^{j+N-1} = E_x^{j + N -1}$. 

\subsubsection{Extra induction procedure}\label{sec:extra_ind}

As mentioned at the beginning of Section \ref{sec:Inductive_procedure}, if $M=N$, then~\eqref{thm:approx_sol},~\eqref{thm:formaKR}, and~\eqref{thm:regularity} follow from~\eqref{proofThmApprox:RegKR} and~\eqref{proofThmApprox:iterative_eq} by taking $j = 2,\dots, \ell - N +1$. When $M < N$, we also have to analyse the equation for $K^j$ with $j = \ell - N +2,\dots, \ell -M+1$. This means that we need to add extra homogeneous terms to $K_y$ to obtain~\eqref{thm:approx_sol} and~\eqref{thm:formaKR}. We will prove it by an inductive procedure. We assume that $K^{\ell - N +1}$ and $R^{\le \ell}$ are of the form~\eqref{proof:FormKandR} and they satisfy~\eqref{proofThmApprox:RegKR} and~\eqref{proofThmApprox:iterative_eq} for $j = \ell -N +1$. We want to prove that, for any $\ell - N +2 \le j \le \ell - M +1$, there exists $K^j$ in such a way that 
\begin{equation*}
    K^{\le j} = K^{\le \ell - N +1} + \sum_{l = \ell - N +2}^j K^l \hspace{5mm} \mbox{with $K^l \in \mathcal{H}^l$, $K_x^l = 0$}
\end{equation*}
and it satisfies~\eqref{thm:regularity}, and $E^{>j} = F \circ K^{\le j} -  K^{\le j} \circ R^{\le \ell} \in \mathcal{H}^{>\ell} \times \mathcal{H}^{> j + M -1}$.

As an inductive hypothesis, we assume that the previous result holds for $j-1$. We need to compute $K^j$ and $R^{j+N-1}$. Following the lines of Section \ref{sec:Inductive_procedure} and taking $K_x^j=R^{j+N-1}=0$, one can prove the existence of $E_y^{j + M -1} \in \mathcal{H}^{j+M-1}$ satisfying~\eqref{proof:Reg_ExEy} and $\tilde E^{>j}_x \in \mathcal{H}^{>j + N -1} \subset \mathcal{H}^{> \ell}$, and $\tilde E^{>j}_y \in \mathcal{H}^{> j + M -1}$ such that 
\begin{eqnarray*}
    E_x^{>j}(x) &=& F_x \circ K^{\le j}(x) -  K^{\le j}_x(x) \circ R(x) = D_yp(x,0)K^j_y(x) + \tilde E^{>j}_x(x)\\
    E_y^{>j}(x) &=& F_y \circ K^{\le j}(x) -  K^{\le j}_y(x) \circ R(x)  = -DK_y^j(x)p(x,0) + D_yq(x,0)K^j_y(x)\\
    &+& E_y^{j+M-1}(x) + \tilde E_y^{>j}(x)
\end{eqnarray*}
where we recall that $R$ is defined by~\eqref{thm:formaKR}. We refer to Section $4.1$ of~\cite{BFM2} for more details. We recall that $M < N$, so we can choose $K_y^j$ as a solution of~\eqref{proof:eq_for_Ky_N>M} and of course, it satisfies~\eqref{proof:Kjyreg1}. This implies that $D_yp(\cdot,0)K^j_y \in \mathcal{H}^{j + N -1} \subset \mathcal{H}^{>\ell}$ and hence $E_x^{>j} \in \mathcal{H}^{>\ell}$. 

We can follow this procedure for all $j = \ell - N +2,\dots, \ell -M+1$. This concludes the proof of Theorem \ref{ThmMapsPS}. In particular we proved the existence of $K$ and $R$ of the form~\eqref{thm:approx_sol} satisfying~\eqref{thm:approx_sol}. Moreover, it is an inductive procedure with a finite number of steps. Hence, there exist $0 < \varrho < \varrho_0$ and $0 < \gamma \le \gamma_0$ such that~\eqref{thm:regularity} is satisfied. 


\section{Proof of Theorem \ref{ThmMapsPS}} \label{Proof_ThmMapsPS}

We prove Theorem \ref{ThmMapsPS} using a fixed point argument. 
Throughout this section, we will assume that all the hypotheses of Theorem \ref{ThmMapsPS} are satisfied. Furthermore, we denote by $C$ a generic positive constant which may vary throughout this section. 

This section is organized as follows. First, we formally write down the fixed-point equation we aim to solve. This is the content of Section \ref{proof:Thm_PS_PA}. In Section \ref{sec:lin_op_S}, we verify that the aforementioned fixed-point equation is well defined, and in Section~\ref{sec:op_N} we prove the existence of a solution satisfying the required decay properties.

\subsection{Preliminary analysis and functional setting}\label{proof:Thm_PS_PA}
Let $\ell \in \N$ such that $\ell_0 < \ell \le r$, where $\ell_0$ is defined by~\eqref{thm_apost:condition_l}. We expand $F$ as in~\eqref{F} around $(x,y)=(0,0)$ and we rewrite it as
\begin{equation}\label{def:F=P+G}
    F(x,y) = P(x,y) + G_\ell(x,y)
\end{equation}
where $P$ is a sum of homogeneous functions of order less than $\ell -1$ and $G_\ell \in \mathcal{H}^{\geq \ell}$.
According to the hypotheses of Theorem \ref{ThmMapsPS}, there exist analytic maps $K^\le$ and $R$ satisfying~\eqref{thm:hyp_existence_KR} such that 
\begin{equation}\label{proof:thm_PS_eq_P}
    P \circ K^\le - K^\le \circ R = \mathcal{O}(|x|^\ell).
\end{equation}
We aim to prove the existence of an analytic solution $K^>$ of the equation
\begin{equation}\label{proof:Thm_PS_eq}
    F\circ (K^\le + K^>) - (K^\le + K^>) \circ R=0.
\end{equation}
The rest of this section is dedicated to writing~\eqref{proof:Thm_PS_eq} a fixed-point equation.  First, using the ideas in~\cite{BFM1}, we introduce a suitable scaling (to obtain better estimates) and prove several technical lemmas. Finally, we rewrite~\eqref{proof:Thm_PS_eq} as a fixed-point equation. 

Given $\delta >0$, we consider the following scaling $S_\delta(x,y) = (x, \delta y)$ in the $y$- variable and we rewrite equations~\eqref{proof:thm_PS_eq_P} and~\eqref{proof:Thm_PS_eq} as
\begin{align}
&\tilde P \circ \tilde K^\le - \tilde K^\le \circ R = \mathcal{O}(|x|^\ell),\nonumber \\ \label{proof:Thm_PS_eq_tilde}
&\tilde F\circ (\tilde K^\le + \tilde K^>) - (\tilde K^\le + \tilde K^>) \circ R=0,
\end{align}
where $\tilde P = S_\delta^{-1} \circ P \circ S_\delta$, $\tilde F = S_\delta^{-1} \circ F \circ S_\delta$, $\tilde K^{\le} = S_\delta^{-1} \circ K^\le$ and $\tilde K^> = S_\delta^{-1} \circ K^>$. From now on, we drop the symbol $\,\tilde{}\,$ from the scaled functions. 

For the rest of the proof of Theorem \ref{ThmMapsPS}, we fix the constants $a$, $a^*$, $b$, $b^*$, $A$, $B$ such that 
\begin{equation}\label{proof:ThmMapsPS_constants}
\begin{aligned}
    & 0<a<a_p, \quad a^* < a(N-1), \quad b > b_p, \quad b^* > b(N-1),  \\
    & A < A_p, \quad B > B_p, \quad A>b,\\
    &\ell_0 < N - 1 + {B \over a^*} < \ell \le r. 
\end{aligned}
\end{equation}
In addition, in order to verify the required decay properties, we fix the  constants
\begin{equation}\label{proof:ThmMapsPS_constants_Gamma}
\begin{aligned}
    & A^\Gamma > A_p^\Gamma, \quad C^\Gamma \ge C^\Gamma_0 >\max\{A_p^\Gamma, B^\Gamma_p\}\\
    &\ell_1 < N + {C^\Gamma + A^\Gamma \over a^*} < \ell \le r. 
\end{aligned}
\end{equation}
where $A^\Gamma_p$ and $B^\Gamma_p$ are defined by~\eqref{GammaConstantsApproxThm}.

Taking into account~\eqref{proof:Thm_PS_eq_tilde}, we rewrite equation~\eqref{proof:Thm_PS_eq}. Indeed, we observe that 
\begin{align*}
    0 &= F\circ (K^\le + K^>) - (K^\le + K^>) \circ R\\
    &= P \circ K^\le - K^\le \circ R + G_\ell \circ (K^\le + K^>) + P \circ (K^\le + K^>) - P \circ K^\le\\
    &- \left(DP \circ K^\le\right) K^>+ \left(DP \circ K^\le \right) K^> - K^>\circ R,
\end{align*}
where the latter is obtained by~\eqref{proof:Thm_PS_eq_tilde} by adding and subtracting  $P \circ K^\le$ and $\left(DP \circ K^\le \right)K^>$. We observe that we can rewrite~\eqref{proof:Thm_PS_eq_tilde} as
\begin{equation}\label{proof:Thm_PS_eq_LN}
    \mathcal{L}(K^>) = - \mathcal{N}(K^>),
\end{equation}
where 
\begin{align}\label{proof:Thm_PS_def_L}
    \mathcal{L}(K^>) &= \left(DP \circ K^\le\right) K^> - K^> \circ R,\\
    \mathcal{N}(K^>) &= P \circ K^\le - K^\le \circ R + G_\ell \circ (K^\le + K^>) + P \circ (K^\le + K^>)\nonumber\\ \label{proof:Thm_PS_def_N}
    &- P \circ K^\le - \left(DP \circ K^\le\right) K^>.
\end{align}
To express~\eqref{proof:Thm_PS_eq_LN} as a fixed point equation, we need to find a right inverse of the operator $\mathcal{L}$. A formal right inverse of  $\mathcal{L}$ is given by
\begin{equation}\label{proof:thm_right_inv}
    \mathcal{S}(T) = \sum_{j=0}^\infty \left[\prod_{m=0}^i (DP)^{-1} \circ K^\le \circ R^m\right]T\circ R^j,
\end{equation}
Hence, we rewrite equation~\eqref{proof:Thm_PS_eq_LN} as the fixed-point equation
\begin{equation}\label{proof:Thm_PS_eq_FP}
    K^> = - \mathcal{S}\circ \mathcal{N}(K^>).
\end{equation}
We point out that if $K^>$ is a solution of the latter, then it is a solution of~\eqref{proof:Thm_PS_eq_LN}.

\subsection{The linear operator $\mathcal{S}$}\label{sec:lin_op_S}
Here, we deal with the linear operator $\mathcal{S}$ formally introduced by~\eqref{proof:thm_right_inv}. We will see that, on suitable Banach spaces, it is well-defined and bounded. For this purpose, this section is divided into two parts. In Section \ref{sec:opS_preliminary}, we prove some technical lemmas we will use in Section \ref{sec:opS_properties} in order to prove the properties of the linear operator $\mathcal{S}$. 

\subsubsection{Preliminary estimates}\label{sec:opS_preliminary}

We define the operator $\mathbf{M}=DP \circ K^\le- \mathrm{Id}$. We have that
\begin{equation}\label{opS:def_M}
    \mathbf{M} = \begin{pmatrix} D_x p \circ K^\le + D_x \hat f \circ K^\le & \delta D_y p\circ K^\le + \delta D_y \hat f \circ K^\le\\
    {1 \over \delta} D_x q \circ K^\le + {1 \over \delta} D_x \hat g \circ K^\le & D_yq \circ K^\le + D_y \hat g \circ K^\le\end{pmatrix}
\end{equation}
for all $x \in \Omega(\varrho, \gamma)$, where 
\begin{equation*}
\hat f = F_x^{N+1}+ \cdots + F^{\ell -1}_x,\quad \hat g = F_y^{N+1}+ \cdots + F^{\ell -1}_y,
\end{equation*}
(see definitions~\eqref{F} and~\eqref{proof:thm_PS_eq_P}). 
\begin{lemma}\label{opS:lemma_M}
     We assume $A_p > b_p$. Let $C_0^\Gamma > \max \left\{A_p^\Gamma, B_q^\Gamma\right\}$, where $A_p^\Gamma$ and $B_q^\Gamma$ are defined in~\eqref{GammaConstantsApproxThm}. Then, for $\varrho$, $\gamma$ and $\delta$ small enough 
    \begin{equation*}
        \left|\mathbf{M}(x)\right|_\Gamma \le C_0^\Gamma |x|^{N-1},
    \end{equation*}
     for all $x \in \Omega(\varrho, \gamma)$.
\end{lemma}
     \begin{proof}
         We provide estimates of all the components of~$\mathbf{M}$. First, we look for an upper bound of
\begin{equation*}
   |D_x p \circ K^\le(x) + D_x \hat f \circ K^\le(x)|_\Gamma \le |D_x p \circ K^\le(x)|_\Gamma + |D_x \hat f \circ K^\le(x)|_\Gamma 
\end{equation*}
for all $x \in \Omega(\varrho, \gamma)$.  To this end, using Lemma \ref{Lemma3.8},~\eqref{GammaConstantsApproxThm},~\eqref{def:OmegaSet},~\eqref{thm:hyp_existence_KR},~\eqref{thm:aporsteriori_sol_Hyp_decay}, Proposition \ref{prop:NormGammaProdBil}, Proposition \ref{prop:cauchyBS_Decay} and the homogeneous character of $p$
\begin{align*}
    \left|D_x p \circ K^\le(x)\right|_\Gamma &= \left|D_x p(\mathrm{Re}\, x, 0)\right|_\Gamma +  \left|\int_0^1D_x^2 p(\mathrm{Re}\, x + \tau i\mathrm{Im}\, x, 0)d\tau\right|_\Gamma \gamma |x| \\
    &+ \left|\int_0^1DD_xp \left(x + \tau (K_x^\le(x) -x), \tau K_y^{\le}(x)\right)d\tau\right|_\Gamma \left|K^\le(x) - (x,0)\right|\\
    &\le A_p^\Gamma |x|^{N-1} + (\gamma + \varrho) C  |x|^{N-1}
\end{align*}
for all $x \in \Omega(\varrho, \gamma)$. Similarly, one can prove that
\begin{equation*}
    |D_x \hat f \circ K^\le(x)|_\Gamma \le \varrho C |x|^{N-1},
\end{equation*}
and combining the previous two estimates, we obtain
\begin{equation}\label{opS:stima_M_1}
   |D_x p \circ K^\le(x) + D_x \hat f \circ K^\le(x)|_\Gamma\le A_p^\Gamma |x|^{N-1} + (\gamma + \varrho) C  |x|^{N-1},
\end{equation}
for all $x \in \Omega(\varrho, \gamma)$.
Using a similar argument, taking $\varrho>0$ small enough so that ${\sqrt{\varrho} \over \delta}<1$,  one can verify that for all $x \in \Omega(\varrho, \gamma)$
\begin{equation}
\begin{aligned}\label{opS:stima_M_others}
    \left|\delta D_y p\circ K^\le(x) + \delta D_y \hat f \circ K^\le(x)\right|_\Gamma &\le \delta C |x|^{N-1},\\
    \left| {1 \over \delta} D_x q \circ K^\le(x) + {1 \over \delta} D_x \hat g \circ K^\le(x) \right|_\Gamma &\le {\varrho \over \delta}C|x|^{N-1} \le C\sqrt{\varrho}|x|^{N-1},\\
    \left|D_yq \circ K^\le(x) + D_y \hat g \circ K^\le(x)\right|_\Gamma &\le B_q^{\Gamma} |x|^{N-1} + (\gamma + \varrho) C  |x|^{N-1},
    \end{aligned}
\end{equation}
where, in the second estimate above, we have used hypothesis H2. 

Combining~\eqref{opS:stima_M_1} and~\eqref{opS:stima_M_others} with~\eqref{opS:def_M}, for $\varrho$, $\gamma$, and $\delta$ small enough, we obtain the desired estimate, which concludes the proof of the lemma. 
     \end{proof}

Next lemma allows us to control the iterates of $R$.
\begin{lemma}\label{proof:ThmPS_IterR}
We assume $A_p > b_p$ and $a$, $a^*$, $b$, $b^*$, $A$ are such that  satisfy~\eqref{proof:ThmMapsPS_constants}. 
Then, there exist positive $\varrho$, $\gamma$ small enough such that
\begin{enumerate}
    \item the set $\Omega(\varrho, \gamma)$ is invariant by $R$, that is
    \begin{equation*}
        R\left(\Omega(\varrho, \gamma)\right) \subset \Omega(\varrho, \gamma),
    \end{equation*}
    \item for all $k \ge 0$ and $x \in \Omega(\varrho, \gamma)$
    \begin{equation*}
        {|x| \over (1 + k b^*|x|^{N-1})^{1 \over N-1}} \le \left| R^k(x) \right| \le  {|x| \over (1 + k a^*|x|^{N-1})^{1 \over N-1}}.
    \end{equation*}
\end{enumerate}
\end{lemma}
\begin{proof}
    The proof is similar to that given in Appendix C of~\cite{BFMARMA}.
\end{proof}

\begin{lemma}\label{opS:lemma_DP-1_DP-1Gamma}
    We assume $A_p > b_p$. Let $a^*$ be the constant introduced in~\eqref{proof:ThmMapsPS_constants}. We consider a constant $C^\Gamma \ge C^\Gamma_0 > \max \left\{A_p^\Gamma, B_q^\Gamma\right\}$, where $A_p^\Gamma$ and $B_q^\Gamma$ are defined in~\eqref{GammaConstantsApproxThm} and $C^\Gamma_0$ is introduced in Lemma \ref{opS:lemma_M}. Then, for $\varrho$, $\gamma$ and $\delta$ small enough there exists a positive constant $C_0$ such that 
    \begin{align}\label{opS:lemma_stimaprodDP}
        &\left|\prod_{m=0}^j (DP)^{-1} \circ K^\le \circ R^m(x)\right|_\Gamma \le 2\left|\mathrm{Id} \right|_\Gamma \left(1 + j a^*|x|^{N-1}\right)^{C^\Gamma \over a^*},\\ \label{opS:lemma_stimaDDP}
        &\left|D\left[(DP)^{-1}\circ K^\le\right](x)\right|_\Gamma \le C_0 |x|^{N-2},
    \end{align}
    for all $x \in \Omega(\varrho, \gamma)$.
\end{lemma}
\begin{proof}
First, we prove~\eqref{opS:lemma_stimaprodDP}. To this end, we observe that, thanks to Lemma \ref{opS:lemma_M}, the operator $DP \circ K^\le$ is close to the identity when $\varrho$ is small enough. Hence,  there exists an operator $\mathbf{M}_1$ such that 
\begin{equation*}
    \mathbf{M}_1 = (DP)^{-1}\circ K^\le - \mathrm{Id},
\end{equation*}
provided that $\varrho$ is sufficiently small. Moreover, there exists a constant $C^\Gamma \ge C^\Gamma_0$ such that, if $\varrho$, $\gamma$ and $\delta$ are suitably small 
\begin{equation*}
        \left|\mathbf{M}_1(x)\right|_\Gamma \le C^\Gamma |x|^{N-1}
\end{equation*}
for all $x \in \Omega(\varrho, \gamma)$. 

Using the previous estimate, Lemma \ref{proof:ThmPS_IterR} and Proposition \ref{app:prop_prod_with_identity}, for all $x \in \Omega(\varrho, \gamma)$, we have that 
\begin{align*}\label{opS:proof_stima_intermedia_DP-1}
    &\left|\prod_{m=0}^j (DP)^{-1} \circ K^\le \circ R^m(x)\right|_\Gamma = \left|\prod_{m=0}^j \left( \mathrm{Id} + \mathbf{M}_1\circ R^m(x)\right)\right|_\Gamma\\
    &\qquad \le |\mathrm{Id}|_\Gamma \prod_{m=0}^j \left(1 + \left|\mathbf{M}_1\circ R^m(x)\right|_\Gamma\right) \le |\mathrm{Id}|_\Gamma \prod_{m=0}^j \left(1 + C^\Gamma \left|R^m(x)\right|^{N-1}\right)\\
    &\qquad =  |\mathrm{Id}|_\Gamma  e^{\sum_{m=0}^j \ln \left(1 + C^\Gamma|R^m(x)|^{N-1}\right)} \le |\mathrm{Id}|_\Gamma e^{\sum_{m=0}^j C^\Gamma|R^m(x)|^{N-1}} \\
    &\qquad \le |\mathrm{Id}|_\Gamma  e^{\sum_{m=0}^j C^\Gamma{|x|^{N-1}\over \left(1 + ma^*|x|^{N-1}\right)}} \le |\mathrm{Id}|_\Gamma e^{C^\Gamma|x|^{N-1}}e^{\ln(1 + ja^*|x|^{N-1})^{C^\Gamma \over a^*}} \\
    &\qquad \le |\mathrm{Id}|_\Gamma e^{C^\Gamma\varrho^{N-1}}(1 + ja^*|x|^{N-1})^{C^\Gamma \over a^*} \le 2|\mathrm{Id}|_\Gamma (1 + ja^*|x|^{N-1})^{C^\Gamma \over a^*} 
\end{align*}
for $\varrho$ small enough. We point out that, in the third line of the latter, we used the trivial estimate $\ln \left(1 + C^\Gamma|R^m(x)|^{N-1}\right) \le  C^\Gamma|R^m(x)|^{N-1}$. Furthermore, in the second estimate in the fourth line of the above inequalities, we used the Cauchy–Maclaurin sum–integral inequality. This proves~\eqref{opS:lemma_stimaprodDP}.  It remains to verify the estimate~\eqref{opS:lemma_stimaDDP}. First, by Proposition \ref{prop:NormGammaProd}, we obtain 
\begin{equation*}
    \left|D\left[(DP)^{-1}\circ K^\le\right](x)\right|_\Gamma \le  \left|D(DP)^{-1}\circ K^\le(x)\right|_\Gamma\left|DK^\le (x)\right|_\Gamma
\end{equation*}
for all $x \in \Omega(\varrho, \gamma)$. By hypothesis~\eqref{thm:aporsteriori_sol_Hyp_decay_2}, we know that $\left|DK^\le (x)\right|_\Gamma \le C$. Therefore, it remains to find an upper bound for the other term on the right-hand side of the above inequality. Using the Cauchy inequalities in Proposition \ref{prop:cauchyBS_Decay}, hypotheses~\eqref{thm:hyp_existence_KR}, and~\eqref{thm:aporsteriori_sol_Hyp_decay} and the homogeneous character of the components of $D(DP)^{-1}$, it is a straightforward computation to verify that $\left|D(DP)^{-1}\circ K^\le(x)\right|_\Gamma \le C|x|^{N-2}$. This concludes the proof of~\eqref{opS:lemma_stimaDDP}.
\end{proof}

In order to control the differential of the iterates of $R$ in the norm $|\cdot|_\Gamma$, we have the following
\begin{lemma}\label{opS:lemma_iter_DR}
 We assume $A_p > b_p$. Let $A^\Gamma_p$ and $a^*$ be the constants introduced in~\eqref{GammaConstantsApproxThm} and  $A^\Gamma > A_p^\Gamma$ the one defined by~\eqref{proof:ThmMapsPS_constants_Gamma}. 
Then, for $\gamma$ and $\varrho$ small enough and $j\ge 1$ 
\begin{equation*}
    \left|DR^j(x)\right|_\Gamma \le 2 \left|\mathrm{Id}\right|_\Gamma \left(1 + (j-1) a^* |x|^{N-1}\right)^{A^\Gamma \over a^*}
\end{equation*}
for all $x \in \Omega(\varrho, \gamma)$.
\end{lemma}
\begin{proof}
We define the following operator $\mathbf{N} = DR - \mathrm{Id}$ and we observe that, for all $j\ge 0$
\begin{equation*}
    DR^j(x) = \prod_{l=0}^{j-1}DR\circ R^l(x) = \prod_{l=0}^{j-1} \left(\mathrm{Id} + \mathbf{N}\circ R^l(x)\right),
\end{equation*}
for all $x \in \Omega(\varrho, \gamma)$. We point out that the product in the previous expression should be understood as $DR^j = DR \circ R^{j-1}DR \circ R^{j-2}\cdots DR$. Thanks to Proposition \ref{app:prop_prod_with_identity}, we obtain that
\begin{equation}\label{opS:dim_stima_DRj_1}
    \left|DR^j(x)\right|_\Gamma \le \left|\mathrm{Id}\right|_\Gamma \prod_{l=0}^{j-1}\left(1 + \left|\mathbf{N}\circ R^l(x)\right|_\Gamma\right),
\end{equation}
for all $x \in \Omega(\varrho, \gamma)$. First, we provide an upper bound for $\left|\mathbf{N}(x)\right|_\Gamma$. To this end, by the definition of $\mathbf{N}$, Lemma \ref{Lemma3.8},~\eqref{GammaConstantsApproxThm},~\eqref{thm:aporsteriori_sol_Hyp_decay} and~\eqref{thm:aporsteriori_sol_Hyp_decay_2}, we observe that 
\begin{align*}
    \left|\mathbf{N}(x)\right|_\Gamma &\le \left|DR(x) - \mathrm{Id} - D_xp(x,0) \right|_\Gamma + \left|D_xp(x,0) \right|_\Gamma \\
    &\le A^\Gamma_p |x|^{N-1} + (\gamma + \varrho) C |x|^{N-1}
    \le A^\Gamma |x|^{N-1},
\end{align*}
for all $x \in \Omega(\varrho, \gamma)$ and $\varrho$ and $\gamma$ small enough. Replacing the latter into~\eqref{opS:dim_stima_DRj_1} we obtain that 
\begin{equation*}\label{opS:dim_stima_DRj_2}
    \left|DR^j(x)\right|_\Gamma \le \left|\mathrm{Id}\right|_\Gamma \prod_{l=0}^{j-1}\left(1 + A^\Gamma \left|R^l(x)\right|^{N-1}\right).
\end{equation*}
Thus, using the same arguments used the first part of the proof of Lemma~\ref{opS:lemma_DP-1_DP-1Gamma}, the claim follows. 
\end{proof}

\subsubsection{Properties of the operator $\mathcal{S}$}\label{sec:opS_properties}

We introduce the Banach spaces that we will use for the rest of the proof of Theorem \ref{ThmMapsPS}. Given $\ell \in \Z$, and $0 < \varrho, \gamma \le 1$ we define 
\begin{equation}
\label{def:XlGamma}
    \mathcal{X}_{\ell, \Gamma} = \left\{
    f :\Omega(\varrho, \gamma) \to \ellX \times \ellY \;\middle|\;
    \mbox{ $f$ real analytic,}  \;|f|_{\ell, \Gamma}   < \infty
   \right\}
\end{equation}
endowed with the norm
\begin{equation*}
   |f|_{\ell, \Gamma} = \sup_{x \in \Omega(\varrho, \gamma) } {|f(x)| \over |x|^\ell} + \sup_{x \in \Omega(\varrho, \gamma) } {|Df(x)|_\Gamma \over |x|^{\ell-1}}.
\end{equation*}
It is a Banach space. Furthermore, it is straightforward to verify that if $\ell < s$, then $\mathcal{X}_{s, \Gamma} \subset \mathcal{X}_{\ell, \Gamma}$, and if $f \in \mathcal{X}_{s, \Gamma}$ then $|f|_{\ell, \Gamma} \le \varrho^{s-\ell}|f|_{s, \Gamma}$. Finally, if $f \in \mathcal{X}_{s, \Gamma}$ and $g \in \mathcal{X}_{\ell, \Gamma}$ then  $fg \in \mathcal{X}_{\ell+s, \Gamma}$ and $|fg|_{\ell+s, \Gamma} \le 2|f|_{s, \Gamma} |g|_{\ell, \Gamma}$. 
\begin{lemma}\label{lemma:S_Inv}
    We assume that $A_p > b_p$ and 
    \begin{align}\label{opS:lemma_hyp_ell}
        &{\ell \over N-1} - {B \over a^*} >1,\\ \label{opS:lemma_hyp_ell_Gamma}
        &{\ell -1 \over N-1} - {C^\Gamma + A^\Gamma \over a^*} >1,
    \end{align}
     where  $B$ and $a^*$ are the constants defined in~\eqref{proof:ThmMapsPS_constants}, while $A^\Gamma$ and $C^\Gamma$ in~\eqref{proof:ThmMapsPS_constants_Gamma}. 
    Then the linear operator $\mathcal{S}: \mathcal{X}_{\ell, \Gamma} \to \mathcal{X}_{\ell -N+1, \Gamma}$ is well defined and bounded. 
\end{lemma}
\begin{proof}
    We fix $a^*$, $b^*$, $\varrho$, and $\gamma$ satisfying the conditions of Lemma \ref{proof:ThmPS_IterR}. As a consequence $\Omega(\varrho, \gamma)$ is invariant by $R$. We have that 
    \begin{equation}\label{opS:proof_stima_finale_S_NoGamma}
        \left|\mathcal{S}(T)(x)\right| \le C|T|_\ell |x|^{\ell -N+1}
    \end{equation}
    for all $x \in \Omega(\varrho, \gamma)$. The proof is similar to that in~\cite{BFM1, BFMARMA}. For this reason, it is omitted.
    
     By the definition of $\mathcal{X}_{\ell, \Gamma}$ in~\eqref{def:XlGamma}, it is enough to findof $\left|D\mathcal{S}(T)(x)\right|_\Gamma$, for all $x \in \Omega(\varrho, \gamma)$. To this end, we observe that 
    \begin{equation}\label{opS:proof_DS}
        D\mathcal{S}(T) = \mathcal{S}_1 + \mathcal{S}_2
    \end{equation}
    with 
    \begin{align*}
       \mathcal{S}_1(x) &= \sum_{j=0}^\infty \left[\prod_{m=0}^j(DP)^{-1}\circ K^\le \circ R^m(x)\right]DT\circ R^j(x) DR^j(x),\\
       \mathcal{S}_2(x) &= \sum_{j=0}^\infty\Bigg[\sum_{m=0}^j\left(\prod_{l=0}^{m-1}(DP)^{-1}\circ K^\le \circ R^l(x)\right)D\Big((DP)^{-1}\circ K^\le \circ R^m(x)\Big)\\
       &\times\left(\prod_{l=m+1}^{j}(DP)^{-1}\circ K^\le \circ R^l(x)\right)\Bigg]T\circ R^j.
    \end{align*}
    We will estimate $\mathcal{S}_1$ and $\mathcal{S}_2$ separately. First, we analyze $\mathcal{S}_1$. Using Proposition~\ref{prop:NormGammaProd} and Lemma~\ref{opS:lemma_DP-1_DP-1Gamma}, we have that, for all $x \in \Omega(\varrho, \gamma)$
    \begin{equation}\label{opS:proof_est_S1_1}
    \begin{aligned}
        \left|\mathcal{S}_1(x)\right|_\Gamma &\le \sum_{j=0}^\infty\left|\prod_{m=0}^j(DP)^{-1}\circ K^\le \circ R^m(x)\right|_\Gamma \left|DT\circ R^j(x)\right|_\Gamma\left|DR^j(x)\right|_\Gamma\\
        &\le 4\left|\mathrm{Id}\right|^2_\Gamma\sum_{j=0}^\infty \left(1 + ja^*|x|^{N-1}\right)^{C^\Gamma + A^\Gamma \over a^*}\left|DT\circ R^j(x)\right|_\Gamma ,
        \end{aligned}
    \end{equation}
    where in the last line of the latter we used that $(1 + (j-1)a^*|x|^{N-1})^{A^\Gamma \over a^*} \le (1 + ja^*|x|^{N-1})^{A^\Gamma \over a^*}$. Since $T \in \mathcal{X}_{\ell,\Gamma}$ and using Lemma \ref{proof:ThmPS_IterR}, we observe that
    \begin{equation*}
       \left|DT\circ R^j(x)\right|_\Gamma \le |T|_{\ell, \Gamma} \left|R^j(x)\right|^{\ell -1} \le |T|_{\ell, \Gamma} {|x|^{\ell -1} \over \left(1 + j a^*|x|^{N-1}\right)^{\ell-1 \over N-1}},
    \end{equation*}
    for all $x \in \Omega(\varrho, \gamma)$.
    Replacing the above estimate into~\eqref{opS:proof_est_S1_1} we can see that
    \begin{equation*}
        \left|\mathcal{S}_1(x)\right|_\Gamma\le 4\left|\mathrm{Id}\right|^2_\Gamma |T|_{\ell, \Gamma}\sum_{j=0}^\infty {|x|^{\ell -1} \over \left(1 + ja^*|x|^{N-1}\right)^{{\ell -1 \over N-1} -{C^\Gamma + A^\Gamma \over a^*}}}.
    \end{equation*}
    We recall that we are assuming ${\ell -1 \over N-1} -{C^\Gamma + A^\Gamma \over a^*} >1$ (see~\eqref{opS:lemma_hyp_ell}) and hence, thanks to the Cauchy–Maclaurin sum–integral inequality, we can conclude that 
    \begin{equation}\label{opS:proof_stima_finale_S1}
       \left|\mathcal{S}_1(x)\right|_\Gamma\le C|T|_{\ell, \Gamma} |x|^{\ell-N},
    \end{equation}
    for all $x \in \Omega(\varrho, \gamma)$. It remains to deal with the second term $\mathcal{S}_2$ in the right-hand side of~\eqref{opS:proof_DS}. For this purpose, using Propositions~\ref{prop:NormGammaProdBil} and~\ref{prop:NormGammaProd}, 
    \begin{align}\label{opS:proof_est_S2_1}
        \left|\mathcal{S}_2(x)\right|_\Gamma &\le \sum_{j=0}^\infty\Bigg[\sum_{m=0}^j\left|\prod_{l=0}^{m-1}(DP)^{-1}\circ K^\le \circ R^l(x)\right|_\Gamma \left|D\Big((DP)^{-1}\circ K^\le \circ R^m(x)\Big)\right|_\Gamma \nonumber\\
       &\times\left|\prod_{l=m+1}^{j}(DP)^{-1}\circ K^\le \circ R^l(x)\right|_\Gamma\left|T\circ R^j(x)\right|
    \end{align}
     for all $x \in \Omega(\varrho, \gamma)$. We need to estimate each norm in the right-hand side of the latter. First, we observe that, for all $j \ge 0$ and $0 \le m \le j$, similarly to the proof of Lemma~\ref{opS:lemma_DP-1_DP-1Gamma}, we obtain
     \begin{align}\label{opS:proof_est_S2_est_PiPi}
         &\left|\prod_{l=0}^{m-1}(DP)^{-1}\circ K^\le \circ R^l(x)\right|_\Gamma \left|\prod_{l=m+1}^{j}(DP)^{-1}\circ K^\le \circ R^l(x)\right|_\Gamma \nonumber\\
         &\le \left|\mathrm{Id}\right|_\Gamma^2 \prod_{l=0}^{m-1} \left(1 + C^\Gamma \left|R^l(x)\right|^{N-1}\right)\prod_{l=m+1}^j \left(1 + C^\Gamma \left|R^l(x)\right|^{N-1}\right)\nonumber\\
         &\le \left|\mathrm{Id}\right|_\Gamma^2 \prod_{l=0}^{j} \left(1 + C^\Gamma \left|R^l(x)\right|^{N-1}\right) \le 2\left|\mathrm{Id}\right|_\Gamma^2 (1 + j a^*|x|^{N-1})^{C^\Gamma \over a^*}
     \end{align}
    where in the last line of the latter, first we used the trivial estimate $1 \le 1 + C^\Gamma |R^m(x)|$ and the last inequality follows by the same argument as in the proof of~\eqref{opS:lemma_stimaprodDP} of Lemma \ref{opS:lemma_DP-1_DP-1Gamma}. Furthermore, by estimates~\eqref{opS:lemma_stimaDDP}, Proposition~\ref{prop:NormGammaProdBil} and Lemma~\ref{proof:ThmPS_IterR}
    \begin{align}\label{opS:proof_est_S_2_est_DDP}
        \left|D\Big((DP)^{-1}\circ K^\le \circ R^m(x)\Big)\right|_\Gamma &\le \left|D\Big((DP)^{-1}\circ K^\le\Big) \circ R^m(x)\right|_\Gamma \left|DR^m(x)\right|_\Gamma\nonumber\\
        &\le C\left|\mathrm{Id}\right|_\Gamma|R^m(x)|^{N-2}\left(1 + (m-1)a^*|x|^{N-1}\right)^{A^\Gamma \over a^*}\nonumber\\
        &\le C\left|\mathrm{Id}\right|_\Gamma {|x|^{N-2} \over \left(1 + ma^*|x|^{N-1}\right)^{{N-2 \over N-1} -{A^\Gamma \over a^*}}},
    \end{align}
    where, in the last line of the above inequalities, we used the trivial estimate $\left(1 + (m-1)a^*|x|^{N-1}\right)^{{A^\Gamma \over a^*}} \le \left(1 + ma^*|x|^{N-1}\right)^{{A^\Gamma \over a^*}}$. It remains to analyze $|T\circ R^j(x)|$. Remembering that $T \in \mathcal{X}_{\ell, \Gamma}$ and using Lemma \ref{proof:ThmPS_IterR}, we can see that 
    \begin{equation}\label{opS:proof_est_S2_est_TRj}
        |T\circ R^j(x)| \le |T|_{\ell, \Gamma} |R^j(x)|^\ell \le |T|_{\ell, \Gamma}{|x|^\ell \over (1 + ja^* |x|^{N-1})^{\ell \over N-1}}.
    \end{equation}   
    Replacing the inequalities~\eqref{opS:proof_est_S2_est_PiPi},~\eqref{opS:proof_est_S_2_est_DDP} and~\eqref{opS:proof_est_S2_est_TRj} into~\eqref{opS:proof_est_S2_1}, we obtain that 
    \begin{align}\label{opS:proof_est_S2_2}
        \left|\mathcal{S}_2(x)\right|_\Gamma &\le C \left|\mathrm{Id}\right|^3_\Gamma|T|_{\ell, \Gamma} \sum_{j=0}^\infty {|x|^{\ell +N-2}\over \left(1 +ja^*|x|^{N-1}\right)^{{\ell \over N-1} - {C^\Gamma \over a^*}}} \nonumber\\
        &\times\left(\sum_{m=0}^j {1 \over \left(1 + m a^*|x|^{N-1}\right)^{{N-2 \over N-1} - {A^\Gamma \over a^*}}}\right).
    \end{align}

    Now, we denote 
    \begin{equation}\label{opS:proof_alpha}
        \alpha = {N-2 \over N-1} - {A^\Gamma \over a^*}.
    \end{equation}
    In order to provide an upper bound for the sum in the last line of~\eqref{opS:proof_est_S2_2}, we need to consider separately the cases $\alpha\le 0$, $0 < \alpha<1$, and $\alpha>1$. 
     We point out that, taking $\varrho$ and $\gamma$ suitably small, we can assume, without loss of generality, that $\alpha\ne 1$. In the above-mentioned three cases, we provide estimates systematically using the Cauchy–Maclaurin sum–integral inequality without explicitly mentioning it each time.

    \textit{Case $\alpha \le 0$.} We observe that $\left(1 + m a^*|x|^{N-1}\right)^{-\alpha}$ is a monotone increasing function in the $m$ variable. Hence, we can write  
    \begin{equation*}
        \sum_{m=0}^j \left(1 + m a^*|x|^{N-1}\right)^{-\alpha} = \sum_{m=0}^{j-1} \left(1 + m a^*|x|^{N-1}\right)^{-\alpha} + \left(1 + j a^*|x|^{N-1}\right)^{-\alpha}
    \end{equation*}
    and we can estimate the sum on the right-hand side of the latter as
    \begin{multline*}
        \sum_{m=0}^{j-1} \left(1 + m a^*|x|^{N-1}\right)^{-\alpha} \le \int_0^j \left(1 + m a^*|x|^{N-1}\right)^{-\alpha} dm \\ = {1 \over a^*|x|^{N-1}}\int_1^{1 + ja^*|x|^{N-1}} u^{-\alpha}du
        \le {1 \over a^*|x|^{N-1}(1-\alpha)}(1 + ja^*|x|^{N-1})^{1 - \alpha}. 
    \end{multline*}
   Hence, for $\varrho$ small enough, we obtain that
   \begin{equation}\label{opS:est_S2_sum_m_alphale0}
       \sum_{m=0}^j \left(1 + m a^*|x|^{N-1}\right)^{-\alpha} \le {2 \over a^*|x|^{N-1}(1-\alpha)}(1 + ja^*|x|^{N-1})^{1 - \alpha},
   \end{equation}
    for all $x \in \Omega(\varrho, \gamma)$.

   \textit{Case $0 < \alpha <1$.} In this case $\left(1 + m a^*|x|^{N-1}\right)^{-\alpha}$ is a monotone decreasing function. Thus, we have that 
    \begin{align}\label{opS:est_S2_sum_m_0alpha1}
        \sum_{m=0}^j \left(1 + m a^*|x|^{N-1}\right)^{-\alpha} &\le 1 + \int_0^j \left(1 + m a^*|x|^{N-1}\right)^{-\alpha} dm \\
        &=  1+ {1 \over a^*|x|^{N-1}}\int_1^{1 + ja^*|x|^{N-1}} u^{-\alpha}du\nonumber\\
        &\le 1 + {1 \over a^*|x|^{N-1}(1-\alpha)}(1 + ja^*|x|^{N-1})^{1 - \alpha}\nonumber\\
        &\le  {2 \over a^*|x|^{N-1}(1-\alpha)}(1 + ja^*|x|^{N-1})^{1 - \alpha}
    \end{align}
    for all $x \in \Omega(\varrho, \gamma)$ and $\varrho$ small enough.
    
    \textit{Case $\alpha >1$.} Here again, $\left(1 + m a^*|x|^{N-1}\right)^{-\alpha}$ is a monotone decreasing function. Similarly to the previous cases, for all $x \in \Omega(\varrho, \gamma)$
    \begin{align}\label{opS:est_S2_sum_m_alpha>1}
        \sum_{m=0}^j \left(1 + m a^*|x|^{N-1}\right)^{-\alpha} &\le 1 + \int_0^j \left(1 + m a^*|x|^{N-1}\right)^{-\alpha} dm\nonumber\\
        &\le 1 - {1 \over a^*|x|^{N-1}(\alpha-1)}\Big[(1 + ja^*|x|^{N-1})^{1 - \alpha} - 1 \Big] \nonumber \\
        &\le 1 + {1 \over a^*|x|^{N-1}(\alpha-1)} \le {2 \over a^*|x|^{N-1}(\alpha-1)}
    \end{align}
    for $\varrho$ small enough. We point out that in the above estimates we used the trivial inequality $- {1 \over a^*|x|^{N-1}(\alpha-1)}(1 + ja^*|x|^{N-1})^{1 - \alpha} \le 0$.

    Finally, we derive an upper bound for $\left|\mathcal{S}_2(x)\right|_\Gamma$ in~\eqref{opS:proof_est_S2_2} for all $x \in \Omega(\varrho, \gamma)$. We point out that the estimate provided for the sum in the last line of~\eqref{opS:proof_est_S2_2} are the same for $\alpha \le 0$ and $0<\alpha <1$. Thus, here, we only consider the cases $\alpha <1$ and $\alpha >1$ separately. 

    \textit{Estimate $\left|\mathcal{S}_2(x)\right|_\Gamma$ case $\alpha <1$.} Combining~\eqref{opS:est_S2_sum_m_alphale0} and~\eqref{opS:est_S2_sum_m_0alpha1} with~\eqref{opS:proof_est_S2_2}, rememebering the definition of $\alpha$ in~\eqref{opS:proof_alpha} and using hypothesis~\eqref{opS:lemma_hyp_ell_Gamma}, we have that 
    \begin{equation}\label{opS:proof_stima_finale_S2_1}
        \left|\mathcal{S}_2(x)\right|_\Gamma \le K |T|_{\ell, \Gamma} \sum_{j=0}^\infty {|x|^{\ell -1}\over \left(1 +ja^*|x|^{N-1}\right)^{{\ell -1 \over N-1} - {C^\Gamma + A^\Gamma\over a^*}}} \le C|T|_{\ell, \Gamma} |x|^{\ell -N},
    \end{equation}
    for all $x \in \Omega(\varrho, \gamma)$.

    \textit{Estimate $\left|\mathcal{S}_2(x)\right|_\Gamma$ case $\alpha >1$.} Similarly to the previous case, thanks to~\eqref{opS:est_S2_sum_m_alpha>1}, and hypothesis~\eqref{opS:lemma_hyp_ell_Gamma}
    \begin{equation}\label{opS:proof_stima_finale_S2_2}
        \left|\mathcal{S}_2(x)\right|_\Gamma \le C |T|_{\ell, \Gamma} \sum_{j=0}^\infty {|x|^{\ell -1}\over \left(1 +ja^*|x|^{N-1}\right)^{{\ell \over N-1} - {C^\Gamma \over a^*}}} \le C|T|_{\ell, \Gamma} |x|^{\ell -N},
    \end{equation}
    for all $x \in \Omega(\varrho, \gamma)$. We poin out that $A^\Gamma$, $C^\Gamma$ and $a^*>0$ and hence~\eqref{opS:lemma_hyp_ell_Gamma} implies ${\ell \over N-1} - {C^\Gamma \over a^*}>1$. Using~\eqref{opS:proof_stima_finale_S_NoGamma},~\eqref{opS:proof_DS},~\eqref{opS:proof_stima_finale_S1},~\eqref{opS:proof_stima_finale_S2_1}, and~\eqref{opS:proof_stima_finale_S2_2}, we conclude the proof of this lemma. 
\end{proof}

\subsection{The nonlinear operator $\mathcal{N}$}\label{sec:op_N}
This section is devoted to estimate the Lipschitz constant of the nonlinear operator $\mathcal{N}$, defined in~\eqref{proof:Thm_PS_def_N}, on the Banach spaces $\mathcal{X}_{\ell, \Gamma}$ in~\eqref{def:XlGamma}. 

Given $\varsigma >0$,  we denote by $B_{\varsigma, \ell -N+1, \Gamma}$ the closed balls centered at the origin of radius $\varsigma$ contained in $\mathcal{X}_{\ell-N+1, \Gamma}$. 

\begin{lemma}\label{lemma:contr_E12_Gamma}
      We assume that $\ell \ge N$ and $\varsigma>0$ are suitably small in such a way that if $h \in B_{\varsigma, \ell -N+1, \Gamma}$, the range of $K^\le+h$ is contained in the domain of $F$. Then,  $\mathcal{N}\left(B_{\varsigma, \ell -N+1, \Gamma}\right)\subset \mathcal{X}_{\ell, \Gamma}$. Moreover, if $\ell \ge N+1$, there exists a positive constant $C$ such that 
     \begin{equation*}
        \mathrm{Lip} \, \mathcal{N}|_{B_{\varsigma, \ell -N+1, \Gamma}} \le C\varrho.
    \end{equation*}
\end{lemma}
\begin{proof}
We begin with the proof of the first part of this lemma.  Let $h \in \mathcal{X}_{\ell-N+1, \Gamma}$, remembering the definition~\eqref{proof:Thm_PS_def_N}, we can rewrite $\mathcal{N}$ as
\begin{equation*}
    \mathcal{N}(h) = N_1 + N_2 + N_3  ,
\end{equation*}
with
\begin{align*}
    N_1 &= P \circ K^\le - K^\le \circ R,\\
    N_2 &= G_\ell \circ (K^\le + h),\\
    N_3 &= P \circ (K^\le + h) - P \circ K^\le - \left(DP \circ K^\le \right)h =\int_0^1 D^2 P \circ (K^\le + \tau h)\, d\tau \cdot (h,h),
\end{align*}
where the right-hand side of the last equality of the latter stands for the element $h$ given twice as an argument of the symmetric bilinear form $\int_0^1 D^2 P \circ (K^\le + \tau h)d\tau$. Using~\eqref{def:F=P+G},~\eqref{proof:thm_PS_eq_P}, Propositions \ref{prop:NormGammaProdBil}, \ref{prop:NormGammaProd}, \ref{prop:cauchyBS_Decay}, \ref{cor:cauchyBS_Decay} and the condition over $\ell$, a straightforward computation shows that $N_1, \, N_2, \, N_3 \in \mathcal{X}_{\ell, \Gamma}$. This concludes the first part of the proof of this lemma.  It remains to estimate the Lipschitz constant of $\mathcal{N}$ restricted to $B_{\varsigma, \ell -N +1, \Gamma}$. To this end, we observe that
\begin{equation}\label{eq:fixed_point_N_no_Gamma}
   | \mathcal{N}(h)(x) - \mathcal{N}(g)(x)| \le C\varrho |h-g|_{\ell-N+1} |x|^\ell,
\end{equation}
for all $h$, $g \in B_{\varsigma, \ell -N +1, \Gamma}$ and $x \in \Omega(\varrho, \gamma)$.
The proof of the latter is similar to that in~\cite{BFM1, BFMARMA}. For this reason, it is omitted.

Remembering the definition of the Banach spaces $\mathcal{X}_{\ell, \Gamma}$ in Section \ref{sec:opS_properties}, given $h, g \in \mathcal{B}_{\varsigma, \ell - N+1, \Gamma}$, it suffices to analyze $D\left(\mathcal{N}(h) - \mathcal{N}(g)\right)$. To this end, we observe that 
    \begin{equation}\label{eq:Lip_DN_Gamma}
        D\left(\mathcal{N}(h) - \mathcal{N}(g)\right) = D\mathcal{N}(h) \left(Dh - Dg\right) + \left(D\mathcal{N}(h) - D\mathcal{N}(g)\right) Dg.
    \end{equation}

    We need to analyze the two terms on the right-hand side of the latter separately.
    
    Using Proposition \ref{prop:NormGammaProd}, we have that 
    \begin{align*}
        |D\mathcal{N}(h)(x) \left(Dh - Dg\right)(x)|_\Gamma &\le  |D\mathcal{N}(h)(x)|_\Gamma  |Dh(x) - Dg(x)|_\Gamma \\
        &\le |x|^{\ell-1} |\mathcal{N}(h)|_{\ell, \Gamma} |h-g|_{\ell- N+1, \Gamma}|x|^{\ell-N},
    \end{align*}
    for all $x \in \Omega(\varrho, \gamma)$ and hence, thanks to the condition on $\ell$,
    \begin{equation}\label{est:DN_1}
        |D\mathcal{N}(h) \left(Dh - Dg\right)|_{\ell-1, \Gamma} \le \varrho^{\ell-N} |\mathcal{N}(h)|_{\ell, \Gamma} |h-g|_{\ell- N+1, \Gamma}.
    \end{equation}
    It remains to estimate the second term in the right-hand side of~\eqref{eq:Lip_DN_Gamma}. For this reason, we recall that 
    \begin{align*}
        \mathcal{N}(h) - \mathcal{N}(g) &= P \circ (K^\le + h) - \left(DP\circ K^\le\right) h - P \circ (K^\le + g) + \left(DP\circ K^\le\right) g \\
        &+G_\ell \circ (K^\le +h) - G_\ell \circ (K^\le +g).
    \end{align*}
    Differentiating the above equality and after some Taylor expansions, we obtain that 
    \begin{equation*}
        \left(D\mathcal{N}(h) - D\mathcal{N}(g)\right) Dg = T_1+T_2+T_3+T_4+T_5+T_6+T_7+T_8,
    \end{equation*}
    where
    \begin{align*}
        T_1 &= \int_0^1 D^2 P \circ (K^\le + g + \tau (h-g))d\tau \,(h-g)\, DK^\le  \, Dg,\\
        T_2&=\int_0^1\int_0^1 D^3 P \circ (K^\le + \tau g + s \tau(h-g)) ds \,\tau \,(h-g) d \tau \,h \,Dh\,Dg,\\
        T_3 &= \int_0^1 D^2 P \circ (K^\le + \tau g)d \tau \,(h-g)\, Dh\,Dg,\\
        T_4 &= \int_0^1 D^2 P \circ (K^\le + \tau g)d \tau \,g \,D(h-g)\,Dg,\\
        T_5 &= - \left(D^2 P\circ K^\le\right) \,DK^\le \,(h-g)\,Dg,\\
        T_6 &= \int_0^1 D^2 G_\ell \circ (K^\le + g + \tau (h-g)) d \tau \,(h-g)\, DK^\le\, Dg,\\
        T_7 &= \int_0^1 D^2 G_\ell \circ (K^\le +g + \tau (h-g)) d \tau \,(h-g) \,Dh\,Dg,\\
        T_8 &= DG_\ell \circ (K^\le + g) \,D(h-g)\,Dg.
    \end{align*}
    We need to estimate each $T_i$, with $i = 1,\dots,8$, separately. For this purpose, we observe that, for all $x \in \Omega(\varrho, \gamma)$, by Propositions~\ref{prop:NormGammaProdBil}, \ref{prop:NormGammaProd}, \ref{prop:cauchyBS_Decay}, \ref{cor:cauchyBS_Decay}, hypotheses~\eqref{thm:aporsteriori_sol_Hyp_decay},~\eqref{thm:aporsteriori_sol_Hyp_decay_2} and the condition over $\ell$, we have that 
    \begin{align}\label{est:contr_BSGamma_T1}
        |T_1(x)|_\Gamma &\le C|x|^{N-2} |h-g|_{\ell-N+1, \Gamma}|x|^{\ell-N+1}|g|_{\ell-N+1, \Gamma} |x|^{\ell-N} \nonumber\\
        &\le C|x|^{\ell-1}\varrho^{\ell-N}|h-g|_{\ell-N+1, \Gamma},
    \end{align}
    for all $x \in \Omega(\varrho, \gamma)$, where we recall that $C$ denotes a generic positive constant. Similarly, one can verify that 
    \begin{equation}\label{est:contr_BSGamma_T2345}
        \begin{aligned}
            |T_2(x)|_\Gamma &\le C|x|^{\ell-1} \varrho^{3(\ell-N)}|h-g|_{\ell-N+1, \Gamma},\\
            |T_3(x)|_\Gamma &\le C|x|^{\ell-1} \varrho^{2(\ell-N)}|h-g|_{\ell-N+1, \Gamma}, \\
            |T_4(x)|_\Gamma &\le C|x|^{\ell-1} \varrho^{2(\ell-N)}|h-g|_{\ell-N+1, \Gamma}, \\
            |T_5(x)|_\Gamma &\le C|x|^{\ell-1} \varrho^{\ell-N}|h-g|_{\ell-N+1, \Gamma} ,
        \end{aligned}
    \end{equation}
    for all $x \in \Omega(\varrho, \gamma)$. Concerning $T_6$, we observe that, using the definition of $G_\ell$ in~\eqref{def:F=P+G}, Propositions~\ref{prop:NormGammaProdBil}, \ref{prop:NormGammaProd}, \ref{prop:cauchyBS_Decay}, \ref{cor:cauchyBS_Decay}, hypotheses~\eqref{thm:aporsteriori_sol_Hyp_decay},~\eqref{thm:aporsteriori_sol_Hyp_decay_2} and the condition over $\ell$, we obtain that 
    \begin{align}\label{est:contr_BSGamma_T6}
        |T_6(x)|_\Gamma &\le C|x|^{\ell-2} |h-g|_{\ell-N+1, \Gamma}|x|^{\ell-N+1}|g|_{\ell-N+1, \Gamma} |x|^{\ell-N} \nonumber\\
        &\le C|x|^{\ell-1}\varrho^{2(\ell-N)}|h-g|_{\ell-N+1, \Gamma},
    \end{align}
    for all $x \in \Omega(\varrho, \gamma)$. Similar to the previous case, one can prove that 
    \begin{equation}\label{est:contr_BSGamma_T78}
        \begin{aligned}
            |T_7(x)|_\Gamma &\le C|x|^{\ell-1} \varrho^{3(\ell-N)}|h-g|_{\ell-N+1, \Gamma},\\
            |T_8(x)|_\Gamma &\le C|x|^{\ell-1} \varrho^{2(\ell-N)}|h-g|_{\ell-N+1, \Gamma},
        \end{aligned}
    \end{equation}
    for all $x \in \Omega(\varrho, \gamma)$.

    Now, combining~\eqref{est:contr_BSGamma_T1},~\eqref{est:contr_BSGamma_T2345},~\eqref{est:contr_BSGamma_T6} and~\eqref{est:contr_BSGamma_T78}, we conclude that 
    \begin{equation}\label{est:DN_2}
        |\left(D\mathcal{N}(h) - D\mathcal{N}(g)\right) Dg|_{\ell-1, \Gamma} \le \varrho^{\ell-N} |h-g|_{\ell-N+1, \Gamma},
    \end{equation}
    for $\varrho$ suitably small. Thus, using~\eqref{eq:fixed_point_N_no_Gamma},~\eqref{eq:Lip_DN_Gamma},~\eqref{est:DN_1}, and~\eqref{est:DN_2}, and the condition on $\ell$, the Proposition is established.
    \end{proof}

Assuming the hypotheses of Theorem~\ref{ThmMapsPS}, combining Lemmas~\ref{lemma:S_Inv} and~\ref{lemma:contr_E12_Gamma}, and using a standard argument, one can prove the existence of a $\varsigma_0>0$ in such a way that $\mathcal{S}\circ \mathcal{N}:B_{\varsigma_0, \ell-n+1, \Gamma} \to B_{\varsigma_0, \ell-n+1, \Gamma} $ is a contraction. Thus, there exists a unique solution $K^> \in B_{\varsigma_0, \ell-n+1, \Gamma}$ of equation~\eqref{proof:Thm_PS_eq_FP}. This concludes the proof of Theorem \ref{ThmMapsPS}.

\appendix
\section{Banach spaces and analytic functions}\label{A}
This section is divided into two subsections.  First, we introduce the definition of a complexification of a Banach space. Then, we define the concept of analytic function for a mapping defined on Banach spaces. For more details, we refer to~\cite{MR0894477,MR1688213,MR0273396}.

\subsection{Complexification of Banach spaces}\label{App: Complexification}
 First, let us introduce the definition of a complexification of real Banach spaces.  In this first part, we follow the lines of~\cite{MR1688213}.  In order to clarify what we mean by a complexification, we begin with the following
\begin{definition}
\label{complex}
A complex vector space $\C E$ is a complexification of a real vector space $E$ if the following two conditions hold:
\begin{enumerate}
\item there is a one-to-one real-linear map $j : E \to \C E$
\item $\mathrm{complex-span}\left(j(E)\right) = \C E$.
\end{enumerate}
\end{definition}
We observe that, up to complex isomorphism, a real vector space has just one complexification.  Let us consider the following descriptions.  If $E$ is a real vector space, we can make $E\times E$ into a complex vector field by defining 
\begin{align*}
(x,y) + (u,v) =& (x+u, y+v)   &&\mbox{for all $x$, $y$, $u$, $v \in E$},\\
(\alpha + i \beta)(x,y) =& (\alpha x- \beta y, \beta x+ \alpha y)   &&\mbox{for all $x$, $y \in E$, and $\alpha$, $\beta \in \R$}.
\end{align*}
It is straightforward to verify that the map $j: E \to E\times E$, $x \to (x,0)$ satisfies the conditions of Definition \ref{complex}.  Hence, the previous complex vector space $E \times E$ is a complexification of $E$. It is convenient to denote it by
\begin{equation}
\label{EiE}
\C E = E \oplus i E
\end{equation}
and to suppress reference to $j$ by writing $z = x + iy$ for the element $(x,y) = j(x) + ij(y)$. 
It is natural to write $x = \mathrm{Re}\,(x)$ and $y = \mathrm{Im}\,(z)$. In the reference~\cite{MR1688213}, the authors provide two more equivalent descriptions. 

Now, concerning the complexification of real Banach spaces, we have the following
\begin{definition}
\label{reasonablecomplex}
Let $E$ be a real Banach space. We say that a norm $|\cdot|$ on the complexification $\C E$ is \emph{reasonable} if
\begin{enumerate}
\item[3.] $|j(x)| = |x|$ for all $x \in E$
\item[4.] $|x+iy| = |x-iy|$ for all $x$, $y \in E$.
\end{enumerate}
\end{definition}
We point out that, in the latter, we use the same notation for the norms in $E$ and $\C E$.  When $\C E$ is equipped with such a norm, we call it a reasonable complexification of $E$.  
\begin{proposition}
Let $\tilde E$ be a complexification of the real Banach space $E$. Among all the reasonable complexification norms in $\tilde E$, the smallest is given by
\begin{equation*}
|x + iy|_T = \sup_{0 \le t\le 2\pi}|x \cos t - y \sin t|. 
\end{equation*} 
All other complexification norms $|\cdot|$ on $\tilde E$ are equivalent to $|\cdot|_T$.  Indeed, for any $x$, $y \in E$,
\begin{equation*}
|x + iy|_T \le |x+iy| \le 2 |x + iy|_T. 
\end{equation*} 
\end{proposition}
\begin{proof}
We refer to~\cite{MR1688213}.
\end{proof}

Let us define the following norm
\begin{equation}
\label{analynorm2}
\|x+iy\| = \max\{|x|, |y|\}
\end{equation}
for all $x+iy \in \tilde E$. 

\begin{proposition}
\label{ReasNorm}
The norm $\|\cdot\| $ defined by~\eqref{analynorm2} is resonable. 
\end{proposition}
\begin{proof}
We have to verify that this norm satisfies the properties of Definition \ref{reasonablecomplex}. In fact,
\begin{eqnarray*}
\|j(x)\| &=& \|(x,0)\| = \|x\|,\\
\|x+iy\| &=& \max \{|x|, |y|\} = \max \{|x|, |-y|\} = \|x-iy\| 
\end{eqnarray*}
for all $(x,y) \in E\times E$.
\end{proof}

\subsection{Analytic functions in Banach spaces}

In this section,  we consider functions defined on Banach spaces. We introduce the definition of analytic functions and some properties. We refer to~\cite{MR0894477,MR0273396} for more details. 
First, let us recall what we mean by continuous functions.  To this end,  in this section, we consider the following Banach spaces $\left(E_1, |\cdot|\right)$,  $\left(E_2, |\cdot|\right)$, and an open subset $U \subset E_1$.
\begin{definition}
Let $\{x_n\}_{n \in \N} \subset U$ and $x \in U$.  We say that $x_n \to x$ strongly in $U$ if
\begin{equation*}
\lim_{n \to +\infty}|x_n -x| =0.
\end{equation*}
\end{definition}
We have the following definition
\begin{definition}[Continuous functions]
A map $f:U \to E_2$ is continuous on $U$ if it maps strongly convergent sequences in $U$ into strongly convergent sequences in $E_2$. That is, if $x_n \to x$ strongly in $U$, then $f(x_n) \to f(x)$ strongly in $E_2$.
\end{definition}

Now, we turn to differentiable maps, and we consider the following
\begin{definition}[Differentiable functions]
A map $f:U \to E_2$  is differentiable at $x \in U$, if there exists a bounded linear map
\begin{equation*}
d_xf:E_1\longrightarrow E_2
\end{equation*}
such that 
\begin{equation*}
|f(x+th) - f(x) - d_xf(h)| = o(|h|).
\end{equation*}
That is, for every $\varepsilon >0$ there is a $\delta >0$ such that $|f(x+th) - f(x) - d_xf(h)| \le \varepsilon |h|$ for all $h$ with $|h|<\delta$.  The linear map is uniquely determined and is called the derivative of $f$ at $x$. 
\end{definition}

The map $f$ is differentiable on $U$, if it is differentiable at each point in $U$.  In this case, the derivative is a map from $U$ into the Banach space $L(E_1,E_2)$ of all bounded linear maps from $E_1$ into $E_2$, denoted by $df$. If this map is continuous, then $f$ is continuously differentiable, or of class $C^1$, on $U$.  Concerning the higher derivatives (that are defined inductively) and the definition of $C^\infty$ functions, we refer to~\cite{MR0894477}, as they are quite intuitive.  

We want to recall the Taylor formula with integral reminder.  To this end, we consider $f:U \to E_2$ and we assume that $f$ is $p$ times continuously differentiable, $p \ge 1$, and if the segment $x + th$, $0 \le t \le1$, is contained in $U$, then
\begin{eqnarray}
f(x+h) &=& f(x) + d_x f(h) +...+{1 \over (p-1)!} d_x^{p-1}f(h,...,h) \nonumber\\
\label{TaylorFormulaRem}
&+& \int_0^1 {(1-t)^{p-1} \over (p-1)!} d^p_{x+th}f(h,...,h) dt.
\end{eqnarray}

If in addition, $f$ is smooth, and the integral remainder converges to $0$ as $p$ tends to infinity uniformly in some ball $|h| <r'$, then $f$ admits a Taylor series expansion
\begin{equation*}
f(x+h) = \sum_{k \ge 0}{1\over k!} d_x^k f(h,...,h)
\end{equation*}
at $x$ in this ball. We refer to~\cite{MR0894477} for more details. 

Finally, we can introduce the definition of analytic function.  
 
 \begin{definition}[Analytic functions]
Let $f:U \to E_2$ be a function from an open subset $U$ of a complex Banach space $E_1$ into a complex Banach space $E_2$.  We say that $f$ is analytic on $U$ if it is continuously differentiable on $U$. 
\end{definition}

In reference~\cite{MR0894477}, the authors say that it is convenient to introduce another notion of analyticity.  
For this purpose, for a given a complex Banach space $E_2$ we define
\begin{equation*}
E_2^* =\{\mbox{$f:E_2 \to \C$ such that $f$ is a bounded linear function}\}.
\end{equation*}
\begin{definition}[Weakly analytic function]
\label{weakanalytic}
Let $E_1$ and $E_2$ be complex Banach spaces and $U\subset E_1$ an open subset.  The map $f : U \to E_2$ is weakly analytic on $U$, if for each $x \in U$, $h \in E_1$ and $L \in E_2^*$, the function
\begin{equation*}
z \to Lf(x + zh)
\end{equation*}
is analytic in some neighborhood of the origin in $\C$ in the usual sense of one complex variable.  The radius of weak analyticity of $f$ at $x$ is the supremum of all $r \ge 0$ such that the above function is defined and analytic in the disk $|z| <1$ for all $L \in E_2^*$ and $h \in E_1$ with $|h|<r$. 
\end{definition}

We point out that the radius $r$ of weak analyticity at $x$ is equal to the distance $\rho$ of $x$ to the boundary of $U$ (we refer to~\cite{MR0894477} for the proof).
The following theorem provides equivalent definitions of analytic function. 

\begin{theorem}
\label{TheoremAnaiytic}
Let $f:U \to E_2$ be a map from an open subset $U$ of a complex Banach space $E_1$ into a complex Banach space $E_2$. Then the following statements are equivalent
\begin{enumerate}
\item $f$ is analytic on $U$
\item $f$ is locally bounded and weakly analytic on $U$.
\item $f$ is infinitely often differentiable on $U$, and it is represented by its Taylor series into a neighborhood of each point in $U$.
\end{enumerate}
\end{theorem}
\begin{proof}
We refer to~\cite{MR0894477} for the proof. 
\end{proof}

Furthermore, we have the following properties for analytic functions defined on Banach spaces

\begin{proposition}\label{prop:CauchyFormula}
    Let $f:U \to E_2$ be a map from an open subset $U$ of a complex Banach space $E_1$ into a complex Banach space $E_2$. Suppose $f$ is weakly analytic and continuous on $U$. Then, for every $x \in U$ and $h \in E_1$,
    \begin{equation*}
        f(x + zh) = {1 \over 2 \pi i} \int_{|\zeta| = \rho} {f(x + \zeta h) \over \zeta - z} d \zeta, \qquad |z| < \rho < {r \over |h|},
    \end{equation*}
    where $r$ is the radius of weak analyticity of $f$ at $x$.
\end{proposition}
\begin{proof}
    We refer to~\cite{MR0894477}.
\end{proof}

For each $x \in E_1$ and $\sigma>0$, we denote by $B_\sigma(x)$ the open ball of radius $\sigma$ centered at $x \in E_1$. 
\begin{proposition}\label{prop:CauchyEstimate}
Let $f:U \to E_2$ be a map from an open subset $U$ of a complex Banach space $E_1$ into a complex Banach space $E_2$. For all $x \in U$, let $\sigma>0$ in such a way that $B_\sigma(x) \subset U$. Then, for all positive integers $k \ge 1$, $v_i \in E_1$ with $1 \le i \le k$,
\begin{equation*}
    \left|D^kf(x)(v_1,\dots,v_k)\right| \le  {k^k\over \sigma^k} \left(\displaystyle \sup_{w \in \bar B_\sigma (x)} \left|f(w)\right| \right)|v_1| \cdots |v_k|.
\end{equation*}
\end{proposition}
\begin{proof}
  We first observe that if any of the vectors $v_i$ is zero, then the inequality holds trivially. We assume that $v_i \ne 0$ for all $1 \le i \le k$. 
   We consider $z_i \in \C$ in such a way that $|z_i| \le {\sigma \over k |v_i|}$ for all $1 \le i \le k$. Using Proposition \ref{prop:CauchyFormula}, one can see that 
   \begin{equation*}
       f\left(x + \sum_{i=1}^k z_i v_i\right) = \left({1 \over 2 \pi i}\right)^k \int_{|\zeta_1|={\sigma \over k|v_1|}} \cdots \int_{|\zeta_k|={\sigma \over k|v_k|}} {f\left(x + \sum_{i=1}^k \zeta_i v_i\right) \over \prod_{i=1}^k\left(\zeta_i - z_i\right)}d\zeta_k \cdot \cdot \cdot d\zeta_1.
   \end{equation*}
   We want to point out that, in contrast to the statement of Proposition \ref{prop:CauchyFormula}, we may assume, without loss of generality, that $|z_i| \le {\sigma \over k|v_i|}$. This is justified by the fact that $B_\sigma(x) \subset U$; and if the inclusion is not strict, we can always reduce it to a smaller ball.

   By differentiating with respect to each variable $z_i$ and evaluating the $k$-th order derivative at $z_i = 0$ for all $i = 1, \dots, k$, we obtain
    \begin{equation*}
        D^kf(x)\left(v_1,...,v_k\right) =\left({1 \over 2 \pi i}\right)^k  \int_{|\zeta_1|={\sigma \over k|v_1|}} \cdots \int_{|\zeta_k|={\sigma \over k|v_k|}} {f\left(x + \sum_{i=1}^k \zeta_i v_i\right) \over \prod_{i=1}^k\zeta_i^2} d\zeta_k \cdot \cdot \cdot d\zeta_1.
    \end{equation*}
   Now, we can estimate the left-hand side of the latter as follows
   \begin{align*}
       \left|D^kf(x)\left(v_1,...,v_k\right)\right| &\le \left({1 \over 2 \pi}\right)^k  \int_{|\zeta_1|={\sigma \over k|v_1|}} \cdots \int_{|\zeta_k|={\sigma \over k|v_k|}} {\left|f\left(x + \sum_{i=1}^k \zeta_i v_i\right)\right| \over \prod_{i=1}^k\left({\sigma \over k|v_i|}\right)^{2}} d\zeta_k \cdot \cdot \cdot d\zeta_1\\
       &\le {k^k\over \sigma^k} \left(\displaystyle \sup_{w \in \bar B_\sigma (x)} \left|f(w)\right| \right)|v_1| \cdots |v_k|.
   \end{align*}
\end{proof}

We also introduce the notion of a real-analytic map. Let $E_1$, $E_2$ be real Banach spaces, $\C E_1$, $\C E_2$ their complexification, and $U \subset E_1$  an open subset. 

 \begin{definition}[Real-analytic functions]
A map $f:U \to E_2$ is real-analytic on $U$, if for each point in $U$ there is a neighborhood $V \subset \C E_1$ and an analytic map $g:V \to \C E_2$, such that 
\begin{equation*}
f=g \quad \mbox{on} \quad U \cap V.
\end{equation*}
\end{definition}
A real-analytic map can be expanded into a Taylor series with real coefficients in a ball at each point. The converse is also true (see always~\cite{MR0894477}).  Moreover, we have the following theorem
\begin{theorem}
Every analytic mapping $f:U \to E_2$ from an open subset $U$ of a real Banach space $E_1$, to a complex Banach space $E_2$ may be extended to an analytic map $\hat f$, defined on some open $V \subset \C E_1$ containing $U$.
\end{theorem}
\begin{proof}
We refer to~\cite{MR0273396} for the proof. 
\end{proof}

\section{Decay functions}\label{B}
The aim of this appendix is to collect several properties of spaces of decay functions. Section \ref{LinearkLinear} is dedicated to a series of algebraic properties of $k$-linear maps with decay, which we state and refer to~\cite{FdlLM11} for the proofs. In Section \ref{AnalyticFunctions}, we prove some properties satisfied by homogeneous analytic functions with decays.

\subsection{$k$-linear maps with decay}\label{LinearkLinear}

Given $\mathcal{X}$ and $\mathcal{Y}$ be Banach spaces, we denote by $L^k(\mathcal{X},\mathcal{Y})$ the space of the $k$-linear maps from $\mathcal{X}$ to $\mathcal{Y}$. We recall that for non-symmetric $k$-linear maps, there are $k$ possible identifications
\begin{equation}\label{iotaj}
\begin{aligned}
    &\iota_j:L^k(\mathcal{X}, \mathcal{Y}) \longrightarrow L\left(\mathcal{X}, L^{k-1}(\mathcal{X}, \mathcal{Y})\right)\\
    &\iota_j (A)(v) (u_1,...,u_{j-1}, u_{j+1}, ..., u_k) = A  (u_1,...,u_{j-1}, v, u_{j+1}, ..., u_k)
    \end{aligned}
\end{equation}
for all $1 \le j\le k$. Now, we denote by $\mathcal{X} = \{\mathcal{X}_i\}_{i \in \Z^d}$ and $\mathcal{Y} = \{\mathcal{Y}_i\}_{i \in \Z^d}$ two families of Banach spaces. The space $\ell^\infty(\mathcal{X})$ is introduced in Definition \ref{EllInfty}. We use the following identifications $L^k(\ell^\infty(\mathcal{X}), \ell^\infty(\mathcal{Y})) \cong \ell^\infty(L^k(\ell^\infty(\mathcal{X}), \mathcal{Y}))$, and, thanks to~\eqref{iotaj}, $L^k(\ell^\infty(\mathcal{X}), \ell^\infty(\mathcal{Y})) \cong L(\ell^\infty(\mathcal{X}), \ell^\infty(L^{k-1}(\ell^\infty(\mathcal{X}), \mathcal{Y})))$.
Following~\cite{FdlLM11}, we define the space of $k$-linear maps with decay $\Gamma$ as
\begin{align*}
    L^k_\Gamma\left(\ell^\infty(\mathcal{X}), \ell^\infty(\mathcal{Y})\right) &= \Big\{A \in L^k\left(\ell^\infty(\mathcal{X}), \ell^\infty(\mathcal{Y})\right) \hspace{2mm} : \\
    &\hspace{2mm} \iota_m(A) \in L_\Gamma(\ell^\infty(\mathcal{X}), \ell^\infty(L^{k-1}(\ell^\infty(\mathcal{X}), \mathcal{Y}))), \hspace{2mm} m=1,...,k   \Big\}
\end{align*}
where the space $L_\Gamma$ is defined by~\eqref{LGamma}. This space is endowed with the following norm
\begin{equation}\label{normkGamma}
    |A|_\Gamma = \max \{|A|, \gamma(A)\}
\end{equation}
where
\begin{equation*}
 \gamma(A) = \max_{1 \le m \le k} \sup_{i, j \in \Z^d} \sup_{\substack{|u|\le 1 \\ \pi_l u=0, l \ne j}} \sup_{\substack{|v_p|\le 1 \\ 2 \le p \le k}} |\iota_m(A)_i(u)(v_2,...,v_k)|\Gamma(i - j)^{-1}.
\end{equation*}
With the norm $|A|_\Gamma$, the space $L_\Gamma^k$ is a Banach space. 

We want to state some algebraic properties of the above norm. For this purpose, we need to introduce the following notation. Given $k \ge 1$ we denote by $S_k$ the group of permutations. For a suitable set $E$, $v = (v_1, ..., v_k) \in E\times \cdot \cdot \cdot \times E$ and $\tau \in S_k$, we define $\tau(v) = (v_{\tau(1)},...,v_{\tau(k)})$.

\begin{proposition}\label{prop:NormGammaProdBil}
Let $A \in L^k_\Gamma\left(\ell^\infty(\mathcal{X}), \ell^\infty(\mathcal{Y})\right)$, and $u\in\ell^\infty(\mathcal{X})$. Then, for any $\tau \in S_k$ the map $B_{\tau, u} : \underbrace{\ell^\infty(\mathcal{X})\times \cdot \cdot \cdot \times\ell^\infty(\mathcal{X})}_{(k-1) \hspace{2mm} times} \to \ell^\infty(\mathcal{Y})$ defined by
\begin{equation*}
    B_{\tau, u}(v_1,...,v_{k-1}) = A \left(\tau(v_1, ...,v_{k-1}, u)\right)
\end{equation*}
belongs to $L^{k-1}_\Gamma\left(\ell^\infty(\mathcal{X}), \ell^\infty(\mathcal{Y})\right)$. Moreover
\begin{equation*}
    |B_{\tau, u}|_\Gamma \le |A|_\Gamma |u|.
\end{equation*}
\end{proposition}

Let $\mathcal{Z} = \{\mathcal{Z}_i\}_{i \in \Z^d}$ be a third family of Banach spaces. 

\begin{proposition}\label{prop:NormGammaProd}
If $A \in L^k_\Gamma\left(\ell^\infty(\mathcal{Y}), \ell^\infty(\mathcal{Z})\right)$ and $B_j \in L^{l_j}_\Gamma\left(\ell^\infty(\mathcal{X}), \ell^\infty(\mathcal{Y})\right)$, for $j = 1,...,k$, then the composition $AB_1\cdot\cdot\cdot B_k \in L^{l_1+\cdot\cdot\cdot l_k}_\Gamma\left(\ell^\infty(\mathcal{X}), \ell^\infty(\mathcal{Z})\right)$ and 
\begin{equation*}
    |AB_1\cdot\cdot\cdot B_k|_\Gamma \le |A|_\Gamma |B_1|_\Gamma \cdot\cdot\cdot |B_k|_\Gamma.
\end{equation*}
\end{proposition}


As a consequence of the previous proposition, we have the following
\begin{proposition}\label{app:prop_prod_with_identity}
    Let $A \in L_\Gamma\left(\ell^\infty(\mathcal{X}), \ell^\infty(\mathcal{X})\right)$ and $B \in L_\Gamma\left(\ell^\infty(\mathcal{X}), \ell^\infty(\mathcal{X})\right)$. Then
    \begin{equation}\label{app:prop_prod_with_identity_1}
        \left| (\mathrm{Id} + A)B \right|_\Gamma \le \left(  1 + |A|_\Gamma\right)|B|_\Gamma.
    \end{equation}
    Moreover, for any fixed $j \in \N$, letting $A_m \in L_\Gamma\left(\ell^\infty(\mathcal{X}), \ell^\infty(\mathcal{X})\right)$ for all $0 \le m \le j$, then
    \begin{equation}\label{app:prop_prod_with_identity_2}
        \left|\prod_{m=0}^j(\mathrm{Id} + A_m)\right|_\Gamma \le |\mathrm{Id}|_\Gamma \prod_{m=0}^j \left(1 + |A_m|_\Gamma\right).
    \end{equation}
\end{proposition}
\begin{proof}
    Noticing that $(\mathrm{Id} + A)B = B + AB$, the proof of~\eqref{app:prop_prod_with_identity_1} is a straighforward consequence of Proposition \ref{prop:NormGammaProd}. Concerning~\eqref{app:prop_prod_with_identity_2}, using~\eqref{app:prop_prod_with_identity_1} one can prove by induction that 
    \begin{align*}
        \left|\prod_{m=0}^j(\mathrm{Id} + A_m)\right|_\Gamma &\le  \left|\mathrm{Id} + A_0\right|_\Gamma\prod_{m=1}^j \left(1 + |A_m|_\Gamma\right)\\
        &\le  \left(\left|\mathrm{Id}\right|_\Gamma + \left|A_0\right|_\Gamma\right)\prod_{m=1}^j \left(1 + |A_m|_\Gamma\right)
    \end{align*}
    For the inequality in the second line of the latter, we used the trivial estimate $\left|\mathrm{Id} + A_0\right|_\Gamma \le \left(\left|\mathrm{Id}\right|_\Gamma + \left|A_0\right|_\Gamma\right)$. Now, using that $\left|\mathrm{Id}\right|_\Gamma>1$, we conclude the proof of~\eqref{app:prop_prod_with_identity_2}.
\end{proof}

\subsection{Analytic functions with decay}\label{AnalyticFunctions}

This section is divided into two parts. First, we provide a Cauchy formula for homogeneous analytic functions defined on special open subset of suitable complexifications of Banach spaces. In the second part, we prove the same results for homogeneous analytic functions with decay. 

Let $E_1$ and $E_2$ be two Banach spaces. For the sake of simplicity, we use the same notation $|\cdot|$ for the norms on the two spaces. We need to recall some notation. We consider positive parameters $\gamma >0$, $\varrho>0$, and recall that $B_\varrho$ stands for the open ball centered at the origin of radius $\varrho>0$. Given $V \subset E_1$ such that $0 \in \partial V$, we denote $V_\varrho = V \cap B_\varrho$.

Given $\gamma >0$ and an open set $V \subset E_1$ star-shaped with respect to $0$, we recall that 
\begin{equation*}
\begin{aligned}
\Omega(\gamma) &= \left\{x \in \C E_1 : \mathrm{Re}\,x \in V, |\mathrm{Im}\,x|<\gamma|\mathrm{Re}\,x|\right\}\\
\Omega(\varrho, \gamma) &= \left\{x \in \C E_1 : \mathrm{Re}\,x \in V_\varrho, |\mathrm{Im}\,x|<\gamma|\mathrm{Re}\,x|\right\}.
\end{aligned}
\end{equation*}
where $\C E_1$ stands for the complexification of $E_1$ (see Appendix \ref{App: Complexification}), and we refer to Definition \ref{def:star-shaped} for the definition of star-shaped set. 


We recall that $\mathcal{H}^\ell$ is the space of functions defined by~\eqref{Hspaces}, and we have the following version of Cauchy's estimate
\begin{proposition}\label{cor:cauchyBS}
    Given positive integers $k \ge 0$ and $\ell \ge 1$, we consider the following function
    \begin{equation*}
        f : \Omega(\varrho, \gamma) \subset \C E_1 \to \C E_2.
    \end{equation*}
    We assume that $f \in \mathcal{H}^\ell$, and $f$ is weakly analytic and continuous on $\Omega(\varrho, \gamma)$. Then, for any $0 < \gamma' < \gamma$, and $v_i \in \Omega(\gamma)$, with $|v_i| \le 1$, for each $1 \le i \le k$,
    \begin{equation*}
        |D^k f(x)(v_1,\dots,v_k)| \le C(\gamma, \varrho, k , \ell){{\displaystyle\sup_{a \in \Omega(\varrho, \gamma)}}|f\left(a\right)| \over \left(\gamma - \gamma'\right)^k} |x|^{\ell - k}
    \end{equation*}
    for all $x \in \Omega(\varrho, \gamma')$ and for a suitable positive constant $C(\gamma, \varrho, k , \ell)$ depending on $\gamma$, $k$, $\ell$, and $\varrho$. 
\end{proposition}
\begin{proof}
    We denote by $r$ the distance between $\partial \Omega(\gamma)$ and $\partial \Omega(\gamma')$. Let $\alpha$ be the angle between the lines $y = \gamma'\mathrm{Re} \, x$ and $y = \gamma\mathrm{Re}\,  x$ a straightforward computation shows that, for all $x \in \partial \Omega(\gamma')$
    \begin{equation}\label{App:r}
        r = |x|\sin \alpha  \mbox{\hspace{2mm} where \hspace{2mm} $\sin \alpha = {\gamma - \gamma' \over 1 + \gamma\gamma'} + \mathcal{O}_3(\gamma-\gamma')$}.
    \end{equation}
     For all $v_i \in \Omega(\gamma)$ with $|v_i| \le 1$ for each $1 \le i \le k$, thanks to Proposition \ref{prop:CauchyFormula}, one has that 
    \begin{align*}
        &f\left(x + \sum_{i=1}^k z_i v_i\right) =\left({1 \over 2 \pi i}\right)^k \int_{|\zeta_1|=r} \cdots \int_{|\zeta_k|=r} {f\left(x + \sum_{i=1}^k \zeta_i v_i\right)\over \prod_{i=1}^k\left(\zeta_i - z_i\right)}d\zeta_k \cdot \cdot \cdot d\zeta_1.
    \end{align*}
    We point out that $f$ is well-defined at the points $x + \sum_{i=1}^k z_i v_i$ and $x + \sum_{i=1}^k \zeta_i v_i$, since $f$ can be extended by homogeneity to the whole domain $\Omega(\gamma)$.

    By differentiating with respect to each variable $z_i$ and evaluating the $k$-th order derivative at $z_i = 0$ for all $i = 1, \dots, k$, we obtain
    \begin{align*}
        &D^k_xf\left(x\right)(v_1, \dots, v_k)=\left({1 \over 2 \pi i}\right)^k \int_{|\zeta_1|=r} \cdots \int_{|\zeta_k|=r} {f\left(x + \sum_{i=1}^k \zeta_i v_i\right)\over \prod_{i=1}^k\zeta^2_i}d\zeta_k \cdot \cdot \cdot d\zeta_1.
    \end{align*}
     For all $x \in \Omega(\varrho, \gamma')$, we can estimate the norm of the left-hand side of the latter equality as follows
    \begin{align}
       &|D^k_xf\left(x\right)(v_1, \dots, v_k)| \nonumber \\ \label{est1:cauchyBS}
       &\le C(\gamma) \frac{1}{(2\pi)^k} \int_{|\zeta_1|=r} \cdots \int_{|\zeta_k|=r} {\left|f\left({x + \sum_{i=1}^k \zeta_i v_i \over |x + \sum_{i=1}^k \zeta_i v_i|} {\varrho \over 2}\right)\right| \over r^k \left(\gamma -\gamma'\right)^k} {\left|x + \sum_{i=1}^k \zeta_i v_i\right|^{\ell }  \over |x|^k} \frac{2^\ell}{\varrho^\ell} d\zeta_k \cdot \cdot \cdot d\zeta_1  \\
       &\le C(\gamma, \varrho, k , \ell) {\displaystyle\sup_{a \in \Omega (\varrho, \gamma)} \left|f(a)\right| \over (\gamma -\gamma')^k} |x|^{\ell - k} \label{est2:cauchyBS}
    \end{align}
    Here, in line~\eqref{est1:cauchyBS}, we used $f \in \mathcal{H}^\ell$ and~\eqref{App:r}. The estimate~\eqref{est2:cauchyBS} follows from the fact that, by~\eqref{App:r}, we have that $\left|x + \sum_{i=1}^k \zeta_i v_i \right| \le |x| + k |\sin \alpha| |x| \le (1 + k) |x|$.
\end{proof}

In the second part of this section, we consider two families of Banach spaces $\mathcal{X} = \{\mathcal{X}_i\}_{i \in \Z^d}$,  and $\mathcal{Y} = \{\mathcal{Y}_i\}_{i \in \Z^d}$. We recall that the space $\ell^\infty(\mathcal{X})$ is introduced in Definition~\ref{EllInfty}.

We assume that $E_1 = \ell^{\infty}(\mathcal{X})$,  and $E_2 = \ell^{\infty}(\mathcal{Y})$. We want to prove the analog of Proposition \ref{prop:CauchyEstimate} and Proposition \ref{cor:cauchyBS} for analytic functions with decay. To this end, given $\mathcal{U} \subset \ell^\infty\left(\mathcal{X}\right)$ an open subset, we recall that
\begin{align*}
C^1_\Gamma \left(\mathcal{U}, \ell^\infty\left(\mathcal{Y}\right)\right) = \Big\{ F \in C^1 \left(\mathcal{U}, \ell^\infty \left(\mathcal{Y}\right)\right) : & DF(x) \in L_\Gamma \hspace{2mm} \mbox{for all}  \hspace{2mm}  x \in \mathcal{U},\\
&\sup_{x \in \mathcal{U}} \left|F(x)\right| < \infty, \hspace{2mm}  \sup_{x \in \mathcal{U}} \left|D_xF(x)\right|_\Gamma < \infty \Big\}
\end{align*}
with the norm 
\begin{equation*}
    |F|_{C^1_\Gamma} = \max\left\{\sup_{x \in \mathcal{U}} \left|F(x)\right|, \sup_{x \in \mathcal{U}}|DF(x)|_\Gamma\right\}
\end{equation*}
the space of functions $C^1_\Gamma \left(\mathcal{U}, \ell^\infty\left(\mathcal{Y}\right)\right)$ is a Banach space. We point out that the Banach space $L_\Gamma$ is defined by~\eqref{LGamma} and the associated norm by~\eqref{Def:GammaLinearNorm}. 

\begin{proposition}\label{prop:cauchyBS_Decay}
    Let $f : U \subset \C \ell^{\infty}(\mathcal{X}) \to \C \ell^{\infty}(\mathcal{Y})$ be a map from an open subset $U$ of $\C\ellX$ into $\C\ellY$. For all $x \in U$, let $\sigma>0$ in such a way that $B_\sigma(x) \subset U$.  
    We assume that $f \in C^1_\Gamma (U, \C \ell^{\infty}(\mathcal{Y}))$. Then, for any $0 < \gamma' < \gamma$, 
    \begin{equation}
    \label{DkxDmyEst}
        |D^{k+1} f(x)|_\Gamma \le {k^k \over \sigma^k} \sup_{a \in \bar B_\sigma(x)}|Df\left(a\right)|_\Gamma , 
    \end{equation}
    for all $x \in U$.
\end{proposition}
\begin{proof}
    First, we recall that 
    \begin{equation*}
        |D^{k+1} f(x)|_\Gamma = \max \left\{|D^{k+1} f(x)|, \gamma\left(D^{k+1} f(x)\right)\right\},
    \end{equation*}
    where $\gamma\left(D^{k+1} f(x)\right)$ is defined in~\eqref{Def:GammaLinearNorm}.
    We have to verify that both terms on the right-hand side of the latter satisfy the bound in~\eqref{DkxDmyEst}. The first term $|D_x^{k+1} D_y^m f(x, 0)|$, by Proposition~\ref{prop:CauchyEstimate}, satisfies~\eqref{DkxDmyEst}. Concerning the second term, let $u$, $v_l \in \ellX$ with $|v_l|\le 1$ for each $1 \le l \le k$. We fix $i, j \in \Z^d$ and we assume that $\pi_n u=0$ if $n \ne j$. Thanks to Proposition~\ref{prop:CauchyEstimate},
    \begin{equation*}
        |D^{k+1} f_i(x)(u, v_1, \dots, v_k)|\Gamma(i-j)^{-1}\le {k^k \over\sigma^k}\sup_{a \in  B_\sigma}|Df_i\left(a\right)u|\Gamma(i-j)^{-1},
    \end{equation*}
    for all $x \in U$. Remembering the definition of the norm $|\cdot|_\Gamma$ (see~\eqref{Def:GammaLinearNorm}), one can conclude the proof. 
\end{proof} 

\begin{proposition}\label{cor:cauchyBS_Decay}
    Let the  positive integers $k \ge 0$ and $\ell \ge 1$ be  fixed. Assume that
    \begin{equation*}
        f : \Omega(\varrho, \gamma) \subset \C \ell^{\infty}(\mathcal{X}) \to \C \ell^{\infty}(\mathcal{Y}).
    \end{equation*}
    belongs to $\mathcal{H}^\ell \cap  C^1_\Gamma (\Omega(\varrho, \gamma), \C \ell^{\infty}(\mathcal{Y}))$.
    Then, for any $0 < \gamma' < \gamma$, 
    \begin{equation*}
        |D^{k +1} f(x)|_\Gamma \le C(\gamma, k, \ell, \varrho) {{\displaystyle\sup_{a \in \Omega(\varrho, \gamma)}}|Df\left(a\right)|_\Gamma \over \left(\gamma - \gamma'\right)^k} |x|^{\ell - k-1},
    \end{equation*}
    for all $x \in \Omega(\varrho, \gamma')$ and for a suitable positive constant $C(\gamma, k, \ell, \varrho)$ depending on $\gamma$, $k$, $\ell$, and $\varrho$. 
\end{proposition}
\begin{proof}
    This result follows from Proposition \ref{cor:cauchyBS}. 
    The proof is omitted since it is similar to that of Proposition \ref{prop:cauchyBS_Decay}.
\end{proof}

\section{Failure of the decay property under complexification}\label{app:contr_ex_gamma_norm_ext_real_complex}
We show that boundedness with respect to the seminorm $\gamma(\cdot)$ defined in~\eqref{Def:GammaLinearNorm} on the real domain does not, in general, extend to arbitrarily small complex neighborhoods. For this purpose, we consider the decay function $\Gamma : \Z \to \R^+$ satisfying the properties of Definition \ref{def:decay_fun}. For each $z \in \C$, we consider the following infinite matrix
\begin{equation*}
    A(z) = \{a_{kj}(z)\}_{k,j \in \Z}  \quad \mbox{such that} \hspace{2mm} a_{kj}(z) = \Gamma (k-j)e^{i(k-j)z}.
\end{equation*}
We consider the following trivial family of Banach spaces $\mathcal{C} = \{C_k\}_{k\in\Z}$ with $C_k = \C$ for all $k \in \Z$ and associated norms $|\cdot|_k = |\cdot|$, where $|\cdot|$ stands for the modulus of complex numbers. 

Remembering the definition~\eqref{Def:GammaLinearNorm}, we notice that, for real $x \in \R$, for any $k,j \in \Z$ and for all $u \in \ell^\infty(\mathcal{C})$ such that $|u| \le 1$ and $\pi_l u = 0$ if $l \ne j$, we have that, 
\begin{equation*}
    |\left(A(x) u\right)_k|\Gamma(k-j)^{-1} = |a_{kj}(x)u_j| = |e^{i(k-j)x}u_j|\le 1. 
\end{equation*}
On the other hand, let $z= x+iy$ with $y \ne 0$, then
\begin{equation*}
    |\left(A(z) u\right)_k|\Gamma(k-j)^{-1} = |a_{kj}(z)u_j| = |e^{i(k-j)z}u_j| =  e^{(j-k)y}|u_j|.
\end{equation*}
We fix $j \in \Z$. We observe that, if $y >0$ (resp. $y<0$), then
\begin{equation*}
   \lim_{k \to -\infty} e^{(j-k)y}|u_j| = +\infty \quad \mbox{(resp. $\lim_{k \to +\infty} e^{(j-k)y}|u_j| = +\infty$)}.
\end{equation*}
This proves that, for any $x \in \R$ and $z=x+iy \in \C$ with $y \ne 0$
\begin{equation*}
    \gamma(A(x)) < \infty, \qquad \gamma(A(z)) = \infty. 
\end{equation*}

\bibliographystyle{amsalpha}
\bibliography{refLattice}
\end{document}